\documentclass{amsart}
\usepackage{amscd,amssymb,amsmath,amsthm,amsfonts,srcltx,enumitem,subcaption,url,graphicx}

\dedicatory{
In memory of Maryam Mirzakhani\\
who inspired us in so many ways.
}

\numberwithin{equation}{section}

\newcommand{\nofcyl}{m}  
\newcommand{\mult}{\nu}  
\newcommand{\degofz}{j}  
\newcommand{\cyl}{cyl}   
\newcommand{\prob}{\operatorname{P}}
\newcommand{\Aut}{\operatorname{Aut}}

\newcommand{\Card}{\operatorname{Card}}

\newcommand{\Vol}{\operatorname{Vol}}

\newcommand{\CP}{{\mathbb C}\!\operatorname{P}^1}
\newcommand{\Area}{\operatorname{Area}}

\renewcommand{\Re}{\operatorname{Re}}
\renewcommand{\Im}{\operatorname{Im}}
\renewcommand{\emptyset}{\varnothing}
\renewcommand{\epsilon}{\varepsilon}

\newcommand{\GLR}{\operatorname{GL}_2(\mathbb{R})}
\newcommand{\SLR}{\operatorname{SL}_2(\mathbb{R})}
\newcommand{\SLZ}{\operatorname{SL}_2(\mathbb{Z})}
\newcommand{\UhR}{\operatorname{U}_h(\mathbb{R})}
\newcommand{\UvR}{\operatorname{U}_v(\mathbb{R})}
\newcommand{\UhZ}{\operatorname{U}_h(\mathbb{Z})}
\newcommand{\UvZ}{\operatorname{U}_v(\mathbb{Z})}

\newcommand\C{\mathbb C}
\renewcommand\L{\mathbb L}
\newcommand\N{\mathbb N}
\newcommand\Q{\mathbb Q}
\newcommand\R{\mathbb R}
\newcommand\T{\mathbb T}
\newcommand\Z{\mathbb Z}

\newcommand{\cD}{\mathcal{D}}
\newcommand{\cH}{\mathcal{H}}
\newcommand{\cG}{\mathcal{G}}
\newcommand{\cL}{\mathcal{L}}
\newcommand{\cM}{\mathcal{M}}
\newcommand{\cN}{\mathcal{N}}
\newcommand{\cP}{\mathcal{P}}
\newcommand{\cQ}{\mathcal{Q}}
\newcommand{\cS}{\mathcal{S}}
\newcommand{\cT}{\mathcal{T}}

\newlength{\halfbls}

\newtheorem{Theorem}{Theorem}[section]
\newtheorem*{NNTheorem}{Theorem}

\newtheorem{Proposition}[Theorem]{Proposition}
\newtheorem{Lemma}[Theorem]{Lemma}
\newtheorem{Corollary}[Theorem]{Corollary}

\theoremstyle{definition}
\newtheorem{Definition}[Theorem]{Definition}
\newtheorem{Convention}[Theorem]{Convention}

\theoremstyle{remark}
\newtheorem{Example}[Theorem]{Example}
\newtheorem{Remark}[Theorem]{Remark}

\title{Enumeration of meanders and Masur--Veech volumes}

\author[V.~Delecroix]{Vincent Delecroix}
\address{
LaBRI,
Domaine universitaire,
351 cours de la Lib\'eration, 33405 Talence, FRANCE
}
\email{20100.delecroix@gmail.com}

\author[\'E.~Goujard]{\'Elise Goujard}
\thanks{Research of the second author was partially supported  by a public grant as part of the FMJH}
\address{
Institut de Math\'ematiques de Bordeaux,
Universit\'e de Bordeaux,
351, cours de la Lib\'eration, 33405 Talence, FRANCE
}
\email{elise.goujard@gmail.com}

\author[P.~G.~Zograf]{Peter~Zograf}
\thanks{Research of Section~\ref{s:Computations:for:square:tiled:surfaces} is supported by the RScF grant 16-11-10039.}
\address{
St.~Petersburg Department, Steklov Math. Institute, Fontanka 27,
St. Petersburg 191023, and Chebyshev Laboratory,
St. Petersburg State University, 14th
Line V.O. 29B, St.Petersburg 199178 Russia}
\email{zograf@pdmi.ras.ru}

\author[A.~Zorich]{Anton Zorich}
\thanks{Research of the fourth author was partially supported by IUF}
\address{
Center for Advanced Studies, Skoltech;
Institut de Math\'ematiques de Jussieu --
Paris Rive Gauche,
Case 7012,
8 Place Aur\'elie Nemours,
75205 PARIS Cedex 13, France}
\email{anton.zorich@gmail.com}


\begin{document}

\begin{abstract}
A \textit{meander} is a topological configuration of a line and a
simple closed curve in the plane (or a pair of simple closed curves
on the 2-sphere) intersecting transversally. Meanders can be traced
back to H.~Poincar\'e and naturally appear in various areas of
mathematics, theoretical physics and computational biology (in
particular, they provide a model of polymer folding). Enumeration
of meanders is an important open problem. The number of meanders with
$2N$ crossings grows exponentially when $N$ grows, but the
long-standing problem on the precise asymptotics is still
out of reach.

We show that the situation becomes more tractable if one
additionally fixes the topological type (or the total number of
minimal arcs) of a meander. Then we are able to derive simple
asymptotic formulas for the numbers of meanders as $N$ tends to
infinity. We also compute the asymptotic probability of getting a
simple closed curve on a sphere by identifying the endpoints of two
arc systems (one on each of the two hemispheres) along the common
equator.

The new tools we bring to bear are based on interpretation of
meanders as square-tiled surfaces with one horizontal and one vertical
cylinders. The proofs combine recent results on Masur--Veech
volumes of the moduli spaces of meromorphic quadratic differentials
in genus zero with our new observation that horizontal and vertical
separatrix diagrams of integer quadratic differentials are
asymptotically uncorrelated. The additional combinatorial
constraints we impose in this article yield explicit polynomial
asymptotics.

\end{abstract}

\maketitle
\tableofcontents

\section{Introduction}
\label{s:introduction:and:main:results}

In the seminal
paper~\cite{Mirzakhani:growth:of:simple:geodesics}
M.~Mirzakhani computed the asymptotics for the number of
simple closed geodesics on a hyperbolic surface of constant
negative curvature. In particular, she proved that
asymptotically, when the bound for the length of simple
closed geodesics tends to infinity, the probability of
getting a separating or non-separating geodesic becomes
independent of the hyperbolic metric.

We count the asymptotics for the number of \textit{pairs}
of transverse simple closed curves, or meanders, of a fixed combinatorial
type on a sphere when the number of intersections tends to
infinity.
Our starting observation is that a pair of transverse
simple closed curves which is seemingly a purely
combinatorial object defines a natural complex structure
and an ``integer'' meromorphic quadratic differential on
the original sphere.

M.~Mirzakhani established a relation between counting of
simple closed curves and Weil--Petersson volumes of the
moduli spaces of bordered hyperbolic surfaces, see Theorem
5.3 in~\cite{Mirzakhani:growth:of:simple:geodesics}.
Counting pairs of transverse simple closed curves leads
naturally to Masur--Veech volumes of the moduli spaces of
meromorphic quadratic differentials with at most simple
poles. In both situation, an essential ingredient is the
ergodicity of certain group action. In Mirzakhani's case,
it is the action of the mapping class group
$\operatorname{Mod}_{g,n}$ on the space of measured
laminations $\cM\cL_{g,n}$. In our setting, it will be the
$\GLR$-action on (strata of) the moduli space of quadratic
differentials $\cQ(\xi)$,where $\xi$ denotes the number of
simple poles, the number of zeroes, and the degrees of
zeroes of meromorphic quadratic differentials in the
stratum $\cQ(\xi)$. Both moduli spaces have integral
piecewise linear structures and both counting problems
(count of multicurves and count of pairs of transverse
multicurves) can be formulated in terms of count of integer
points in respectively $\cM\cL_{g,n}$ and $\cQ(\xi)$. A
simple example is provided by the set of primitive integer
points in $\Z^2\subset\R^2$. They have asymptotic density
$\delta=\frac{6}{\pi^2}$, meaning that in a ball of radius
$R$ centered around the origin, there are $\delta \pi R^2 +
o(R^2)$ such primitive points. Moreover, the density is
\textit{uniform} (see Section~\ref{ss:densities} for formal
definitions).

The other essential ingredient of our proof that has no equivalent
in Mirzakhani setting is a non-correlation result that we deduce
from the product structure on the strata of quadratic differentials.
More precisely, we prove that for suitable subsets $\cD_1$ and
$\cD_2$ of square-tiled surfaces, the density $\delta(\cD_1 \cap \cD_2)$
of their intersection is equal to the product of densities
$\delta(\cD_1)\cdot\delta(\cD_2)$ (see Theorem~\ref{th:uncorrelated}).
This non-correlation result is visible in our meander count:
the constant $\cyl_{1,1}(\cQ(1^s,-1^{s+4}))$ appearing in
Theorem~\ref{th:meander:counting} has an explicit product type expression
given in~\eqref{eq:c11:as:c1:squared:over:Vol}.

The equidistribution and non-correlation results are of
independent and much more general interest for enumerative
geometry of the moduli spaces of Abelian and quadratic
differentials. Application to the meander count is
just one out of many manifestations of these new results.

\subsection{Structure of the paper}
\label{ss:structure:of:the:paper}

In the first Section we state our results
on meander enumeration. The link with quadratic
differentials and Masur-Veech volumes is explained in
Section~\ref{s:strategy}. Section~\ref{s:From:arc:systems:and:meanders:to:square:tiled:surfaces} provides the proof of our results on meander
count; it uses general results from the subsequent
Section~\ref{s:Equidistribution}. This last
Section of the article proves the equidistribution and non-correlation
results in the moduli space of quadratic differentials.

The paper is organized in such way that Section~\ref{s:Equidistribution}
can be omitted by the readers interested only in meanders. On
the other hand, the readers interested only in square-tiled
surfaces and moduli spaces of quadratic differentials
can pass directly to Section~\ref{s:Equidistribution}.

Finally, Appendix~\ref{a:Lattices} describes the geometry
underlying two natural normalizations of the Masur--Veech
volume element on the moduli spaces of quadratic
differentials. This clarification is needed to apply the
results from Section~\ref{s:Equidistribution} to meander
count.

\subsection{Counting meanders with given number of minimal arcs}
\label{ss:Asymptotic:number:of:meanders:with:given:number:of:minimal:arcs}

A closed \textit{plane meander} is a smooth closed curve in
the plane transversally intersecting the horizontal line as
in Figure~\ref{fig:meander:types}. According to the
paper~\cite{Lando:Zvonkin} of S.~Lando and A.~Zvonkine
(serving as a reference paper in the literature on
meanders) the notion ``meander'' was suggested by
V.~I.~Arnold in~\cite{Arnold} though meanders were studied
already by H.~Poincar\'e~\cite{Poincare}. Meanders appear
in various contexts, in particular in physics,
see~\cite{DiFrancesco:Golinelli:Guitter}. The number of
meanders with $2N$ crossings is conjecturally asymptotic to
$const\cdot R^{2N}\cdot N^{\alpha}$, where $R^2\approx 12.26$
and $\alpha\approx-3.42$ are constants (we refer to~\cite{Jensen}
for the values of the constants). The conjectural exact
value $\alpha=-\frac{29+\sqrt{145}}{12}$ of the
corresponding critical exponent $\alpha$ in a
two-dimensional CFT with central charge $c=-4$ coupled to
gravity is given in~\cite{DiFrancesco:Golinelli:Guitter:2}.

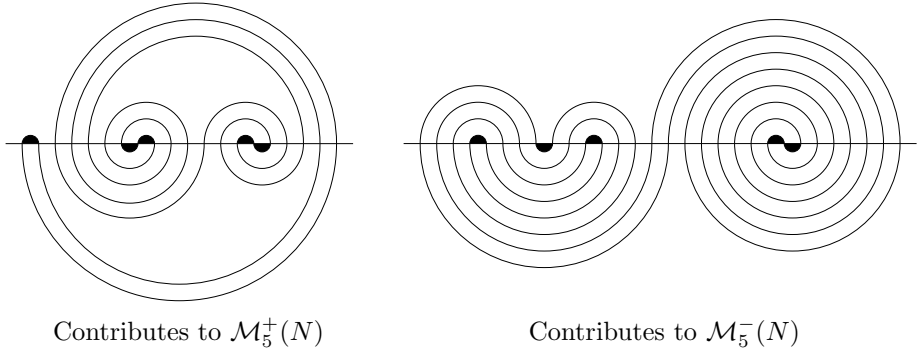
\begin{figure}[hbt]
\includegraphics{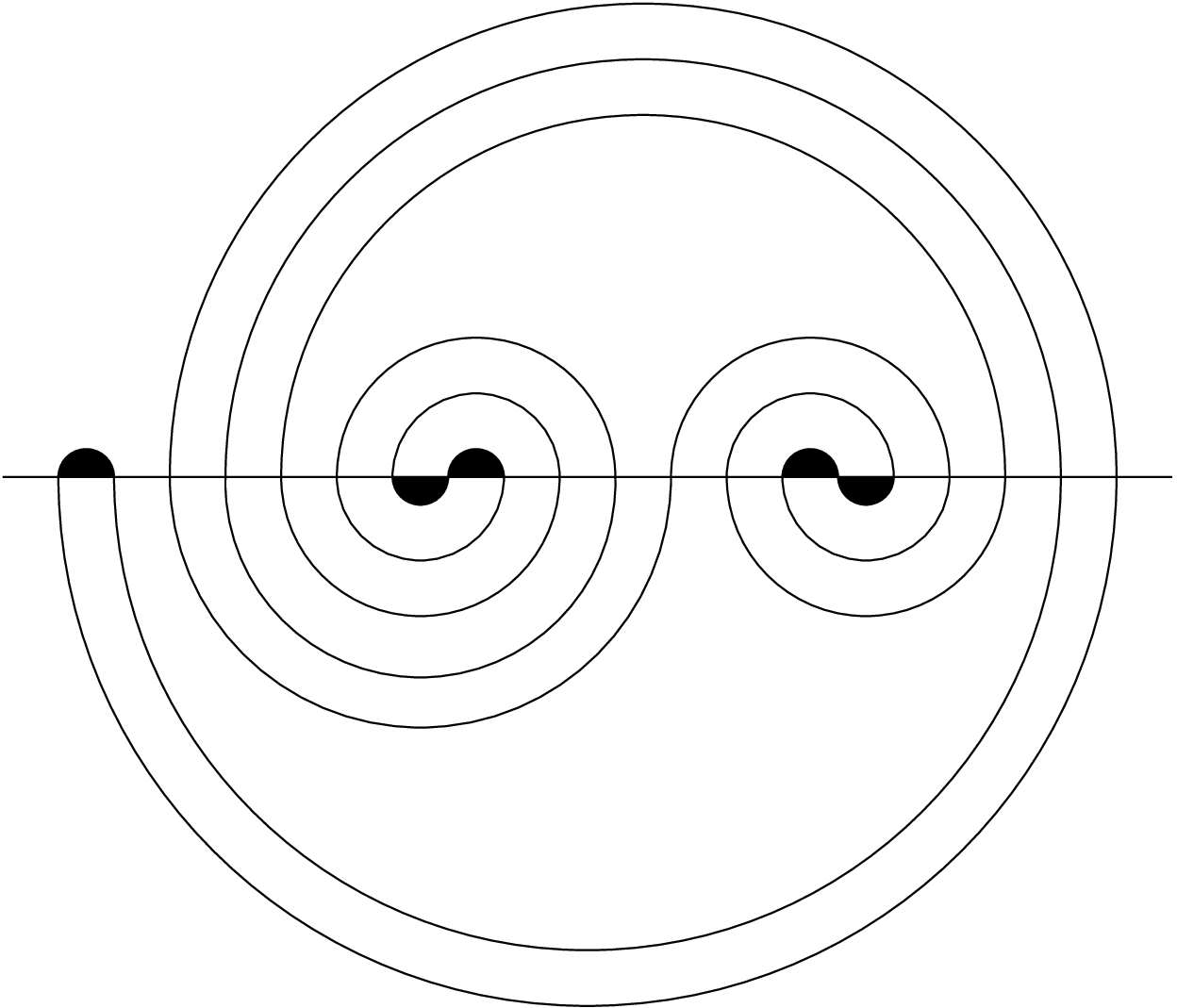}
\includegraphics{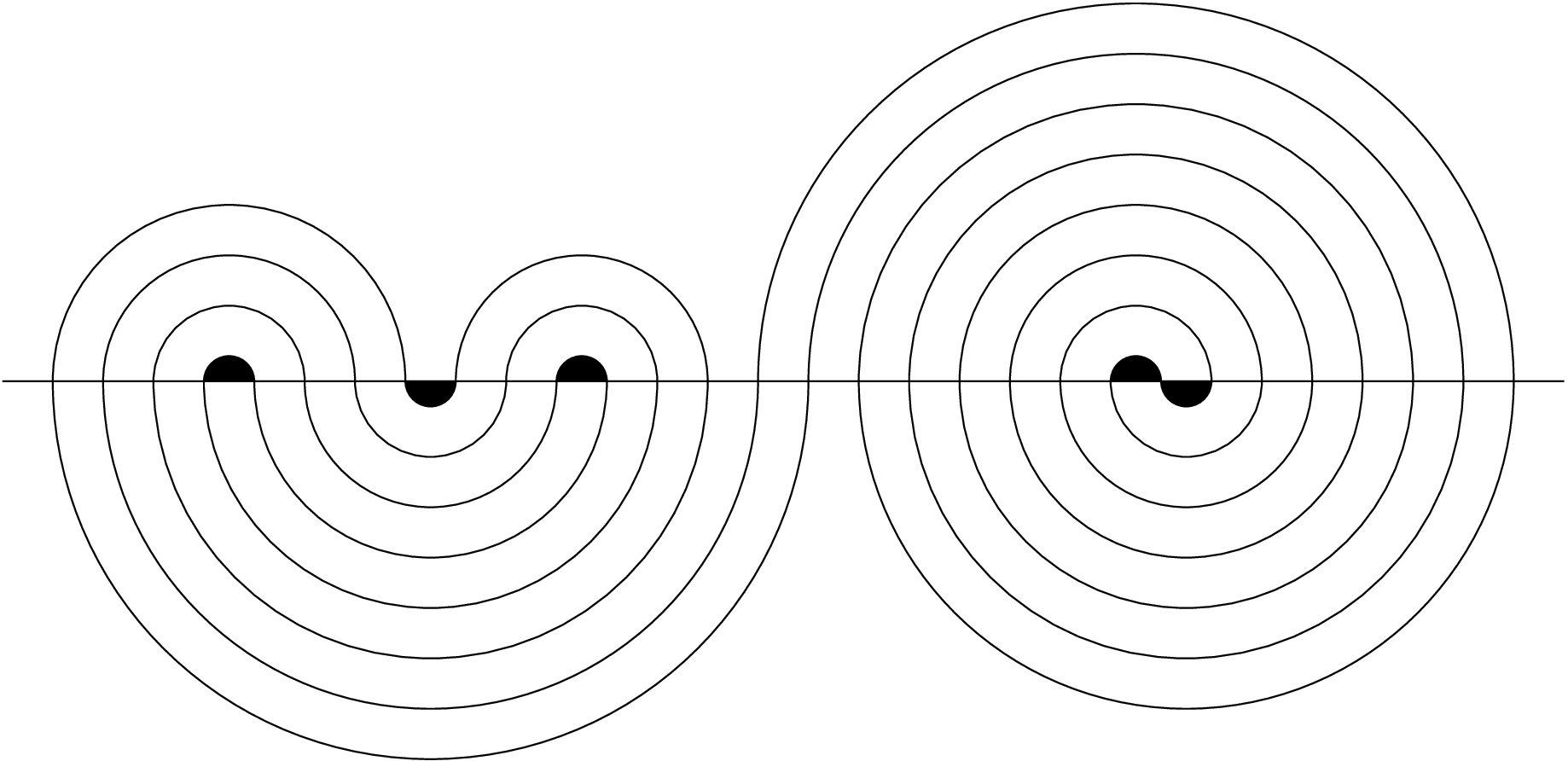}
\begin{picture}(260,20)(0,0) 
\put(-20,-105){Contributes to $\cM^+_5(N)$}
\put(160,-105){Contributes to $\cM^-_5(N)$}
\end{picture}
\vspace{110bp} 
\caption{
\label{fig:meander:types}
Meander with a maximal arc (``rainbow'')
on the left and without one on the right.
Both meanders have $5$ minimal arcs (``pimples'').
}
\end{figure}

We say that a meander has a \textit{maximal arc}
(``\textit{rainbow}'' in terminology
of~\cite{Andersen:Chekhov:Penner:Reidys:Sulkowski}) if it
has an arc joining the leftmost and the rightmost crossings
with the horizontal line. Otherwise the meander \textit{does
not a have maximal arc}. The meander on the left of
Figure~\ref{fig:meander:types} has a maximal arc, while the
one on the right does not.

By \textit{minimal arc} (``\textit{pimple}'' in terminology
of~\cite{Andersen:Chekhov:Penner:Reidys:Sulkowski}, or
``\textit{internal arch}'' in terminology
of~\cite{DiFrancesco:Golinelli:Guitter}) we call an arc which does
not have any crossings inside. The areas between the horizontal line
and the minimal arcs of meanders are colored in black in
Figure~\ref{fig:meander:types}; each of the two meanders has $p=5$
minimal arcs.

By convention, in this paper we do not consider the trivial
meander represented by a circle. All other
meanders satisfy $p\ge 3$ when they have a maximal arc and
$p\ge 4$ when they do not.

Let $\cM^+_p(N)$ and $\cM^-_p(N)$ be the numbers of
meanders respectively with and without maximal arc (``rainbow'') and
having at most $2N$ crossings with the horizontal line and exactly
$p$ minimal arcs (``pimples''). We consider $p$ as a parameter and we
study the leading terms of the asymptotics of $\cM^+_p(N)$ and
$\cM^-_p(N)$ as $N\to+\infty$.

\begin{Theorem}
\label{th:meander:counting}
For any fixed $p$ the numbers $\cM^+_p(N)$ and $\cM^-_p(N)$
of meanders with $p$ minimal arcs (pimples) and with
at most $2N$ crossings have the following asymptotics as
$N\to+\infty$:
\begin{align}
\label{eq:asymptotics:with}
\cM^+_p(N) &=2(p+1)\cdot
\frac{\cyl_{1,1}\big(\cQ(1^{p-3},-1^{p+1})\big)}
{(p+1)!\, (p-3)!}
\cdot \frac{N^{2p-4}}{4p-8}\ +\ o(N^{2p-4})=
\\ \notag
&=
\frac{2}{ p!\, (p-3)!}
\left(\frac{2}{\pi^2}\right)^{p-2}\cdot
\binom{2p-2}{p-1}^2\cdot \frac{N^{2p-4}}{4p-8}\ +\ o(N^{2p-4})\,.
\\
\notag
\\
\label{eq:asymptotics:without}
\cM^-_p(N) &=
\frac{2\,\cyl_{1,1}\big(\cQ(1^{p-4},0,-1^p)\big)}
{p!\,(p-4)!}
\cdot \frac{N^{2p-5}}{4p-10}\ +\ o(N^{2p-5})=
\\ \notag
&=
\frac{4}{ p!\, (p-4)!}
\left(\frac{2}{\pi^2}\right)^{p-3}\cdot
\binom{2p-4}{p-2}^2
\cdot \frac{N^{2p-5}}{4p-10}\ +\ o(N^{2p-5})\,.
\end{align}
\end{Theorem}

The quantities $\cyl_{1,1}\big(\cQ(1^{p-3},-1^{p+1})\big)$
and $\cyl_{1,1}\big(\cQ(1^{p-4},0,-1^p)\big)$ in the above
formulae are related to Masur--Veech volumes of the moduli
space of meromorphic quadratic differentials. Their
definition and role would be discussed in
Section~\ref{s:strategy}. Theorem~\ref{th:meander:counting}
is proved in Section~\ref{ss:Proofs:number:of:meanders}
with the exception of the explicit expressions for these
two quantities evaluated in Corollary~\ref{cor:principal}
in Section~\ref{s:Computations:for:square:tiled:surfaces}.

Note that the number $\cM^+_p(N)$ grows as $N^{2p-4}$ while
$\cM^-_p(N)$ grows as $N^{2p-5}$. This means that for large
$N$ all but a negligible fraction of meanders having any
given number $p$ of minimal arcs (pimples) do have a
maximal arc (rainbow) as the left one in
Figure~\ref{fig:meander:types}.

As the reader could observe in the statement of
Theorem~\ref{th:meander:counting}, our approach to counting
meanders differs from the traditional one: we fix the
combinatorics of the meander and then count the asymptotic
number of meanders of chosen combinatorial type as the
number of intersections $N$ tends to infinity. Our settings
can be seen as a zero temperature limit in the
thermodynamical sense, where the complexity of a meander is
measured in terms of the number of minimal arcs. Namely,
let us count meanders with the weight $e^{-\beta p}$, where
$\beta > 0$ is a parameter and $p$ is the number of minimal
arcs. Then $\beta=0$ corresponds to the standard count of
meanders. In the ``zero temperature limit''
$\beta\to+\infty$ meanders with few minimal arcs, as
considered in this paper, become predominant.

Applying Stirling's formula we get the following asymptotics for the
coefficients in formulae~\eqref{eq:asymptotics:with}
and~\eqref{eq:asymptotics:without} for large values of parameter $p$:
\begin{align*}
\frac{2}{ p!\, (p-3)!}
\left(\frac{2}{\pi^2}\right)^{p-2}\cdot
\binom{2p-2}{p-1}^2\cdot\frac{1}{4p-8}
&\ \sim\ \frac{\pi^2}{256}\cdot \left(\frac{32 e^2}{\pi^2 p^2}\right)^p
\hspace*{15pt} \text{ for }p\gg1\,.
\\
\frac{4}{ p!\, (p-4)!}
\left(\frac{2}{\pi^2}\right)^{p-3}\cdot
\binom{2p-4}{p-2}^2\cdot\frac{1}{4p-10}
&\ \sim\ \frac{\pi^2 e^2}{128 p}\cdot \left(\frac{32 e^2}{\pi^2 p^2}\right)^{p-1}
\text{ for }p\gg1\,.
\end{align*}
(we again recall that in our setting we always assume that $N\gg p$).

In Section~\ref{ss:Proofs:number:of:meanders} we provide
an analogous statement, Theorem~\ref{gth:meanders:fixed:stratum},
which counts meanders in the setting where the combinatorial type is
specified in a more detailed way.

\subsection{Counting meanders  with given reduced arc systems}
\label{ss:Meanders:and:arc:systems}

Extending the horizontal segment of a plane meander to the
infinite line and passing to a one-point compactification
of the plane we get a meander on the 2-sphere. A meander on
the sphere is a pair of transversally intersecting labeled
simple closed curves. It will be always clear from the
context whether we consider meanders in the plane or on the
sphere. Essentially, we adhere to the following dichotomy:
enumerating meanders, as in the previous Section, we work
with meanders in the plane, while considering frequencies
of pairs of simple closed curves among more complicated
pairs of multicurves, as in the current Section, we work
with meanders on the sphere.

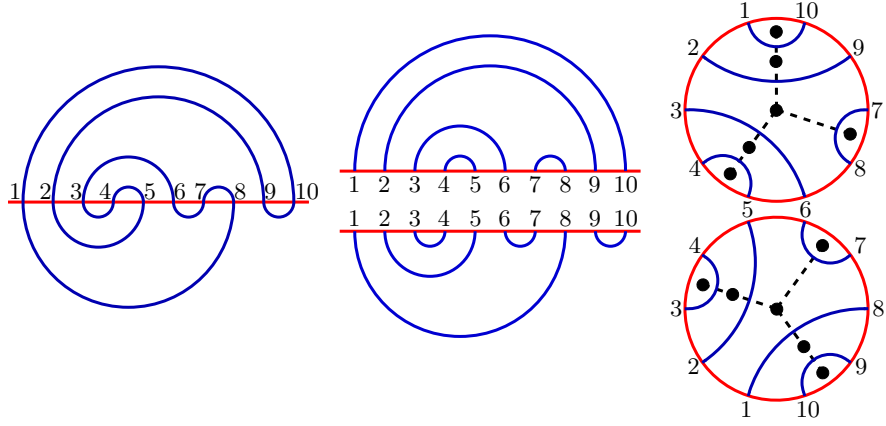
\begin{figure}[hbt]
\includegraphics{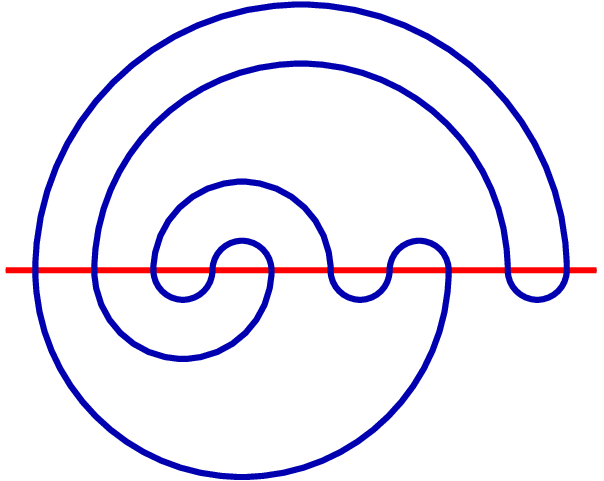}
\includegraphics{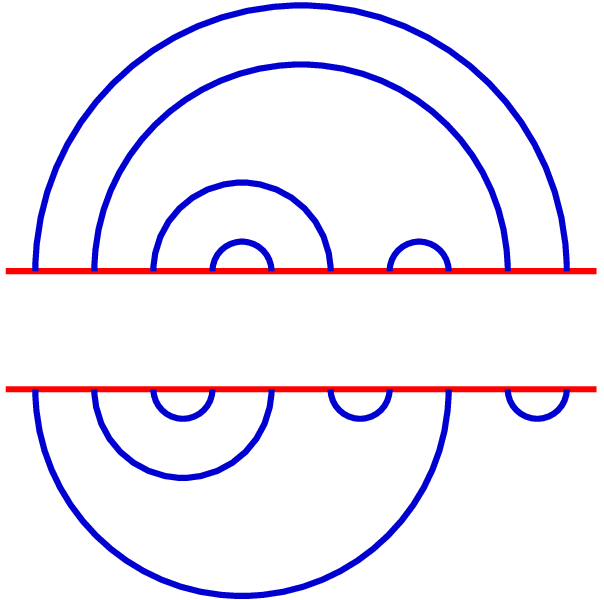}
\includegraphics{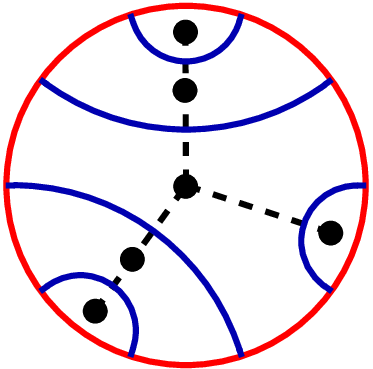}
\includegraphics{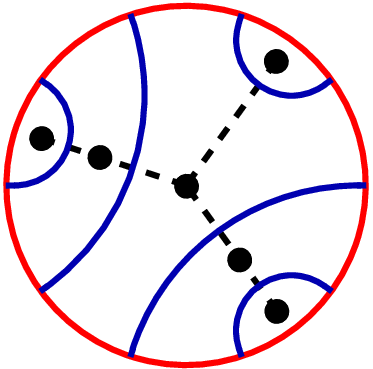}
\begin{picture}(0,0)(161.5,-99.5) 
\put(6.5,-168){\small 1}
\put(18,-168){\small 2}
\put(29.5,-168){\small 3}
\put(40.5,-168){\small 4}
\put(57.5,-168){\small 5}
\put(68.8,-168){\small 6}
\put(76.5,-168){\small 7}
\put(91.5,-168){\small 8}
\put(103,-168){\small 9}
\put(114.5,-168){\small 10}
\end{picture}
   %
\begin{picture}(0,0)(36.7,-102) 
\put(6.5,-168){\small 1}
\put(18,-168){\small 2}
\put(29.5,-168){\small 3}
\put(40.5,-168){\small 4}
\put(52.5,-168){\small 5}
\put(63.5,-168){\small 6}
\put(75,-168){\small 7}
\put(86.5,-168){\small 8}
\put(97.5,-168){\small 9}
\put(107,-168){\small 10}
\end{picture}
\begin{picture}(0,0)(40,-89) 
\put(6.5,-168){\small 1}
\put(18,-168){\small 2}
\put(29.5,-168){\small 3}
\put(40.5,-168){\small 4}
\put(52.5,-168){\small 5}
\put(63.5,-168){\small 6}
\put(75,-168){\small 7}
\put(86.5,-168){\small 8}
\put(97.5,-168){\small 9}
\put(107,-168){\small 10}
\end{picture}
   %
\begin{picture}(0,0)(-67,-135) 
\put(44,-135){\small 1}
\put(24.5,-149){\small 2}
\put(18,-173){\small 3}
\put(25,-195){\small 4}
\put(45,-210.5){\small 5}
\put(67,-210.5){\small 6}
\put(94,-173){\small 7}
\put(87.5,-195){\small 8}
\put(87,-149){\small 9}
\put(65,-135){\small 10}
\end{picture}
\begin{picture}(0,0)(-64.2,-60) 
\put(24.5,-149){\small 4}
\put(17.5,-173){\small 3}
\put(24,-195){\small 2}
\put(44,-211){\small 1}
\put(65,-211){\small 10}
\put(87.5,-195){\small 9}
\put(94,-173){\small 8}
\put(87,-149){\small 7}
\end{picture}
\vspace{150bp}
\caption{
\label{fig:meander}
A meander on the left. The associated pair of arc systems
in the middle. The same arc systems on the discs and the associated
dual  trees on the right.
}
\end{figure}

Each meander defines a pair of arc systems in discs as in
Figure~\ref{fig:meander}. Arcs in each of the two arc
systems do not intersect pairwise.
An arc system on the disc (also known as a
``chord diagram'') can be encoded by the dual tree, see the trees in
dashed lines on the right pictures in Figure~\ref{fig:meander}.
Namely, the vertices of the tree correspond to the faces in which the
arc system cuts the disc; two vertices are joined by an edge if and
only if the corresponding faces have a common arc. It is convenient to
simplify the dual tree by forgetting all vertices of valence two. We call the resulting tree the
\textit{reduced} dual tree.

It is much easier to count arc systems (for example, arc systems
sharing the same reduced dual tree). However, this does not simplify
meanders count since identifying a pair of arc systems with the
same number of arcs by the common equator, we sometimes get a meander
and sometimes a curve with several connected components, see
Figure~\ref{fig:pairs:of:arc:systems}.

\begin{figure}[htb]
\includegraphics{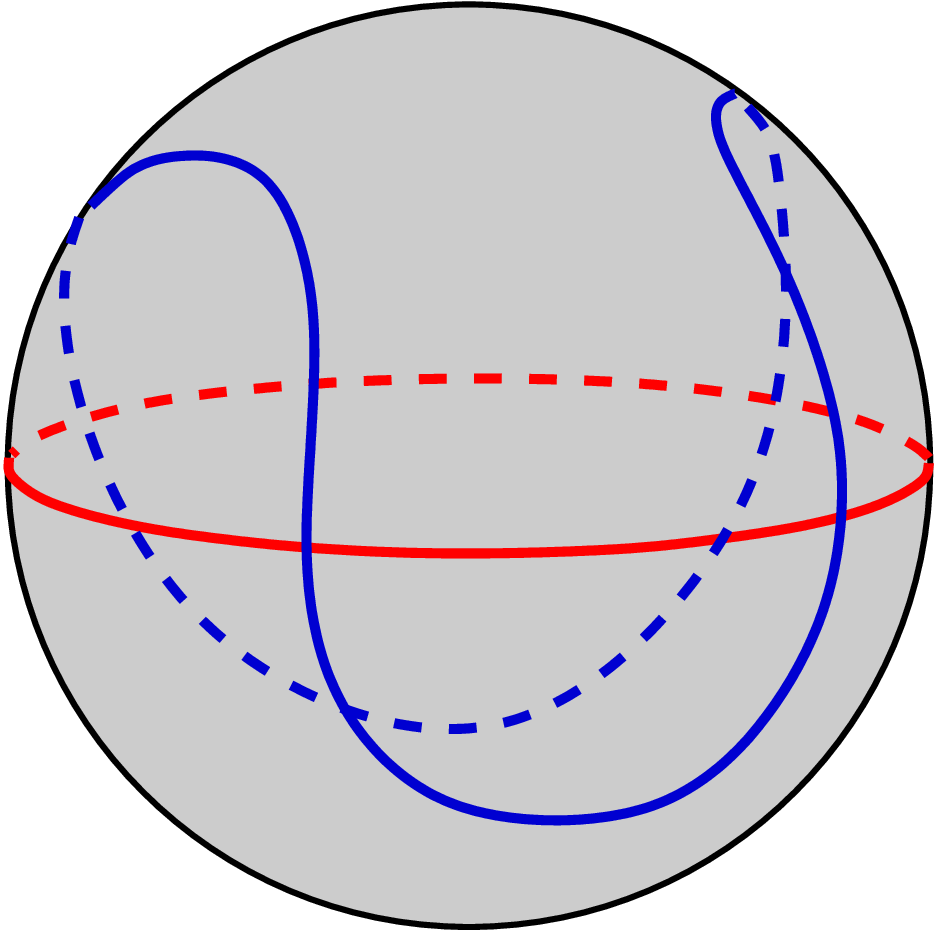}
\includegraphics{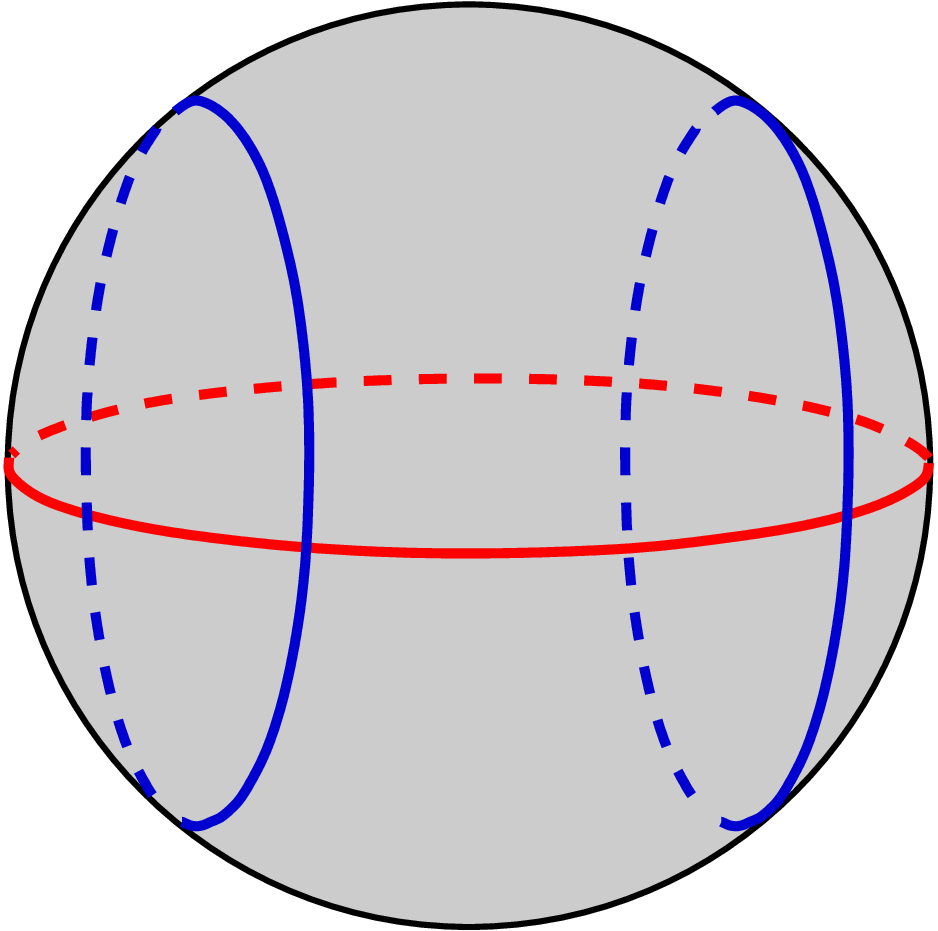}
\vspace{80bp}
\caption{
\label{fig:pairs:of:arc:systems}
Gluing two hemispheres with arc systems along the common
equator we may get either a single simple closed curve (as on the left
picture) or a multicurve with several connected components (as on the
right picture).
}
\end{figure}

We now consider the more specialized setting where we fix a pair of
plane trees and count meanders whose corresponding pair of arc systems
have these given dual trees. Let us mention that everywhere in this
paper we consider only \textit{plane trees}, that is, trees
embedded into the plane.

Let $(\cT_{top}, \cT_{bottom})$ be a pair of plane trees
with no vertices of valence 2. We consider the arc
system with the same number of arcs $n\le N$ on a labeled pair of oriented
discs having $\cT_{top}$ and $\cT_{bottom}$ as reduced dual trees. We
draw the arc system corresponding to $\cT_{top}$ on the northern
hemisphere, and the arc system corresponding to $\cT_{bottom}$ on
the southern hemisphere. There are $2n$ ways (up to isotopy) to identify
the boundaries of two hemispheres
into the sphere in such way that the endpoints of the arcs match.
We consider all possible triples
\begin{equation*}
(\text{$n$-arc system of type $\cT_{top}$;
 $n$-arc system of type $\cT_{bottom}$;
 identification})
\end{equation*}
as described above for all $n\le N$. Define
\begin{equation}
\label{eq:p:connected:iota:kappa}
\prob_{\mathit{connected}}(\cT_{top}, \cT_{bottom}; N):=
\frac{\text{
number of triples giving rise to meanders
}}
{\text{
total number of different triples
}}\,.
\end{equation}

\begin{Theorem}
\label{th:trivalent:trees:connected:proportion}
For any pair of \textit{trivalent} plane trees
$\cT_{bottom},\cT_{top}$,
having the total number $p$ of leaves
(vertices of valence one)
the following limit exists:
\begin{multline}
\label{eq:p1:fixed:number:of:poles}
\lim_{N\to+\infty} \prob_{\mathit{connected}}(\cT_{bottom},\cT_{top}; N)
\ =\
\prob_1(\cQ(1^{p-4},-1^p))
\ =\\=\
\frac{\cyl_1(\cQ(1^{p-4},-1^p))}
{\Vol_1(\cQ_1(1^{p-4},-1^p))}
\ =\
\frac{1}{2}\left(\frac{2}{\pi^2}\right)^{p-3}\cdot
\binom{2p-4}{p-2}\,.
\end{multline}
\end{Theorem}
The quantity $\cyl_{1}\big(\cQ(1^{p-4},-1^p)\big)$ in the
above formula is related to Masur--Veech volume of the
moduli space of meromorphic quadratic differentials. Its
definition and role will be discussed in
Section~\ref{s:strategy}.

The quantity $\prob_1(\cQ(1^{p-4},-1^p))$
can be seen as the asymptotic probability
that a random gluing of a pair of random arc systems
with $p$ minimal arcs produces a meander. To be more accurate,
one should rather speak of asymptotic \textit{density} of meanders
among the resulting multicurves.

Theorem~\ref{th:trivalent:trees:connected:proportion} is proved at the end of
Section~\ref{ss:Proofs:fractions}.We will actually
state and prove a more general statement,
Theorem~\ref{th:any:trees:connected:proportion}, where not only trivalent trees
are considered.

The fact that this asymptotic density is nonzero is
already somehow unexpected. For example, for the pair of
trees as on the right side of Figure~\ref{fig:meander} each
of the \textit{reduced} trees contains a single vertex of
valence three and three vertices of valence one (three
leaves), so we have six leaves in total. The corresponding
asymptotic density for such pair of reduced trees is
equal to
$$
\prob_{\mathit{connected}}(\hspace*{14pt},\hspace*{14pt})=\frac{280}{\pi^6}\approx 0.291245\,,
$$
\mbox{
\includegraphics{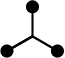}
\includegraphics{meandre3_up.eps}
}

\vspace*{-10pt}
\noindent
which is not even close to 0.

Stirling's formula gives the following asymptotics for
$\prob_1(\cQ(1^{p-4},-1^p))$ for large values of parameter $p$:
$$
\prob_1(\cQ(1^{p-4},-1^p))
=
\frac{1}{2}\left(\frac{2}{\pi^2}\right)^{p-3}\cdot
\binom{2p-4}{p-2}
\sim
\frac{2}{\sqrt{\pi p}}\cdot\left(\frac{8}{\pi^2}\right)^{p-3}
\ \text{for}\ p\gg 1
$$
(we recall that in our setting we always assume that $N\gg p$).

Another unexpected fact that follows from
Theorem~\ref{th:trivalent:trees:connected:proportion} is that the way the leaves
(univalent vertices) are distributed between the two trees is
irrelevant: the answer depends only on the total number $p$ of
leaves. This observation suggests an alternative (and much less
restrictive) way to fix combinatorics of the meanders. Namely, we can
fix only the total number $p$ of leaves (vertices of valence one) of
the two trees together, where $p\ge 4$.

\begin{Theorem}
\label{th:p:fixed:number:of:leaves:general}
Let $p \geq 4$.
The density $\prob_{\mathit{connected}}(p; N)$ of meanders
obtained by all possible identifications of all arc systems with at most $N$ arcs represented by
all possible pairs of (not necessarily trivalent) plane trees having the total number $p$ of leaves
(vertices of valence one) has the same limit $\prob_1(\cQ(1^{p-4},-1^p))$
as the density $\prob_{\mathit{connected}}(\cT_{bottom},\cT_{top}; N)$
of meanders represented by any individual pair of trivalent
plane trees with the total number $p$ of leaves:
\begin{multline}
\label{eq:p1:fixed:number:of:poles:general}
\lim_{N\to+\infty}
\prob_{\mathit{connected}}(p; N)
\ =\
\prob_1(\cQ(1^{p-4},-1^p))
\ =\\=\
\frac{\cyl_1(\cQ(1^{p-4},-1^p))}
{\Vol_1 \cQ_1(1^{p-4},-1^p)}
\ =\
\frac{1}{2}\left(\frac{2}{\pi^2}\right)^{p-3}\cdot
\binom{2p-4}{p-2}\,.
\end{multline}
The same statement with the same limit
is valid if we consider pairs of
plane trees having \textit{at most} $p$ leaves for the two trees
together instead of \textit{exactly} $p$ leaves.
\end{Theorem}

Theorem~\ref{th:p:fixed:number:of:leaves:general}
is proved at the end of
Section~\ref{s:From:arc:systems:and:meanders:to:square:tiled:surfaces}.
The proof is based on the fact that the contribution of any pair of
trees where at least one of the trees has a vertex of valence $4$ or
higher is negligible in comparison with the contribution of any pair
of trivalent trees.

\begin{Remark}
All the results of the Introduction concerning square-tiled
surfaces in strata of quadratic differentials of genus 0 generalize to
higher genera and even to the general situation of arithmetic invariant
suborbifolds (see Section~\ref{ss:background:flat:surf}). In this latter
setting, some subtle finiteness issues arise and will be treated elsewhere
to avoid overloading the current paper.
\end{Remark}
\medskip

\noindent\textbf{Acknowledgements.}
   %
The vague idea that counting results concerning linear
involutions might have applications to meanders was
discussed by the authors in independent conversations with
M.~Kontsevich and with M.~Mirzakhani. We are grateful to
them for these discussions and for their insights in
enumerative geometry that were very inspiring for us.

We thank MPIM in Bonn, where part of this work was
performed, for providing us with friendly and stimulating
environment. We thank J.~Athreya
for helpful suggestions that allowed to improve the presentation.
We are grateful to A.~Eskin for
the important advise on Moore's ergodicity theorem.
We thank E.~Duriev for reading carefully the manuscript and
for useful comments.


\section{Outline of counting technique}
\label{s:strategy}

\subsection{Pairs of transverse multicurves on the
sphere as square-tiled surfaces}
\label{ss:pairs:of:multicurves:as:square:tiled:surfaces}
A \textit{multicurve} on the sphere is a collection of pairwise
nonintersecting smooth simple closed curves.

\begin{Definition}
\label{def:pair:of:multicurves}
We say that two multicurves on the sphere form a
\textit{transverse connected pair} if any intersection
between any connected component of the first curve and any
connected component of the second curve is transverse and
if in addition the union of the two multicurves is connected.
\end{Definition}

Having a transverse connected pair of multicurves we always assume
that the pair is ordered. By convention, the first multicurve is
called ``horizontal'' and the second one ``vertical''. We
consider natural equivalence classes of
transverse connected pairs of multicurves
up to diffeomorphisms preserving the orientation of the
 sphere and respecting horizontal and vertical labelling.

Let $\cG$ be the graph defined by a transverse connected
pair of multicurves. The vertices of $\cG$ are
intersections of the multicurves, so all vertices of $\cG$
have valence $4$. Hence, all faces of the dual graph
$\cG^\ast$ are $4$-gons. The edges of $\cG^\ast$ dual to
horizontal edges of $\cG$ will be called vertical, and
those dual to the vertical edges of $\cG$ will be called
horizontal. By construction, any two non-adjacent edges of
any face of $\cG^\ast$ are either both horizontal or both
vertical.

Choosing identical metric squares as
faces of $\cG^\ast$ we get a
\textit{square-tiled surface}.

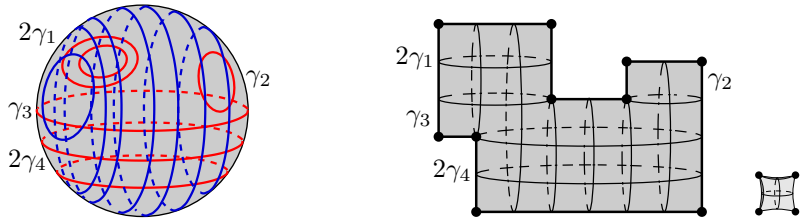
\begin{figure}[htb]
   %
   %
\includegraphics{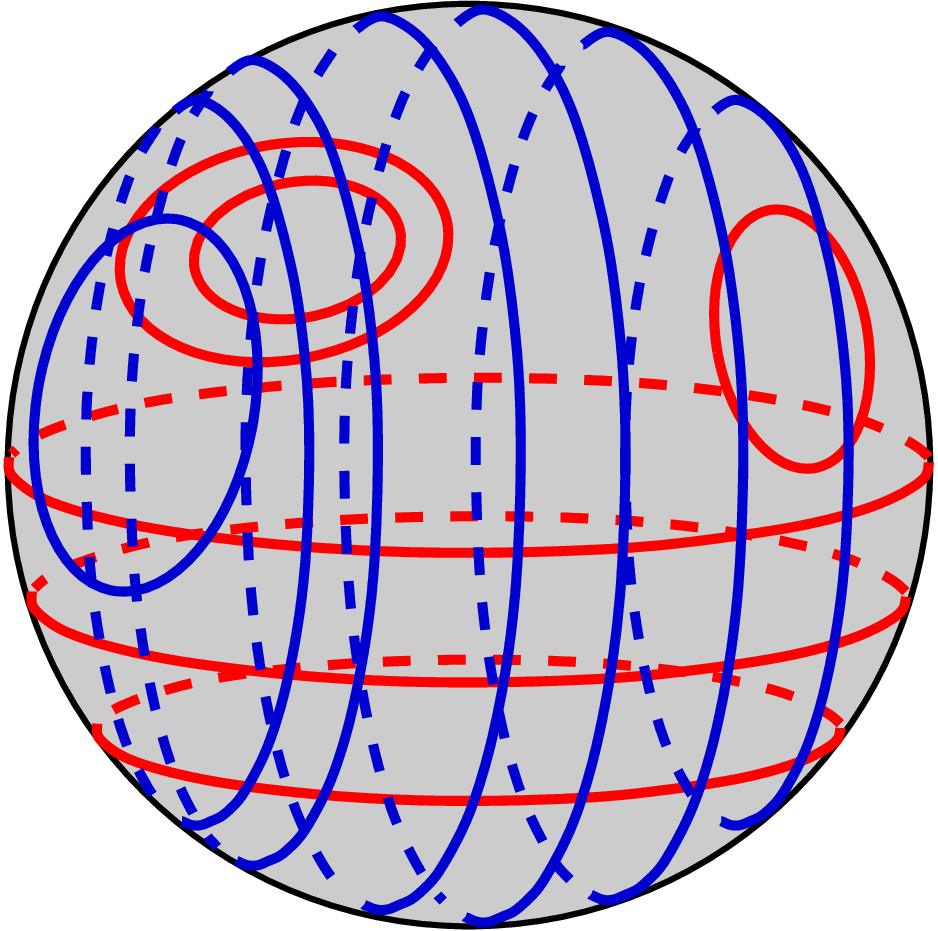}
\includegraphics{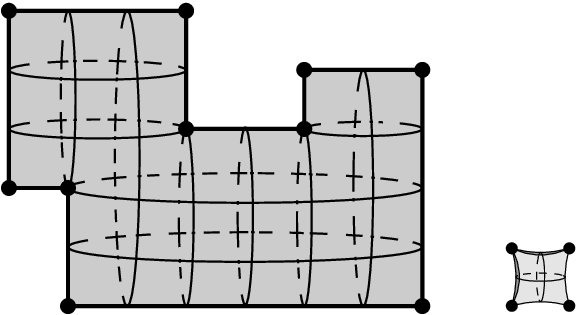}
\begin{picture}(0,0)(0,0)
\put(-10,-17){$2\gamma_1$}
\put(108,-25){$\gamma_2$}
\put(-6,-39){$\gamma_3$}
\put(4,-60){$2\gamma_4$}
\end{picture}
\begin{picture}(0,0)(150,0)
\put(-6,-8){$2\gamma_1$}
\put(80,-24){$\gamma_2$}
\put(-10,-36){$\gamma_3$}
\put(-10,-56){$2\gamma_4$}
\end{picture}
\vspace{70pt}
\caption{
\label{fig:square:pillow}
The dual graph to a transverse connected pair of
multicurves on a sphere defines a square-tiled surface.
In this particular example the
horizontal multicurve has the form
$2\gamma_1+\gamma_2+\gamma_3+2\gamma_4$; its primitive
components $\gamma_1,\dots,\gamma_4$ represent the four maximal
horizontal cylinders; the weights represent the numbers of
horizontal bands of squares in the corresponding cylinders.}
\end{figure}

We have proved the following statement.
\begin{Proposition}
\label{prop:pairs:multicurves:square:tiled:surfaces}
There is a natural one-to-one correspondence between
transverse connected pairs of multicurves on the sphere
and square-tiled surfaces of genus 0, where the square tiling is
given by the dual graph of the graph formed by the union of two
multicurves.
\end{Proposition}

A square-tiled surface defines a meromorphic quadratic differential
$q$ having the form $(dz)^2$ in the natural coordinate on each square.
Simple poles of $q$ correspond to bigons of $\cG$; zeroes of order $\degofz\in\N$
correspond to  $(2\degofz+4)$-gons.

\begin{Remark}
\label{rm:multicurves}
Note that speaking of a ``multicurve'' one usually assumes
that the components of a multicurve are neither
contractible nor peripheral (i.e. not freely homotopic to a
boundary component). Given a transverse connected pair of
multicurves in the sense of
Definition~\ref{def:pair:of:multicurves}
make a single puncture at
every bigon, $6$-gon, $8$-gon, etc (equivalently, puncture
every conical singularity of the associated square-tiled
surface). By construction, each component of the horizontal
(respectively vertical) multicurve is neither contractible
nor peripheral on the resulting punctured surface, so we get
a multicurve in the usual sense.

Traditionally, one represents a multicurve as a weighted
sum $\gamma=h_1\gamma_1+\dots+h_m\gamma_m$ of the primitive
components $\gamma_1,\dots,\gamma_m$, which are already not
pairwise freely homotopic on the punctured surface, and
where the positive integer weight $h_i$ encodes the number
of components of the multicurve $\gamma$ freely homotopic
to the primitive component $\gamma_i$ for $i=1,\dots,m$. In
our case the primitive components $\gamma_i$ and the
weights $h_i$ have particularly transparent interpretation
in terms of geometry of the associated square-tiled
surface. Namely, the primitive components $\gamma_i$ of the
horizontal (respectively vertical) multicurve are in the
natural correspondence with the maximal horizontal
(respectively vertical) cylinders, while the associated
weights $h_i$ represent the numbers of horizontal
(respectively vertical) bands of squares in the
corresponding cylinders.

The \textit{number of components} of a multicurve
$\gamma=h_1\gamma_1+\dots+h_m\gamma_m$ is given by the sum
$h_1+\dots+h_m$ of the weights; the number $m$ represents
the \textit{number of primitive components}. The number of
primitive components of the horizontal multicurve
corresponds to the number of maximal horizontal cylinders
of the associated square-tiled surface, while the number of
components represents the total number of horizontal bands
of squares. The square-tiled surface in
Figure~\ref{fig:square:pillow} has $6$ horizontal bands of
squares organized into $4$ maximal flat cylinders filled
with parallel horizontal closed regular geodesics. These
$6$ horizontal bands of squares correspond to $6$
components of the associated horizontal multicurve
$2\gamma_1+\gamma_2+\gamma_3+2\gamma_4$. The same
square-tiled surface has $7$ vertical bands of squares
organized into $4$ maximal flat cylinders filled with
parallel vertical closed regular geodesics. The associated
vertical multicurve has $7$ components.
\end{Remark}

\subsection{Counting square-tiled surfaces}
\label{ss:General:strategy}

The moduli space of meromorphic quadratic
differentials on $\CP$ with exactly $p$ simple poles is naturally
stratified by the strata $\cQ(\nu, -1^{|\nu|+4})$ of quadratic
differentials with prescribed orders of zeroes ($\nu_i$ zeroes
of order $i$) and with $p=|\nu|+4$ simple poles (see e.g.
\cite{Zorich:flat:surfaces} for references). Here
\begin{equation}
\label{eq:def:abs:nu}
|\nu|=1\cdot \nu_1 + 2\cdot \nu_2 + 3\cdot \nu_3 + \dots\,.
\end{equation}

Under the above interpretation,
transverse connected pairs of multicurves having fixed number of bigonal faces correspond to
square-tiled surfaces with fixed number of simple poles.
The transverse connected pairs of multicurves having fixed number $\nu_1$ of
hexagonal faces, fixed number $\nu_2$ of octagonal faces, fixed
number $\nu_\degofz$ of $2(\degofz+2)$-gonal faces for $\degofz\in\N$ correspond to
square-tiled surfaces in the fixed stratum $\cQ(\nu,-1^{|\nu|+4})$.
In particular, the number of bigonal faces equals $|\nu|+4$.
The number of squares in the square-tiled surface is the total number
of crossings between the two multicurves.

Gluing two hemispheres with arc systems along the common
equator as in Section~\ref{ss:Proofs:number:of:meanders},
we get a transverse connected pair of multicurves. The
horizontal multicurve has a single connected component, it
is just a simple closed curve represented by the equator,
whereas the vertical multicurve may have several connected
components. Such transverse connected pairs of multicurves
correspond to square-tiled surfaces having a single
horizontal band of squares.

Labeled connected pairs of transverse simple closed curves
correspond to square-tiled surfaces having a single
horizontal band of squares and a single vertical band of
squares. Closed meanders in the plane correspond to
square-tiled surfaces as above with a marked vertical side
of one of the squares of the tiling.

Having translated our counting problems into the language of square-tiled
surfaces, we are ready to present our approach in detail.
\medskip

\noindent\textbf{Square-tiled surfaces of fixed combinatorial type and Masur--Veech volumes}
Any (generalized) partition $\nu=[0^{\nu_0} 1^{\nu_1} 2^{\nu_2}\dots]$
determines a complex orbifold $\cQ(\nu, -1^{|\nu|+4})$ called a
\textit{stratum of meromorphic quadratic differentials}. The points of $\cQ(\nu, -1^{|\nu|+4})$ are the equivalence classes
of meromorphic quadratic differentials with only simple poles on the sphere having $\nu_0$ marked points, $\nu_1$ zeroes of order $1$, $\nu_2$ zeroes of order $2$, etc. Each stratum $\cQ(\nu, -1^{|\nu|+4})$ has an integral
linear structure and hence a well-defined measure that is called the
Masur--Veech measure. We denote by $\Vol_1 \cQ_1(\nu,-1^{|\nu|+4})$ the
Masur--Veech volume of the ``unit hyperboloid'' in the stratum $\cQ(\nu,-1^{|\nu|+4})$.
We refer the reader to Section~\ref{ss:background:flat:surf} for precise
definitions. By~\cite{AEZ:genus:0}, the following formula holds:
\begin{equation}
\label{eq:volume}
\Vol_1 \cQ_1(\nu,-1^{|\nu|+4})=2\pi^2\cdot
(f(0)\big)^{\nu_0}(f(1)\big)^{\nu_1}(f(2)\big)^{\nu_2}\cdots\,,
\end{equation}
where $|\nu|=1\cdot\nu_1+2\cdot\nu_2+\dots$ and
$$
f(\degofz)=
\frac{\degofz\,!!}{(\degofz+1)!!}\cdot\pi^{\degofz}\cdot
\begin{cases}
\pi&\text{if $\degofz$ is odd}\\
2  &\text{if $\degofz$ is even}
\end{cases}
$$
(here we use the notation
$$
\degofz\,!!:=
\begin{cases}
1\cdot 3\cdot 5 \dots \cdot \degofz, &\text{when $\degofz$ is odd,}
\\
2\cdot 4\cdot 6 \dots \cdot \degofz, &\text{when $\degofz$ is even}\,,
\end{cases}
$$
and the common convention $0!!:=1$).
This formula was originally
conjectured by M.~Kontsevich and recently proved in~\cite{AEZ:genus:0}.
In this setting zeroes and poles of quadratic differentials
are \textit{labeled}, so the Masur--Veech volume of the stratum with unlabeled zeroes and poles is smaller by
the factor $(|\nu|+4)!\ \nu_0!\ \nu_1!\ \cdots$.
It follows from the definition of the Masur--Veech volume that
the number $\cS^{\mathit{labeled}}_\nu(N)$
of square-tiled surfaces (i.e. the integral points
in $\cQ(\nu, -1^{|\nu|+4})$, see
Section~\ref{ss:quad:as:invariant:orbifold} for
the definition) in the stratum $\cQ(\nu,-1^{|\nu|+4})$
tiled with at most $2N$ squares has asymptotics
\begin{equation}
\label{eq:VolQ:N:d}
\cS^{\mathit{labeled}}_\nu(N)=
\Vol_1 \cQ_1(\nu,-1^{|\nu|+4}) \cdot \frac{N^d}{2d} + o(N^d)
\ \text{ as }\ N\to+\infty\,,
\end{equation}
where
\begin{equation}
\label{eq:dim}
d=\dim_{\mathbb{C}}\cQ(\nu,-1^{|\nu|+4})=\ell(\nu)+|\nu|+2
\end{equation}
and
\begin{equation}
\label{eq:def:ell:nu}
\ell(\nu):=\nu_0+\nu_1+\dots\,.
\end{equation}

In Theorems~\ref{th:meander:counting}
and~\ref{th:p:fixed:number:of:leaves:general}
the
number $p$ of bigons
serves as ``combinatorial type'' of a square-tiled surface.
In this setting formulae~\eqref{eq:VolQ:N:d}
and~\eqref{eq:dim} imply that all but negligible part of
transverse connected pairs of multicurves having large
number $N$ of intersections would have only
bigons, squares, and hexagons as faces and would correspond to
square-tiled surfaces in the principal stratum $\cQ(1^{p-4},-1^p)$.

As an alternative choice of  ``combinatorial type''  of a square-tiled surface one can specify the number of hexagons, octagons, etc,
separately, thus fixing the stratum $\cQ(\nu,-1^{|\nu|+4})$. This
corresponds to the setting of
Theorems~\ref{th:any:trees:connected:proportion}
and~\ref{gth:meanders:fixed:stratum} below.
Under either choice we have a simple asymptotic formula for the number of
transverse connected pairs of multicurves of fixed combinatorial type
with at most $2N$ intersections.

\begin{Remark}[Labeled versus non-labeled zeroes and poles]
\label{rm:labeled:zeroes:and:poles}
When we introduced square-tiled surfaces in
Section~\ref{ss:pairs:of:multicurves:as:square:tiled:surfaces}
and identified them with transverse connected pairs of
multicurves in
Proposition~\ref{prop:pairs:multicurves:square:tiled:surfaces},
we did not label zeroes and poles of the corresponding
quadratic differential, which was quite natural in this
setting. Traditionally, one labels zeroes and poles of a
square-tiled surface in the context of Masur--Veech
volumes, so we usually do label zeroes and poles. We recall
the setting every time when there may be any ambiguity.
The unlabeled and labeled count of square-tiled surfaces in
the stratum $\cQ(\nu, -1^{|\nu|+4})$ simply differ by the
integral factor $(|\nu|+4)!\, \nu_0!\, \nu_1!\, \cdots$.
\end{Remark}
\medskip

\noindent\textbf{Square-tiled surfaces with fixed combinatorics.}
By
Proposition~\ref{prop:pairs:multicurves:square:tiled:surfaces}
a square-tiled surface has exactly $k$ horizontal
(respectively vertical) bands of squares if and only if the
associated horizontal (respectively vertical) multicurve
has exactly $k$ connected components, see
Remark~\ref{rm:multicurves} at the end of
Section~\ref{ss:pairs:of:multicurves:as:square:tiled:surfaces}.
Count of square-tiled surfaces with exactly $k$ horizontal
bands of squares admits efficient combinatorial approach.
In the case of a single horizontal band of squares
corresponding to $k=1$, this count was performed
in~\cite{DGZZ-Yoccoz}. In
Section~\ref{s:Computations:for:square:tiled:surfaces} we
reproduce the relevant computations that become
particularly explicit in the case of the sphere.

\begin{Theorem}
\label{th:c1:in:genus:0}
Let $k \geq 1$ and let $\nu=[0^{\nu_0}
1^{\nu_1} 2^{\nu_2}\dots]$ be a generalized partition. The
number $\cS^{\mathit{labeled}}_{k,\nu}(N)$
of square-tiled surfaces in the stratum
$\cQ(\nu,-1^{|\nu|+4})$ with labeled zeros and poles tiled
with at most $2N$ squares organized into $k$ horizontal
bands (i.e. having the associated horizontal multicurve composed of $k$ connected components)
has asymptotics
\begin{equation}
\label{eq:c1:N:d}
\cS^{\mathit{labeled}}_{k,\nu}(N)=
\cyl_k\left(\cQ(\nu,-1^{|\nu|+4})\right) \cdot \frac{N^d}{2d} + O(N^{d-1})
\ \text{ as }\ N\to+\infty\,,
\end{equation}
where the coefficient $\cyl_k(\cQ(\nu,-1^{|\nu|+4}))$ is a positive rational number.
Moreover, in the case $k = 1$, one has the following explicit expression
\begin{equation}
\label{eq:c1:answer}
\cyl_1\left(\cQ(\nu,-1^{|\nu|+4})\right)=
2\cdot
\sum_{\iota_0=0}^{\nu_0}
\sum_{\iota_1=0}^{\nu_1}
\sum_{\iota_2=0}^{\nu_2}
\sum_{\dots}^{\dots}
\binom{\nu_0}{\iota_0}
\binom{\nu_1}{\iota_1}
\binom{\nu_2}{\iota_2}
\cdots
\binom{|\nu|+4}{|\iota|+2}\,.
\end{equation}
Here $\iota=[0^{\iota_0} 1^{\iota_1} 2^{\iota_2} \dots]$,
$ d=\dim_{\mathbb{C}}\cQ(\nu,-1^{|\nu|+4})=
\ell(\nu)+|\nu|+2$ and $|\nu|$ and $\ell(\nu)$ are
defined in equations~\eqref{eq:def:abs:nu}
and~\eqref{eq:def:ell:nu} respectively.
\end{Theorem}
Theorem~\ref{th:c1:in:genus:0} is proved
is proved at the end of
section~\ref{s:sep:diag:densities}
with exception for the explicit value~\eqref{eq:c1:answer}
for $\cyl_1\left(\cQ(\nu,-1^{|\nu|+4})\right)$
which is proved in
Section~\ref{s:Computations:for:square:tiled:surfaces}. For
the case of the principal stratum corresponding to $\nu =
[1^k]$ the formula becomes even more explicit,
see~\eqref{eq:c1:principal:answer}.

By symmetry argument, we get the same
asymptotics~\eqref{eq:c1:answer} with the same constant
$\cyl_k\left(\cQ(\nu,-1^{|\nu|+4})\right)$ for the number
of square-tiled surfaces with $k$ vertical (instead of
horizontal) bands of squares.

We have seen that meanders are in bijective correspondence
with square-tiled surfaces whose associated horizontal and
vertical multicurves have a single component. The following
result allows to transfer to meanders the explicit
asymptotics from Theorem~\ref{th:c1:in:genus:0}.

\begin{Theorem}
\label{th:c11:in:genus:0}
The number
$\cS^{\mathit{labeled}}_{k_{h},k_{v},\nu}(N)$ of
square-tiled surfaces in the stratum
$\cQ(\nu,-1^{|\nu|+4})$ with labeled zeroes and poles tiled
with at most $2N$ squares composed of $k_{h}$
horizontal and $k_{v}$ vertical bands of squares (i.e.
having the associated horizontal multicurve consisting of
$k_{h}$ components and the associated vertical
multicurve consisting of $k_{v}$ components) has the
following asymptotics as $N\to+\infty$:
\begin{equation}
\label{eq:c11:Q:nu}
\cS^{\mathit{labeled}}_{k_{h},k_{v},\nu}(N)=
\cyl_{k_{h},k_{v}}\left(\cQ(\nu,-1^{|\nu|+4})\right)\cdot\frac{N^d}{2d} +
o\left(N^{d}\right)\text{ as } N\to+\infty\,,
\end{equation}
where the constant $\cyl_{k_{h},k_{v}}\left(\cQ(\nu,-1^{|\nu|+4})\right)$
satisfies the following relation:
\begin{equation}
\label{eq:c11:as:c1:squared:over:Vol}
\frac{\cyl_{k_{h},k_{v}}\left(\cQ(\nu,-1^{|\nu|+4})\right)}
{\Vol_1 \cQ_1(\nu,-1^{|\nu|+4})}
=
\frac{\cyl_{k_{h}}\left(\cQ(\nu,-1^{|\nu|+4})\right)}
{\Vol_1 \cQ_1(\nu,-1^{|\nu|+4})}
\cdot
\frac{\cyl_{k_{v}}\left(\cQ(\nu,-1^{|\nu|+4})\right)}
{\Vol_1 \cQ_1(\nu,-1^{|\nu|+4})}\,,
\end{equation}
and the constants $\cyl_{k_{h}}\left(\cQ(\nu,-1^{|\nu|+4})\right), \cyl_{k_{v}}\left(\cQ(\nu,-1^{|\nu|+4})\right)$ are the ones
from Theorem~\ref{th:c1:in:genus:0}.
\end{Theorem}

Theorem~\ref{th:c11:in:genus:0} is proved at the end of
section~\ref{s:sep:diag:densities}.
The relation~\eqref{eq:c11:as:c1:squared:over:Vol} can be viewed as a
statement about independence of horizontal and vertical decompositions
of square-tiled surfaces.

Forgetting the labeling of zeroes and poles (see Remark~\ref{rm:labeled:zeroes:and:poles})
we get the asymptotics of the number of connected pairs of transverse simple closed curves
of fixed combinatorial type with at  most $2N$ crossings.

\medskip
\noindent\textbf{Further remarks.}
It is worth mentioning that all the above quantities have
combinatorial nature, but were computed by alternative
methods. The Masur--Veech volumes in genus zero $\Vol_1
\cQ_1(\nu, -1^{|\nu|+4})$ are closely related to Hurwitz
numbers counting covers of the sphere of some very special
ramification type. However, all attempts to compute these
volumes by purely combinatorial methods have (up to now)
failed even for covers of the simplest ramification type,
see e.~g.~\cite{AEZ:Dedicata}. The proof
in~\cite{AEZ:genus:0} of the formula for the Masur--Veech
volumes implicitly uses the analytic Riemann--Roch theorem
in addition to combinatorics.

Theorem~\ref{th:c11:in:genus:0} about square-tiled surfaces with a fixed
horizontal and vertical combinatorics is proved in Section~\ref{s:Equidistribution} using
ergodicity of the $\SLR$-action with respect to the Masur--Veech measure
and Moore's ergodicity theorem. The proof was inspired by the approach of
M.~Mirzakhani to counting simple closed geodesics on hyperbolic surfaces.

Note that the error term in Theorem~\ref{th:c1:in:genus:0}
has the form $O(N^{d-1})$ while the error term in
Theorem~\ref{th:c11:in:genus:0} has a weaker form $o(N^d)$.
The underlying reason is that the count in
Theorem~\ref{th:c1:in:genus:0} can be expressed in terms of
the Ehrahrt quasi-polynomial associated to the Euclidean
volume of certain rational polytope, see the proof of
Theorem~\ref{thm:density:cylinder:diag} in
Section~\ref{s:sep:diag:densities}, while the ergodic
technique used in the proof of
Theorem~\ref{th:c11:in:genus:0} is insufficient to provide
an explicit error term and leaves us with $o(N^d)$. It
would be very interesting to provide a more precise error
term in~\eqref{eq:c11:Q:nu}. Let us mention that in a
situation that shares many similarities with ours, namely,
for intersection of stable and unstable horospheres in the
space $\operatorname{SL}_d(\R) / \operatorname{SL}_d(\Z)$,
an explicit error term has been recently carried out
in~\cite{ElBazHuangLee}.


\section{From arc systems and meanders to square-tiled surfaces}
\label{s:From:arc:systems:and:meanders:to:square:tiled:surfaces}

In this Section we give precise bijections between meanders and
square-tiled surfaces with
a single maximal cylinder in both horizontal and
vertical directions. We consider meanders in the plane
in sections~\ref{ss:orientation}--\ref{ss:Meanders:and:square:tiled:surfaces:in:a:given:stratum}
and meanders on the sphere in Section~\ref{ss:Proofs:fractions}.
All the proofs of asymptotic results are based on
the results of Section~\ref{s:Equidistribution}.

\subsection{Orientation, marking and weight}
\label{ss:orientation}
We have seen in
Proposition~\ref{prop:pairs:multicurves:square:tiled:surfaces}
from
Section~\ref{ss:pairs:of:multicurves:as:square:tiled:surfaces}
that transverse connected pairs of multicurves on the
sphere are in bijection with square-tiled surfaces of genus
0. A square-tiled surface arising from a pair of arc
systems has a single horizontal band of squares. In
particular, a square-tiled surface arising from a meander
has a single horizontal and a single vertical band of
squares.

However, pairs of arc systems and meanders (both
in the plane and on the sphere) carry an extra marking.
Namely, a pair of arc systems comes with a given choice of a top and
bottom sides. Furthermore, the square-tiled surface
corresponding to a \textit{plane} meander has a special square corresponding
to the leftmost intersection. Summarizing, we get the following
result:

\begin{Lemma}
\label{lem:bijection:meander:square:tiled}
There is a natural bijection between meanders in the plane
and square-tiled surfaces with a marked oriented vertical
side of one of the squares that have a single horizontal
and a single vertical band of squares.
\end{Lemma}

In order to provide exact count of meanders we present the
conventions for count of square-tiled surfaces and see how
these quantities are related to arc systems and meander
count. We will consider square-tiled surfaces
with a marked \textit{vertex} of the square tiling.

\begin{Convention}
\label{conv:marked:on:top}
By convention, the marked vertex is located at the
end of the marked oriented vertical edge on the top boundary
component of the single horizontal cylinder.
\end{Convention}

Note that the two boundary components of the single
horizontal cylinder do not intersect. Thus, the marked vertex
uniquely defines the top boundary component and, hence,
provides us with the canonical orientation of the
waist curve of the single horizontal cylinder.

Let us reconstruct the labeled pair of arc systems in the plane from
a square-tiled surface of genus zero tiled with a single horizontal band
of squares and having a marked vertex. If the marked vertex of the square tiling is a simple pole of the
quadratic differential, there is a single vertical side of the
square tiling incident to it, and the choice of the vertical side is
canonical. If the marked vertex of the square tiling is a regular point of the
quadratic differential, there are two adjacent vertical sides, so
there are two ways to chose a distinguished vertical side which,
generally, lead to two different arc systems. We say ``generally''
because it might happen that the square-tiled surface is particularly
symmetric (like square-tiled surfaces associated to arc systems from
Figure~\ref{fig:pairs:of:arc:systems}) and the resulting two arc
systems are isomorphic.

As soon as we are interested only in the asymptotic
count we can simply neglect this issue: the square-tiled
surfaces with extra symmetries occur too rarely to affect
the asymptotics. To perform exact count we establish the
following standard Convention.

\begin{Convention}
\label{conv:symmetry}
We always count a marked or non-marked square-tiled surface with a weight
reciprocal to the order $|\Aut|$ of its automorphism group.
In the current context we keep track of which sides
of the square-tiled surface are horizontal and which ones are vertical,
but we do not label either the sides or the vertices of the square-tiled
surface. By definition, the automorphism group $\Aut$ acts by flat
isometries sending horizontal (respectively vertical) sides of the
tiling to horizontal (respectively vertical) sides and keeping the
marked point (if any) fixed.
\end{Convention}

In particular, if we have a marked point at a regular vertex of a
square-tiled surface, the automorphism group is either trivial or $\Z/2\Z$.
If we have a marked point at a zero of degree $\degofz$ of a square-tiled
surface, the automorphism group is a (usually trivial) subgroup of the
cyclic group $\Z/(\degofz+2)\Z$.

\subsection{Meanders with a given number of minimal arcs and square-tiled surfaces}
In this Section and in the next one we continue to work
with \textit{plane} meanders. Under
Conventions~\ref{conv:marked:on:top}
and~\ref{conv:symmetry}, any collection of weighted
square-tiled surfaces on the sphere with a single band of
horizontal squares and with a marked regular point defines
twice as many arc systems; the weighted collection of
square-tiled surfaces as above with a marked zero of degree
$\degofz$ defines $(\degofz+2)$ times more arc systems for
any $\degofz\in\N$.

\begin{Lemma}
Let the initial meander in the plane have $p$ minimal arcs,
where $p\ge 3$. The associated square-tiled surface has $p+1$ simple
poles if the initial meander has a maximal arc and $p$ simple poles
if it does not.
\end{Lemma}
\begin{proof}
A maximal arc becomes indistinguishable from a minimal arc after
passing to a labeled pair of transverse simple closed curves on the
sphere. Minimal and maximal arcs are in bijective correspondence with
bigons in the partition of the sphere by the union of these transverse
simple closed curves. Bigons, in turn, are in bijective
correspondence with simple poles of the associated square-tiled surface.
\end{proof}

Recall that $\cM_p^+(N)$ and $\cM_p^-(N)$ denote the number
of meanders with $p$ minimal arcs and with or without a
maximal arc respectively. Denote by $\cP_p(N)$ the number
of square-tiled surfaces of genus zero tiled with at most
$2N$ identical squares, having exactly $p$ simple poles and
having a single horizontal and a single vertical band of
squares. Denote by $\cP_{p,\degofz}(N)$, where
$\degofz=0,1,2,\dots$, the number of square-tiled surfaces
as above having in addition a marked point at a regular
vertex when $\degofz=0$ and at a zero of order $\degofz$
when $\degofz>0$.

Note that a square-tiled surface of genus $0$ with $p$
simple poles cannot have zeroes of order greater than $p-4$.

\begin{Lemma}
\label{lm:M:through:P}
Under Convention~\ref{conv:symmetry} on the weighted count of
square-tiled surfaces the following equalities hold:
\begin{align}
\label{eq:Mplus:P}
\cM^+_p(N)&= 2(p+1)\cdot \cP_{p+1}(N)
\\
\label{eq:Mminus:P}
\cM^-_p(N) &= \sum_{\degofz=0}^{p-4} (\degofz+2)\cdot \cP_{p,\degofz}(N)
\,-\,\frac{1}{2}\,\cM_{p-1}^+(N)
\,.
\end{align}
\end{Lemma}
\begin{proof}
If the meander has $2n$ intersections, then the associated square-tiled
surface is tiled with $2n$ identical squares.

To every meander with a maximal arc and with $p$ minimal
arcs we associated a canonical square-tiled surface of
genus zero with $p+1$ simple poles, a single horizontal and
a single vertical band of squares, see
Proposition~\ref{prop:pairs:multicurves:square:tiled:surfaces}.
Conversely, to every such square-tiled surface we can
associate $2(p+1)$ meanders with one maximal arc and $p$
minimal arcs. Indeed, choose any of the $(p+1)$ simple
poles and choose independently one of the two possible
orientations of the waist curve of the horizontal cylinder.
Cutting this waist curve at the intersection with the
single vertical edge of the square tiling adjacent to the
selected pole we get a meander in the plane with a maximal
arc.

It might happen that some of the resulting $2(p+1)$ meanders are
pairwise isomorphic. However, this implies that the automorphism group of
the square-tiled surface is nontrivial, and
Convention~\ref{conv:symmetry} provides the exact count. This
completes the proof of equality~\eqref{eq:Mplus:P}.

Similarly, to every meander without a maximal arc and with
$p$ minimal arcs we assigned a canonical square-tiled
surface of genus zero having $p$ simple poles, a single
horizontal and a single vertical band of squares, and a
marked vertex following
Convention~\ref{conv:marked:on:top}. The assumption that
the initial meander does not have a maximal arc excludes
coincidence of the marked point with a simple pole on the
upper side. In order to exclude a maximal arc on the lower
side, one needs to subtract a half of $\cM_{p-1}^+(N)$. At
the end of Section~\ref{ss:orientation} we have seen that
under Convention~\ref{conv:symmetry} on weights with which
we count square-tiled surfaces with a marked vertex, any
collection of weighted genus zero square-tiled surfaces
with a single horizontal and a single vertical band of
squares, and with a marked vertex of the tiling that is
regular in the flat metric, defines twice as much meanders
in the plane. A weighted collection of square-tiled
surfaces as above with a marked zero of degree $\degofz$
defines $(\degofz+2)$ times more meanders in the plane for
any $\degofz\in\N$. As before, if some of the resulting
meanders are isomorphic we do not count them several times
since by definition of the automorphism group $\Aut$ of the
corresponding square-tiled surface, the resulting
multiplicity coincides with the order $|\Aut|$ of the
automorphism group. This completes the proof of
equality~\eqref{eq:Mminus:P}.
\end{proof}

\subsection{Meanders and square-tiled surfaces in a given stratum}
\label{ss:Meanders:and:square:tiled:surfaces:in:a:given:stratum}
We now introduce finer count with respect to a fixed
stratum of meromorphic quadratic differentials. For a partition $\nu=[1^{\nu_1} 2^{\nu_2} \dots]$
denote by $\cM^+_\nu(N)$ and $\cM^-_\nu(N)$ the number of
meanders (with and without a maximal arc respectively)
giving rise to square-tiled surfaces in the stratum
$\cQ^{\textit{non-labeled}}(\nu, -1^{|\nu|+4})$. In the current setting
we do not label zeroes and poles of quadratic
differentials. We say that such
meanders are \textit{of type} $\nu$. Similarly, let
$\cP_\nu(N)$ be the number of genus zero square-tiled
surfaces in the stratum $\cQ^{\textit{non-labeled}}(\nu,-1^{|\nu|+4})$ tiled with
at most $2N$ identical squares, with a single horizontal
and a single vertical band of squares. Denote by
$\cP_{\nu,\degofz}(N)$, $\degofz=0,1,2,\dots$, the number
of square-tiled surfaces as above having in addition a
marked point at a regular vertex when $\degofz=0$ and at a
zero of order $\degofz$ when $\degofz>0$. By definition, we
let $\cP_{\nu,\degofz}(N)=0$ for any $N$ when
$\nu_\degofz=0$. Recall that by
Convention~\ref{conv:symmetry} we count square-tiled
surfaces with weights reciprocal to the orders of their
automorphism groups.

\begin{Lemma}
\label{lm:M:through:P:nu}
Under Convention~\ref{conv:symmetry} on weights with which we count
square-tiled surfaces the following equalities hold
\begin{align}
\label{eq:Mplus:P:nu}
\cM^+_\nu(N)&= 2(|\nu|+4)\cdot \cP_\nu(N)\,
\\
\label{eq:Mminus:P:nu}
\cM^-_\nu(N) &= \sum_{\degofz=0}^{|\nu|} (\degofz+2)\cdot \cP_{\nu,\degofz}(N)
\,-\,\frac{1}{2}\,\cM_{\nu}^+(N)
\,.
\end{align}
\end{Lemma}
\begin{proof}
The proof is completely analogous to the proof of
Lemma~\ref{lm:M:through:P}.
\end{proof}

\subsection{Asymptotic density of meanders: general setting}
\label{ss:Proofs:fractions}

In this Section we return to meanders on the sphere. Let
$\cT$ be a plane tree. We associate to $\cT$ a generalized
integer partition $\nu = \nu(\cT) = [0^{\nu_0} 1^{\nu_1}
2^{\nu_2} \ldots]$, where $\nu_\degofz$ denotes the number
of vertices of valence $\degofz+2$ for
$\degofz=0,1,2\dots$. The number of leaves, or,
equivalently, of vertices of valence $1$, is then expressed
in terms of the (generalized) partition $\nu$ as $|\nu|+4$,
where $|\nu|=1\cdot \nu_1 + 2\cdot \nu_2 + 3\cdot \nu_3 +
\dots$.

Given two generalized partitions $\iota = [0^{\iota_0}
1^{\iota_1} 2^{\iota_2} \dots]$ and $\kappa= [0^{\kappa_0}
1^{\kappa_1} 2^{\kappa_2} \dots]$ we define their sum as
$\nu=\iota+\kappa= [0^{\iota_0+\kappa_0}
1^{\iota_1+\kappa_1} 2^{\iota_2+\kappa_2} \dots]$. We say
that $\iota$ is a \textit{subpartition} of $\nu$ if for all
$i \geq 0$ we have $\iota_i \leq \nu_i$. We use notation
$\iota \subset \nu$ to indicate that $\iota$ is a
subpartition of $\nu$ and define the difference
$\kappa=\nu-\iota$ in the natural way.

We formulate and prove the following generalization of
Theorem~\ref{th:trivalent:trees:connected:proportion}
giving a formula for the limit of the
fraction~\eqref{eq:p:connected:iota:kappa} of meanders
among all gluings
which we get identifying arc systems of types $\cT_{top}$
and $\cT_{bottom}$ with the same number of arcs, see
Figure~\ref{fig:pairs:of:arc:systems}.

Though we agreed in Section~\ref{ss:Meanders:and:arc:systems}
to consider reduced trees, suppressing
the vertices of valence $2$, it is often convenient to keep several
marked points, so we state the Theorem below in this slightly more
general setting. Note that since $f(0)=2$, the number $\nu_0$ of zeroes in the
partition $\nu$ affects the value of the
function $\Vol_1 \cQ_1(\nu,-1^{|\nu|+4})$. Adding an extra marked point
we double the Masur--Veech volume of the corresponding stratum.

\begin{Theorem}
\label{th:any:trees:connected:proportion}
For any pair of plane trees $\cT_{top}, \cT_{bottom}$ with
associated generalized partitions $\nu_{top}$ and $\nu_{bottom}$
we have:
$$
\lim_{N\to+\infty} \prob_{\mathit{connected}}(\cT_{top}, \cT_{bottom}; N)
=
\prob_1(\cQ(\nu,-1^{|\nu|+4}))>0\,,
$$
where $\nu = \nu_{top} + \nu_{bottom}$ and
\begin{equation}
\label{eq:probability}
\prob_1\left(\cQ(\nu,-1^{|\nu|+4})\right)
=
\frac{\cyl_1(\cQ(\nu,-1^{|\nu|+4}))}
{\Vol_1 \cQ_1(\nu,-1^{|\nu|+4})}\,.
\end{equation}
Here
\begin{equation}
\label{eq:c1:nu}
\cyl_1(\cQ(\nu,-1^{|\nu|+4}))
= 2\sum_{\mu\subset\nu}
\binom{|\nu|+4}{|\mu|+2}
\binom{\nu_0}{\mu_0}
\binom{\nu_1}{\mu_1}
\binom{\nu_2}{\mu_2}
\cdots
\,.
\end{equation}
and
$\Vol_1 \cQ_1(\nu,-1^{|\nu|+4})$ is given by~\eqref{eq:volume}.
\end{Theorem}

\begin{proof}
Under our correspondence which associates to pairs of
transverse multicurves square-tiled surfaces, the trees
$\cT_{top}$ and $\cT_{bottom}$ represent the trees formed
by the horizontal saddle connections of the square-tiled
surface (see
Section~\ref{s:Computations:for:square:tiled:surfaces} for
details).

Vertices of valence one are in bijective correspondence
with simple poles. Vertices of valence two represent marked
points (if any). Vertices of valence $\degofz+2$ are in
bijective correspondence with zeroes of degree $\degofz$
for $\degofz\in\N$. Recall that the type $\nu=[1^{\nu_1}
2^{\nu_2} 3^{\nu_3} \dots]$ of the graph
$\cT_{top}\sqcup\cT_{bottom}$ encodes the total number
$\nu_\degofz$ of vertices of valence $\degofz+2$ in
$\cT_{top}\sqcup\cT_{bottom}$ for $\degofz\in\N$. We
conclude that a pair of arc systems having $\cT_{top}$
and $\cT_{bottom}$ as dual trees defines a square-tiled
surface in the stratum $\cQ^{\textit{non-labeled}}(\nu,-1^{|\nu|+4})$ of
meromorphic quadratic differentials.

We are ready to express the numerator and the denominator
of the right hand side in~\eqref{eq:p:connected:iota:kappa}
in terms of square-tiled surfaces. First, note that arc
systems are defined on a pair of labeled oriented discs
(called top and bottom, or northern and southern
hemispheres).

Every triple
$$
(\text{$n$-arc system of type $\cT_{top}$;
 $n$-arc system of type $\cT_{bottom}$;
 identification})
$$
in the denominator in~\eqref{eq:p:connected:iota:kappa}
defines a unique isomorphism class of ordered transverse
connected pairs of multicurves in the sense of
Definition~\ref{def:pair:of:multicurves}, such that the
``horizontal'' multicurve is connected and oriented, and
the number of intersections is equal to $2n$. We identify
triples leading to isomorphic pairs of labeled multicurves.
Pairs of multicurves associated to the triples from the
numerator of~\eqref{eq:p:connected:iota:kappa} are
distinguished by the property that both ``horizontal'' and
``vertical'' multicurves are connected.

Applying
Proposition~\ref{prop:pairs:multicurves:square:tiled:surfaces}
and constructions of
Section~\ref{ss:pairs:of:multicurves:as:square:tiled:surfaces}
we associate square-tiled surfaces to the resulting
transverse connected pairs of multicurves. We get a
surjective map from the triples in the denominator
of~\eqref{eq:p:connected:iota:kappa} to the set of
square-tiled surfaces in the stratum
$\cQ^{\textit{non-labeled}}(\nu,-1^{|\nu|+4})$ having a
single horizontal band of squares, and having the
ribbon graph $\Gamma:=\cT_{top}\sqcup\cT_{bottom}$ as the
diagram of horizontal saddle connections. We denote the
number of such square-tiled surfaces tiled with at most
$2N$ identical squares by $\cS_{\Gamma,1,\nu}(N)$.

Restricting the above map to the triples in the numerator
of~\eqref{eq:p:connected:iota:kappa} we get a surjective
map to the subset of square-tiled surfaces as above which
have a single horizontal and a single vertical band of
squares. We denote the number of such square-tiled surfaces
tiled with at most $2N$ identical squares by
$\cS_{\Gamma,1,1,\nu}(N)$.

Note that the waist curve of the
horizontal cylinder of the resulting square-tiled surface
is not oriented, while the boundary circle of each arc
system in the disc is oriented. When
$\cT_{top}$ and $\cT_{bottom}$ are not isomorphic as ribbon
graphs, there is a canonical way to orient the waist curve
of the single horizontal cylinder:
choose the orientation of the vertical
direction so that $\cT_{top}$ is on top of the cylinder;
the canonical orientation of the cylinder
determines the orientation of the waist curve.
In this case the constructed
map is a bijection.

When $\cT_{top}$ and $\cT_{bottom}$ are isomorphic, a
negligible part of square-tiled surfaces as above has extra
symmetry, namely, an isometry interchanging the two
components of the cylinder. Such isometry changes the
orientation of the horizontal waist curve. When $\cT_{top}$
and $\cT_{bottom}$ are isomorphic we count square-tiled
surfaces in the image of the above map with weight $1/2$
when the square-tiled surface is symmetric in the above
sense and with weight $1$ otherwise. In this case both
the numerator and the denominator
in~\eqref{eq:p:connected:iota:kappa} give twice the
weighted count of the associated square-tiled surfaces.

Thus, the limit in
Theorem~\ref{th:any:trees:connected:proportion}
can be expressed as follows:
\begin{equation}
\label{eq:limit:top:bottom}
\lim_{N\to+\infty} \prob_{\mathit{connected}}(\cT_{top}, \cT_{bottom}; N)
=\lim_{N\to+\infty}
\frac{\cS_{\Gamma,1,1,\nu}(N)}{\cS_{\Gamma,1,\nu}(N)}\,.
\end{equation}

It will be convenient to pass to the quantities
$\cS^{\textit{labeled}}_{\Gamma,1,1,\nu}(N)$ and
$\cS^{\textit{labeled}}_{\Gamma,1,\nu}(N)$ counting the
square-tiled surfaces as above, but with labeled zeroes and
poles. By Remark~\ref{rm:labeled:zeroes:and:poles} they
differ from $\cS_{\Gamma,1,1,\nu}(N)$ and
$\cS_{\Gamma,1,\nu}(N)$ by the common constant factor and
hence, we get the same ratio as above. By
Lemma~\ref{lm:fixed:Gamma:h:kvert} at the end of
Section~\ref{s:sep:diag:densities} we have
$$
\lim_{N\to+\infty}
\frac{\cS^{\textit{labeled}}_{\Gamma,1,1,\nu}(N)}{\cS^{\textit{labeled}}_{\Gamma,1,\nu}(N)}
=
\lim_{N\to+\infty}
\frac{\cS^{\textit{labeled}}_{1,\nu}(N)}{\cS^{\textit{labeled}}_{\nu}(N)}\,.
$$
The above formula is just another manifestation of the general
principal according to which ``horizontal and vertical
cylinder decompositions are asymptotically uncorrelated'':
defining the counting functions on the right hand side of
the above equation we omit the conditions on the horizontal
decomposition imposed in the definition of the
corresponding counting functions on the left hand side.

Applying~\eqref{eq:VolQ:N:d} and~\eqref{eq:c1:N:d} for the
numerator and the denominator of the latter fraction
respectively, we prove that the limit in the left hand side
of~\eqref{eq:limit:top:bottom} exists and that its value is
given by~\eqref{eq:probability}. The remaining
formula~\eqref{eq:c1:nu} is equivalent to
formula~\eqref{eq:c1:answer} from
Theorem~\ref{th:c1:in:genus:0} which will be proved in
Section~\ref{s:Computations:for:square:tiled:surfaces}.
\end{proof}

\begin{proof}[Proof of Theorem~\ref{th:trivalent:trees:connected:proportion}]
Theorem~\ref{th:trivalent:trees:connected:proportion} is a particular case
of the Theorem~\ref{th:any:trees:connected:proportion} when the
plane trees $\cT_{top},\cT_{bottom}$
are trivalent and
have the total number $p$ of leaves
(vertices of valence one). In this situation
$\nu=[1^{p-4}]$.
By Theorem~\ref{th:any:trees:connected:proportion} we have
$$
\lim_{N\to+\infty}
\prob_{\mathit{connected}}(\cT_{top},\cT_{bottom}; N)
=
\prob_1(\cQ(1^{p-4},-1^p))
=
\frac{\cyl_1(\cQ(1^{p-4},-1^p))}
{\Vol_1 \cQ_1(1^{p-4},-1^p)}\,.
$$
It remains to apply~\eqref{eq:volume}
and~\eqref{eq:c1:principal:answer} respectively
for the denominator and the numerator of the
latter fraction to complete the proof. Formula~\eqref{eq:c1:principal:answer}
will be proved in
Section~\ref{s:Computations:for:square:tiled:surfaces}.
\end{proof}

\subsection{Counting meanders of special combinatorial types}
\label{ss:Proofs:number:of:meanders}

In this Section we return to plane meanders with exception
for the proof of
Theorem~\ref{th:p:fixed:number:of:leaves:general} at the
very end of the Section, where we work with meanders on the
sphere.

We state now an analog of Theorem~\ref{th:meander:counting}, where
instead of the number of minimal arcs (pimples) we use the partition
$\nu$ as a combinatorial passport of the meander.

\begin{Theorem}
\label{gth:meanders:fixed:stratum}
For any partition $\nu=[1^{\nu_1} 2^{\nu_2} 3^{\nu_3} \dots]$,
the number
$\cM^+_\nu(N)$ (respectively $\cM^-_\nu(N)$) of
plane meanders of type $\nu$, with (respectively without)
a maximal arc and with at most $2N$ crossings has
the following asymptotics as $N\to+\infty$:
\begin{align}
\label{eq:asymptotics:nu:plus}
\cM^+_\nu(N)&=
2(|\nu|+4)\cdot
\frac{\cyl_{1,1}\big(\cQ(\nu,-1^{|\nu|+4})\big)}
{(|\nu|+4)!\cdot\prod_j \nu_j!}
\cdot
\frac{N^{\ell(\nu)+|\nu|+2}}{2\ell(\nu)+2|\nu|+4}
\ +\\
\notag
&\hspace{2.7in} +
o\big(N^{\ell(\nu)+|\nu|+2}\big)\,,
\\
\label{eq:asymptotics:nu:minus}
\cM^-_\nu(N)&=
\frac{2\,\cyl_{1,1}\big(\cQ(\nu,0,-1^{|\nu|+4})\big)}
{(|\nu|+4)!\cdot\prod_j \nu_j!}
\cdot
\frac{N^{\ell(\nu)+|\nu|+3}}{2\ell(\nu)+2|\nu|+6}
\ +\\
\notag
&\hspace{2.7in} +
o\big(N^{\ell(\nu)+|\nu|+3}\big)\,.
\end{align}
Moreover, we have
\begin{multline}
\label{eq:c11:nu:in:Th4}
\cyl_{1,1}\big(\cQ(\nu,-1^{|\nu|+4})\big)
=\\=
\frac{4}{\Vol_1 \cQ_1(\nu,-1^{|\nu|+4})}\cdot
\left(
\sum_{\iota_1=0}^{\nu_1}
\sum_{\iota_2=0}^{\nu_2}
\sum_{\dots}^{\dots}
\binom{\nu_1}{\iota_1}
\binom{\nu_2}{\iota_2}
\cdots
\binom{|\nu|+4}{|\iota|+2}
\right)^2
\end{multline}
and
\begin{equation}
\label{eq:c11:with:0:through:c11:without:0}
\cyl_{1,1}\big(\cQ(\nu,0,-1^{|\nu|+4})\big)
=
2\cdot \cyl_{1,1}\big(\cQ(\nu,-1^{|\nu|+4})\big)\,.
\end{equation}
\end{Theorem}

Note that contrary to the original Theorem~\ref{th:meander:counting},
where the setting is somewhat misleading, in the setting of
Theorem~\ref{gth:meanders:fixed:stratum} we get more natural
formula $\cM^+_\nu(N)=o\big(\cM^-_\nu(N)\big)$ as $N\to+\infty$.

Up to now we performed the exact count. The Lemma below gives the
term with dominating contribution to the asymptotic count when
the bound $2N$ for
the number of squares in the square-tiled surface tends to
infinity.

\begin{Lemma}
\label{lm:P:Pprincipal}
We have the following limits:
\begin{align}
\label{eq:P:pN:to:principal}
\lim_{N\to+\infty}& \frac{1}{\cP_{1^{p-4}}(N)}\cdot
\cP_p(N)
= 1\,,
\\
\label{eq:Pd:pN:to:principal:0}
\lim_{N\to+\infty}& \frac{1}{2\,\cP_{1^{p-4},0}(N)}\cdot
\left(\sum_{\degofz=0}^{p-4} (\degofz+2)\cdot\cP_{p,\degofz}(N)\right)
= 1\,,
\\
\label{eq:P:nu:d:to:P:nu:0}
\lim_{N\to+\infty}& \frac{1}{2\,\cP_{\nu,0}(N)}\cdot
\left(\sum_{\degofz=0}^{|\nu|} (\degofz+2)\cdot \cP_{\nu,\degofz}(N)\right) =1\,.
\end{align}
\end{Lemma}
\begin{proof}
Let $\nu=[1^{\nu_1} 2^{\nu_2}\dots]$ be a partition of
$p-4$. By definition, $\cP_\nu(N)$ and $\cS_{1,1,\nu}(N)$
denote the same quantities. Using
Remark~\ref{rm:labeled:zeroes:and:poles} to express
$\cS_{1,1,\nu}(N)$ in terms of
$\cS^{\textit{labeled}}_{1,1,\nu}(N)$ and
applying~\eqref{eq:c11:Q:nu} with $k_h=k_v=1$ we get
\begin{multline}
\label{eq:P:nu:as:c11}
\cP_\nu(N)
=\cS_{1,1,\nu}(N)
=\frac{1}{p!\cdot \prod_j \nu_j!}
\cdot\cS^{\textit{labeled}}_{1,1,\nu}(N)
=\\=
\frac{\cyl_{1,1}\big(\cQ(\nu,-1^p)\big)}
{p!\cdot \prod_j \nu_j!}
\cdot
\frac{N^d}{2d} + o(N^d)\,,
\text{ when }N\to+\infty\,.
\end{multline}
where $d$ is defined in~\eqref{eq:dim}.

For a given number $p\ge 4$ of simple poles, the only
stratum of the maximal dimension $2p-6$
is the principal stratum
$\cQ(1^{p-4},-1^p)$, for which all zeroes
of the quadratic differentials
are simple. Thus, this is
the only stratum which contributes a term of order
$N^{2p-6}$ to $\cP_p(N)$. This proves~\eqref{eq:P:pN:to:principal}.

For $\degofz\ge 1$ the quantity $\cP_{\nu,\degofz}(N)$
counts square-tiled surfaces with a marked zero of order
$\degofz$ in the stratum $\cQ^{\textit{non-labeled}}(\nu,-1^{|\nu|+4})$. Hence, it
has the asymptotic growth rate of the same order as the
quantity $\cP_{\nu}(N)$ counting unmarked square-tiled
surfaces in the same stratum, i.~e. it grows like $N^d$,
where $d=\dim_{\mathbb{C}}\cQ(\nu,-1^{|\nu|+4})$. The
dimensional count as above implies that the contribution of
any term $\cP_{\nu,\degofz}(N)$ with $\degofz\ge 1$ to the
sum in the right hand side
of~\eqref{eq:Pd:pN:to:principal:0} has the order at most
$N^{2p-6}$.

Let us analyse now the contribution of various strata to $\cP_{p,0}(N)$.
In the same way as we got~\eqref{eq:P:nu:as:c11} we obtain
\begin{multline}
\label{eq:P:nu:0:as:c11}
\cP_{\nu,0}(N) =
\frac{\cyl_{1,1}\big(\cQ(\nu,0,-1^{|\nu|+4})\big)}
{(|\nu|+4)!\cdot\prod_j \nu_j!}
\cdot
\frac{N^{\ell(\nu)+|\nu|+3}}{2\ell(\nu)+2|\nu|+6}
\ +\\+\
o\big(N^{\ell(\nu)+|\nu|+3}\big)\,,
\text{ when }N\to+\infty\,,
\end{multline}
where the constant
$\cyl_{1,1}\big(\cQ(\nu,0,-1^{|\nu|+4})\big)$ is strictly
positive. This implies that for any partition $\nu$ of
$p-4$ different from $1^{p-4}$, its contribution
$\cP_{\nu,0}$ also has order at most $N^{2p-6}$. We
conclude that $\cP_{p,0}(N)$ has order $N^{2p-5}$ for $N$
large, and that the only stratum which gives a contribution
of this order is the principal stratum with a marked point
$\cQ(1^{p-4},0,-1^p)$. This proves
equality~\eqref{eq:Pd:pN:to:principal:0}.

By the same reason the summand $2\,\cP_{\nu,0}(N)$ dominates in the sum in the right-hand side
of~\eqref{eq:P:nu:d:to:P:nu:0}. It is the only term whose
contribution is of order $N^{d+1}$, where
$d=\dim_{\mathbb{C}}\cQ(\nu,-1^{|\nu|+4})$. The asymptotics of other
terms in the sum have lower orders in $N$ as $N\to+\infty$. This proves
equality~\eqref{eq:P:nu:d:to:P:nu:0}.
\end{proof}

Now we have everything for the proofs of
Theorems~\ref{th:meander:counting} and~\ref{gth:meanders:fixed:stratum}.

\begin{proof}[Proof of Theorem~\ref{th:meander:counting}]
The chain of relations starting with~\eqref{eq:Mplus:P} and
continuing with~\eqref{eq:P:pN:to:principal}
and~\eqref{eq:P:nu:as:c11} yields
\begin{multline*}
\cM^+_p(N)= 2(p+1)\cdot \cP_{p+1}(N)
=
2(p+1)\cdot \cP_{1^{p-3}}(N) + o(N^{2p-4})\
=\\=
2(p+1)\cdot
\frac{\cyl_{1,1}\big(\cQ(1^{p-3},-1^{p+1})\big)}
{(p+1)!\,(p-3)!}
\cdot
\frac{N^{2p-4}}{4p-8} + o(N^{2p-4})\,
\text{ when }N\to+\infty\,.
\end{multline*}
This proves the first equality
in~\eqref{eq:asymptotics:with}. The constant
$\cyl_{1,1}\big(\cQ(1^{p-3},-1^{p+1})\big)$ is expressed by
Formula~\eqref{eq:c11:as:c1:squared:over:Vol} in
terms of $\cyl_1\big(\cQ(1^{p-3},-1^{p+1})\big)$ computed
in Corollary~\ref{cor:principal} of
Section~\ref{s:Computations:for:square:tiled:surfaces} and
in terms of the Masur--Veech volume of the stratum
$\cQ(1^{p-3},-1^{p+1})$ given by Formula~\eqref{eq:volume}.

Similarly, the chain of relations including
\eqref{eq:Mminus:P},
\eqref{eq:Pd:pN:to:principal:0} and
\eqref{eq:P:nu:0:as:c11} implies
\begin{multline*}
\cM^-_p(N) = \sum_{\degofz=0}^{p-4} (\degofz+2)\cdot \cP_{p,\degofz}(N)
\,-\,\frac{1}{2}\,\cM_{p-1}^+(N)
\ =\
2\,\cP_{1^{p-4},0}(N)+ o(N^{2p-5})\
=\\=
\frac{2\,\cyl_{1,1}\big(\cQ(1^{p-4},0,-1^p)\big)}
{p!\,(p-4)!}\cdot
\frac{N^{2p-5}}{4p-10} + o(N^{2p-5})\,,
\text{ when }N\to+\infty\,.
\end{multline*}

This proves the first equality in~\eqref{eq:asymptotics:without}. The
constant $\cyl_{1,1}\big(\cQ(1^{p-4},0,-1^p)\big)$ is expressed by our
main formula~\eqref{eq:c11:as:c1:squared:over:Vol} in terms of
$\cyl_1\big(\cQ(1^{p-4},0,-1^p)\big)$ computed in
Corollary~\ref{cor:principal} of
Section~\ref{s:Computations:for:square:tiled:surfaces} and in terms of
the Masur--Veech volume of the stratum $\cQ(1^{p-4},0,-1^p)$ given by
formula~\eqref{eq:volume}.

Thus, the proof of Theorem~\ref{th:meander:counting} is conditional subject to the
explicit count of $\cyl_1\big(\cQ(1^{p-3},-1^{p+1})\big)$ and of
$\cyl_1\big(\cQ(1^{p-4},0,-1^p)\big)$ performed in
Corollary~\ref{cor:principal} below.
\end{proof}

\begin{proof}[Proof of Theorem~\ref{gth:meanders:fixed:stratum}]
The proof of Theorem~\ref{gth:meanders:fixed:stratum} is completely
analogous to the proof of Theorem~\ref{th:meander:counting}.

Combining~\eqref{eq:Mplus:P:nu} with
relation~\eqref{eq:P:nu:as:c11}, we get
\begin{multline*}
\cM^+_\nu(N)= 2(|\nu|+4)\cdot \cP_\nu(N)
\ =\\=\
2(|\nu|+4)\cdot
\frac{\cyl_{1,1}\big(\cQ(\nu,-1^{|\nu|+4})\big)}
{(|\nu|+4)!\cdot\prod_j \nu_j!}
\cdot
\frac{N^{\ell(\nu)+|\nu|+2}}{2\ell(\nu)+2|\nu|+4}
\ +\\+\ o\big(N^{\ell(\nu)+|\nu|+2}\big)\,,
\text{ when }N\to+\infty\,,
\end{multline*}
where we use Formula~\eqref{eq:dim} for the dimension $d$
of the stratum $\cQ(\nu,-1^{|\nu|+4})$. This proves
formula~\eqref{eq:asymptotics:nu:plus} in
Theorem~\ref{gth:meanders:fixed:stratum}.

Similarly, combining~\eqref{eq:Mminus:P:nu}
with~\eqref{eq:P:nu:d:to:P:nu:0}
and~\eqref{eq:P:nu:0:as:c11} we get
\begin{multline*}
\cM^-_\nu(N)
=
\sum_{\degofz=0}^{|\nu|} (\degofz+2)\cdot \cP_{\nu,\degofz}(N)
\,-\,\frac{1}{2}\,\cM_{\nu}^+(N)
=
2\,\cP_{\nu,0}(N)+o\big(N^{2\ell(\nu)+2|\nu|+3}\big)
=\\=
\frac{2\,\cyl_{1,1}\big(\cQ(\nu,0,-1^{|\nu|+4})\big)}
{(|\nu|+4)!\cdot\prod_j \nu_j!}
\cdot
\frac{N^{\ell(\nu)+|\nu|+3}}{2\ell(\nu)+2|\nu|+6}
\ +\\+\
o\big(N^{\ell(\nu)+|\nu|+3}\big)\,,
\text{ when }N\to+\infty\,.
\end{multline*}
This proves formula~\eqref{eq:asymptotics:nu:minus}
in Theorem~\ref{gth:meanders:fixed:stratum}.

We have proved Theorem~\ref{gth:meanders:fixed:stratum} conditional to
expressions~\eqref{eq:c11:with:0:through:c11:without:0}
and~\eqref{eq:c11:nu:in:Th4} for the quantities
$\cyl_{1,1}(\cQ(\nu,0,-1^{|\nu|+4}))$ and
$\cyl_{1,1}(\cQ(\nu,-1^{|\nu|+4}))$. We prove them
together with relation~\eqref{eq:c1:answer} in
Section~\ref{s:Computations:for:square:tiled:surfaces}.
\end{proof}

We conclude this Section with the proof of
Theorem~\ref{th:p:fixed:number:of:leaves:general}.

\begin{proof}[Proof of Theorem~\ref{th:p:fixed:number:of:leaves:general}]
Following the same arguments as in the
proof of Theorem~\ref{th:any:trees:connected:proportion}
we obtain
$$
\lim_{N\to+\infty} \prob_{\mathit{connected}}(p; N)
=\lim_{N\to+\infty}
\frac{\sum_{|\nu|=p-4}\cS^{\textit{labeled}}_{1,1,\nu}(N)}
{\sum_{|\nu|=p-4}\cS^{\textit{labeled}}_{1,\nu}(N)}\,,
$$
where $\cS^{\textit{labeled}}_{1,\nu}(N)$ and
$\cS^{\textit{labeled}}_{1,1,\nu}(N)$ are as in
Theorems~\ref{th:c1:in:genus:0} and~\ref{th:c11:in:genus:0}
respectively. By~\eqref{eq:c1:N:d} and~\eqref{eq:c11:Q:nu},
for any partition $\nu$ the latter quantities have order
$N^d$, where $d=\dim_{\C}\cQ(\nu,-1^p)$. Among all partitions
$\nu$ satisfying $|\nu|=p-4$, there is only one partition, namely $[1^{p-4}]$,
which defines the stratum $\cQ(1^{p-4},-1^p)$ of maximal dimension
$2p-6$. Hence, the contributions of other strata to both sums are
negligible with respect to the contributions of the
principal stratum. This implies that
$$
\lim_{N\to+\infty} \prob_{\mathit{connected}}(p; N)
=\lim_{N\to+\infty}
\frac{\cS^{\textit{labeled}}_{1,1,[1^{p-4}]}(N)}
{\cS^{\textit{labeled}}_{1,[1^{p-4}]}(N)}\,.
$$
Using expressions~\eqref{eq:c1:N:d} and~\eqref{eq:c11:Q:nu}
for the denominator and the numerator of the latter
fraction respectively and
applying~\eqref{eq:c11:as:c1:squared:over:Vol} we get
$$
\lim_{N\to+\infty}
\prob_{\mathit{connected}}(p; N)
=\frac{\cyl_1(\cQ(1^{p-4},-1^p))}
{\Vol_1 \cQ_1(1^{p-4},-1^p)}\,.
$$
The latter ratio was denoted by
$\prob_1(\cQ(1^{p-4},-1^p))$ and computed in
Theorem~\ref{th:trivalent:trees:connected:proportion}.
\end{proof}


\section{Enumeration of square-tiled surfaces}
\label{s:Equidistribution}

In the current Section we prove Theorem~\ref{th:c11:in:genus:0} from the
Introduction and in particular the key equality~\eqref{eq:c11:as:c1:squared:over:Vol}.
This result is actually a particular case of the more general
Corollary~\ref{cor:density:fixed:num:comp} that concerns surfaces of any genera.

In plain terms we prove that the asymptotic proportion of
square-tiled surface having a single vertical band of
squares among square-tiled surface having a single
horizontal band of squares is the same as the asymptotic
proportion of square-tiled surface having a single vertical
band of squares among all square-tiled squares (when the
bound for the number of squares in the square-tiling
grows). This equality can be seen as a particular case of
asymptotic non-correlation of horizontal and vertical
decompositions of square-tiled surfaces tiled with large
number of squares. Together with the explicit
formula~\eqref{eq:volume} for the Masur--Veech volume
$\Vol_1 \cQ_1(\nu,-1^{|\nu|+4})$ of any stratum in the
moduli space of quadratic differentials in genus zero, the
equality~\eqref{eq:c11:as:c1:squared:over:Vol} is the
principal ingredient of the proof of
Theorem~\ref{th:any:trees:connected:proportion}, providing
the asymptotic frequency of meanders out of general
couplings of arc systems on a pair of hemispheres.

Our non-correlation results are, actually, much more
general in two aspects. On the one hand, they are
applicable to much more general horizontal and vertical
cylinder decompositions of square-tiled surfaces compared
to ``single band of squares'' needed for the count of
meanders, see Corollaries~\ref{cor:fixed:horiz:vert:decompositions},
and~\ref{cor:density:fixed:num:comp}
as well as Remark~\ref{rm:other:combinations}.
On the other hand, these non-correlation results
are applicable to general closed connected $\GLR$-invariant
suborbifolds defined over $\Q$ in any genus (and not only
to strata of quadratic differentials in genus zero needed
for meanders).

\subsection{Invariant arithmetic orbifolds}
\label{ss:background:flat:surf}
%

Given a generalized partition $\kappa$ composed of
non-negative integers $\kappa = [0^{\kappa_0} 1^{\kappa_1}
\ldots]$ satisfying $\sum_{i \geq 0} i \kappa_i = 2g-2$, we
denote by $\cH(\kappa)$ the moduli space of Abelian
differentials with $\kappa_i$ zeros of order $i$,
for $i=0,1,\dots$,
on complex curves of genus $g$. Namely,
an element of $\cH(\kappa)$ is a tuple
$(X, \omega, \{P_{i,j}\}_{i \geq 0, 1 \leq j \leq
\kappa_i})$ where $X$ is a smooth complex curve of genus $g$
endowed with a non-zero holomorphic $1$-form $\omega$
having zero of order $i$ at each of the points
$P_{i,1},\dots,P_{i,\kappa_i}$ of $X$ (for those $i$ for
which $\kappa_i>0$) and non-vanishing outside of the finite
set $\Sigma := \{P_{i,j}\}$. Here ``zero of order zero'' is
interpreted as a marked point. Each stratum of Abelian
differential is locally modeled on the relative cohomology
$H^1(X, \Sigma; \C)$.

A \textit{translation chart} for an Abelian differential
$\omega$ is a local coordinate $z$ on $X$ in which $\omega
= dz$. A translation chart defines a flat metric with the
area form $\frac{i}{2}dz\wedge d\bar{z}=dx\wedge dy$. The
group $\GLR$ acts on $\cH(\kappa)$ by postcomposition in
translation charts. By definition
\[
\forall A \in \GLR,
\quad
\Area(A \cdot (X,\omega,\Sigma)) = \det(A)\cdot \Area((X,\omega,\Sigma))\,.
\]
In particular, the subgroup $\SLR$ preserves the ``unit hyperboloid''
$\cH_1(\kappa) \subset \cH(\kappa)$ that consists of Abelian differentials of area one.

\begin{Definition}
A \emph{square-tiled surface} in a stratum of Abelian
differential $\cH(\kappa)$ is an Abelian differential $(X,
\omega, \Sigma)$ such that all periods of $\omega$
relative to $\Sigma$ belong to $\Z \oplus i \Z$.
\end{Definition}

Square-tiled surfaces correspond to integral points $H^1(X,
\Sigma; \Z \oplus i \Z)$ in period coordinates $H^1(X,
\Sigma; \C)$. The action of the subgroup $\SLZ \subset
\SLR$ preserves the set of square-tiled surfaces.

Alternatively, a square-tiled surface can be defined as a
connected finite ramified cover over the square torus $\C /
(\Z \oplus i \Z)$ with all ramification points (if any)
located over the same branch point of the torus.
Thus, in plain terms, a square-tiled translation surface is
obtained by gluing a finite set of unit squares endowed
with distinguished horizontal and vertical directions by
translation respecting these distinguished directions
and identifying sides to sides and vertices to vertices.

The fundamental result of Eskin, Mirzakhani and
Mohammadi generalizes to closed $\SLR$-invariant subsets
all the structures we described for strata.

\begin{NNTheorem}[\cite{Eskin:Mirzakhani}, \cite{Eskin:Mirzakhani:Mohammadi}]
\label{thm:EMM}
Let $\operatorname{P}\subset\SLR$ be the subgroup of
upper-triangular matrices. For any $S=(X,\omega)$ in any
stratum $\cH(\kappa)$ of Abelian differentials the closures
$\overline{P\cdot S}$ and $\overline{\SLR\cdot S}$ of
orbits of $S$ under the actions of $\operatorname{P}$ and of $\SLR$ coincide.

The orbit closure $\cL=\overline{\GLR\cdot S}$ is locally
described in period coordinates as a finite union of
complexifications of linear subspaces in $H^1(X, \Sigma;
\R)$. Any such $\cL$ admits a volume element which is
linear in the corresponding linear subspaces of
$H^1(X, \Sigma;\R)$.

Any ergodic $\operatorname{P}$-invariant probability
measure on the ``unit hyperboloid'' $\cH_1(\kappa)$ is
always $\SLR$-invariant. It is supported on the
intersection $\cH_1(\kappa)$ with some orbit closure
$\cL=\overline{\GLR\cdot S}$. It can be induced from an
appropriately normalized linear volume element in period
coordinates of $\cL$ by the natural restriction to the
level hypersurface $\cL_1=\cL\cap\cH_1(\kappa)$ of the area
function.

Reciprocally, for any orbit closure
$\cL=\overline{\GLR\cdot S}$, the measure corresponding to
the volume element on $\cL_1=\cL\cap\cH_1(\kappa)$ induced
from any linear volume element in period
coordinates of $\cL$ by the natural restriction to the
level hypersurface $\cL_1$ is ergodic and finite.
\end{NNTheorem}

In the current paper we will only consider the case when
all these linear subspaces are defined by systems of linear
equations with \textit{rational} coefficients with respect
to the natural integral structure on $H^1(X, \Sigma;
\Z)\subset H^1(X, \Sigma; \R)$. In such situation we say
that $\cL$ is \emph{defined over $\Q$} or equivalently that
it is \emph{arithmetic} (see~\cite{Wright:field:of:def} for
the more general notion of the \textit{field of definition}
of a $\GLR$-invariant suborbifold).

The following weakened version of results of A. Wright (Theorem~1.9
in~\cite{Wright:cylinder:deformations}) provides efficient
way to recognize affine invariants suborbifold $\cL$
by comparing circumferences of horizontal cylinders
of any translation surface in $\cL$ having at least
one horizontal cylinder.

\begin{NNTheorem}[\cite{Wright:cylinder:deformations}]
Suppose that $S$ is a translation surface in a closed
connected $\GLR$-invariant
submanifold $\cL$ and suppose that $S$ has at least
one horizontal cylinder. ($S$ need not be horizontally
periodic.) Let $w_1,\dots,w_n$ be the set of circumferences
of horizontal cylinders on $S$.

If $n=1$, then $\cL$ is defined over $\Q$. If $n>1$ but
$w_j/w_1\in\Q$ for all $j=2,\dots,n$, then $\cL$ is defined
over $\Q$.
\end{NNTheorem}

In particular, any connected $\GLR$-invariant suborbifold
containing at least one square-tiled surface is defined
over $\Q$.

\begin{Definition}
By an \textit{invariant arithmetic orbifold}
we call a closed connected $\GLR$-invariant
suborbifold defined over $\Q$
in a stratum of Abelian differentials.
\end{Definition}
Note that a general invariant arithmetic orbifold
might have self-intersections, see Example~2.8
in~\cite{Mirzakhani:Wright:boundary}
and~Section 2.1 in~\cite{Lanneau:Nguyen:Wright}.

Any connected component of any stratum of Abelian
differentials is an invariant arithmetic orbifold. The
$\GLR$-orbit of any square-tiled surface is also an
invariant arithmetic orbifold. Another example, which is
of great importance for us, is provided by connected
components of strata of meromorphic quadratic differentials
with at most simple poles, see Section~\ref{ss:quad:as:invariant:orbifold}.

\subsection{Masur--Veech volume element}
\label{ss:Masur:Veech:volume:element}
By definition of the invariant arithmetic orbifold $\cL$,
the intersection of any linear subspace $L$ representing
$\cL$ in period coordinates $H^1(S,\Sigma ;\C)$ with the
lattice $H^1(S,\Sigma ;\Z\oplus i\Z)$ defines a lattice in
$L$. Taking the linear volume element in $L$ normalized in
such way that a fundamental domain of the resulting lattice
in $L$ is equal to one, we get canonically defined
\textit{Masur--Veech volume element} $d\!\Vol$ on $\cL$.

A volume element on a manifold canonically defines
a volume element on any level hypersurface of any smooth function.
Thus, the Masur--Veech volume element on an invariant
arithmetic orbifold $\cL$ of a stratum $\cH(\kappa)$
of Abelian differentials induces the canonical
volume element $d\!\Vol_1$ on the ``unit hyperboloid''
$\cL_1:=\cL\cap\cH_1(\kappa)$ of Abelian differentials of
area one in $\cL$.

Alternatively the volume element $d\!\Vol_1$ (also called the
\textit{Masur--Veech volume element}) on $\cL_1$ can be
defined as follows. Given any subset $U\subset\cL$ we
define the cones over $U$ as
\begin{align}
\label{eq:cone:R}
C_R U&:=\{(X,r\cdot\omega)\,:\,(X,\omega)\in U,\ 0<r\le R\}\,,
\\
\label{eq:cone:infty}
C_\infty U&:=\{(X,r\cdot\omega)\,:\,(X,\omega)\in U,\ 0<r\}\,.
\end{align}
Note that $\Area(X,r\cdot\omega)=r^2\Area(X,\omega)$ for
$r>0$: the flat area of $(X,r\cdot\omega)$ changes
quadratically in $r$. In particular, our definition implies
that
\begin{equation}
\label{eq:C:R:L:1}
C_R\cL_1 =\{(X,\omega)\in\cL\ :\ \Area(X,\omega)\le R^2\}\,.
\end{equation}

By definition, the \textit{Masur-Veech volume} $\Vol_1(V_1)$
of a subset $V_1\subset\cL_1$ of the unit hyperboloid
$\cL_1\subset\cL$ is defined as the Masur--Veech volume of
the ``unit cone'' over $V_1$ normalized by the dimensional
factor $2d$:
\begin{equation}
\label{eq:normalization:masur:veech}
\Vol_1 V_1 = \Vol_1(V_1):=2d\cdot \Vol(C_1 V_1)\,,
\quad\text{where }d=\dim_\C\cL\,.
\end{equation}
It is immediate to see that the two definitions of the
Masur-Veech volume element $d\!\Vol_1$ on an invariant
arithmetic suborbifold $\cL_1$ are equivalent.

By construction the Masur--Veech volume element thus
defined on the unit hyperboloid $\cL_1$ of any invariant
arithmetic suborbifold $\cL$ is $\SLR$-invariant. By
the fundamental result of Eskin, Mirzakhani,
Mohammadi (see page~\pageref{thm:EMM}), for any invariant
arithmetic suborbifold $\cL$, the
measure on $\cL_1$ represented by Masur--Veech volume
element on $\cL_1$ is a finite ergodic measure. In
particular, the total Masur--Veech volume $\Vol_1\cL_1$ of
any invariant arithmetic suborbifold $\cL$ is finite.
(Finiteness of Masur--Veech volumes of the unit
hyperboloids of strata of Abelian and quadratic
differentials was proved much earlier by
H.~Masur~\cite{Masur:82} and
W.~Veech~\cite{Veech:Gauss:measures}.)

Formally speaking, the Masur--Veech volume of any invariant
arithmetic suborbifold $\cL=C\cL_1$ is obviously infinite.
By the Masur--Veech volume of an invariant arithmetic orbifold,
we always mean $\Vol_1 \cL_1$, the volume of the unit hyperboloid.

Given an invariant arithmetic suborbifold $\cL$ in the
ambient stratum $\cH(\kappa)$ of Abelian differentials,
denote by $\cL_\Z$ the set of all square-tiled surfaces in
$\cL$ tiled with unit squares. Equivalently, $\cL_\Z$ is
the set of points of $\cL$ represented in period
coordinates of the ambient stratum $\cH(\kappa)$ by points
of the lattice $H^1(S,\Sigma ,\Z\oplus i\Z)$. It follows
from the definition~\eqref{eq:normalization:masur:veech} of
$\Vol_1\cL_1$ that
\begin{align}
\label{eq:Vol:L}
\Vol_1\cL_1:&=2d\cdot \Vol(C_1 \cL_1)
=2d\cdot\lim_{R\to+\infty}
\frac{\Card(C_R\cL_1\cap\cL_\Z)}{R^{2d}}
=\\
\notag
&= 2d\cdot
\lim_{N\to+\infty}
\frac{
\begin{pmatrix}
\text{number of square-tiled surfaces in $\cL$}
\\
\text{tiled with at most $N$ identical squares}
\end{pmatrix}
}{N^{d}}\,,
\end{align}
where $d=\dim_{\C}\cL$.

\subsection{Strata of quadratic differentials as invariant arithmetic orbifolds}
\label{ss:quad:as:invariant:orbifold}
In this Section we describe strata of meromorphic quadratic
differentials and their Masur--Veech volumes. We then
explain how this description fits the general setting of
invariant arithmetic orbifolds and their Masur--Veech
volumes discussed in
sections~\ref{ss:background:flat:surf}--\ref{ss:Masur:Veech:volume:element}.
The details of constructions are deferred to
Appendix~\ref{a:Lattices}.

Similarly to the case of strata $\cH(\kappa)$ in the moduli
spaces of Abelian differentials introduced in
Section~\ref{ss:background:flat:surf} one can define strata
$\cQ(\xi)$ in the moduli spaces of meromorphic quadratic differentials
with at most simple poles. The strata are parameterized by
generalized partitions $\xi = [-1^{\xi_{-1}} 0^{\xi_0}
1^{\xi_1} \ldots]$ satisfying the condition $\sum_{i \geq
-1} i \cdot \xi_i = 4g - 4$. An element in
$\cQ(\xi)$ is an equivalence class of tuples $(X, q,
\{P_{i,j}\}_{i \geq -1, 1 \leq j \leq \xi_i})$, where $X$ is
a smooth complex curve of genus $g$, $q$ is a meromorphic
quadratic differential on $X$ having simple poles at the
points $P_{-1,j}$ of $X$, marked points at the points
$P_{0,j}$ of $X$, and zeros of order $i$ at the points
$P_{i,j}$ of $X$ for $i \geq 1$, and $q$ is not
the square of an Abelian differential on $X$. The
strata $\cQ(\nu, -1^{|\nu|+4})$ considered in
the Introduction correspond to the case of genus $0$.

By convention, $(X, q, \{P_{i,j}\})$ (which will be
sometimes denoted by $(X, q)$ for brevity) defines an
\textit{integer point} in $\cQ(\xi)$ if and only if it can
be represented by a square-tiled surface tiled by identical
squares of size $\frac{1}{2} \times \frac{1}{2}$ such that
all points $P_{i,j}$ are located at the corners of the
squares. This convention defines a volume form on
$\cQ(\xi)$ and an induced volume form on $\cQ_1(\xi)$. The
resulting volume $\Vol_1\cQ_1(\xi)$ can be expressed as
follows:
\begin{align}
\label{eq:easy:normalization:Vol:cQxi}
\Vol_1\cQ_1(\xi)
= 2d\cdot
&\lim_{N\to+\infty}
\frac{
\begin{pmatrix}
\text{number of square-tiled surfaces in $\cQ(\xi)$}
\\
\text{tiled with at most $2N$ identical squares}
\end{pmatrix}
}{N^{d}}
\\
\notag
= 2d\cdot 2^d\cdot
&\lim_{N\to+\infty}
\frac{
\begin{pmatrix}
\text{number of square-tiled surfaces in $\cQ(\xi)$}
\\
\text{tiled with at most $N$ identical squares}
\end{pmatrix}
}{N^{d}}\,,
\end{align}
where $d=\dim_{\C}\cQ(\xi)$.
This convention for the normalization of the volume form
on $\cQ(\xi)$ follows the one in~\cite{AEZ:Dedicata, AEZ:genus:0}
and~\cite{Goujard:carea, Goujard:volumes}.

Given a quadratic differential $q$ on $X$ that is not the
square of an Abelian differential, there is a unique
(possibly ramified) connected double cover $p: \hat X \to
X$ such that the pull-back $p^* q$ is the square of an
Abelian differential $p^*q = \omega^2$ globally defined on
$\hat X$.
The branch points of $p$ are the poles and the
zeros of odd order of $q$. By convention, throughout
Section~\ref{s:Equidistribution} we always mark preimages
of simple poles of $q$ on $\hat X$, alternative convention
would be considered in the Appendix.

Note that for any $q$ in the same connected component
$\cQ(\xi)^{\mathit{comp}}$ of any stratum of meromorphic quadratic
differentials with at most simple poles, the resulting
Abelian differential $\omega$ belongs to the same connected
component $\cH(\kappa)^{\mathit{comp}}$ of the corresponding stratum
of Abelian differentials. However, the correspondence
$(X,q)\to(\hat X,\omega)$ does not define a map from
$\cQ(\xi)^{\mathit{comp}}$ to $\cH(\kappa)$ for there is an
ambiguity in labeling zeros of $\omega$ as soon as $q$ has
zeroes of even order. Considering all possible ways to
label the zeroes of $\omega=p^\ast q$ for all $(X,
q)\in\cQ(\xi)^{\mathit{comp}}$ we obtain a suborbifold
$\widehat{\cQ}(\xi)^{\mathit{comp}}\subset\cH(\kappa)^{\mathit{comp}}$ in
appropriate component of appropriate stratum of Abelian
differentials together with a covering map $P:
\widehat{\cQ}(\xi)^{\mathit{comp}} \to \cQ(\xi)^{\mathit{comp}}$. By
construction, the degree of the cover $P$ is $\deg(P)=
2^{\ell(\xi_{even})}$, where
$\ell(\xi_{even})=\xi_0+\xi_2+\xi_4+\dots$ is the total
number of marked points and of zeroes of even degrees.

The question of connectedness of the cover
$\widehat{\cQ}(\xi)^{\mathit{comp}}$ would be discussed in full
generality elsewhere. In the hypothetical situation when
$\widehat{\cQ}(\xi)^{\mathit{comp}}$ contains several connected
components, all components are pairwise isomorphic and all
components form invariant arithmetic orbifolds. From now
on, we choose any of these invariant  arithmetic orbifolds
$\cL$.

The following result relates geometry of a connected
component $\cQ(\xi)^{\mathit{comp}}$ of a stratum of meromorphic
quadratic differentials with at most simple poles and
geometry of the induced invariant arithmetic orbifold
$\cL$.

\begin{Proposition}
\label{prop:volumes:quad:and:double:cover}
Let $\cQ(\xi)^{\mathit{comp}}$ be a connected component of a stratum
of meromorphic quadratic differentials with at most simple
poles. Let $\cL$ be the invariant arithmetic orbifold
obtained by the canonical double cover (with marked
preimages of simple poles, if any). Let $P:\cL \to
\cQ(\xi)^{\mathit{comp}}$ be the natural cover. Then
\begin{enumerate}
\item
Any square-tiled surface $(\hat X, \omega)$
in $\cL$ always has even number of squares.
\item
The natural involution $\hat X\to \hat X$ interchanging
the sheets of the double cover $p:\hat X\to X$ sends
squares of the tiling to squares of the tiling and does
not map any square of the tiling to itself. In
particular, the image under the map $P$ of any
square-tiled surface $(\hat X,\omega)$ in $\cL$ tiled
with $2N$ identical squares is a square-tiled surface
$(X,q)$ in $\cQ(\xi)^{\mathit{comp}}$ tiled with $N$ identical
squares of the same size as the initial ones.
\item
The Masur-Veech volume element $d\!\Vol^{\cQ(\xi)}$ on
$\cQ(\xi)^{\mathit{comp}}$ corresponding to
normalization~\eqref{eq:easy:normalization:Vol:cQxi} and
the Masur-Veech volume element $d\!\Vol^{\cL}$ on $\cL$
corresponding to normalization~\eqref{eq:Vol:L} are
pointwise proportional:
\[
P^*(d\!\Vol^{\cQ(\xi)}) = 4^d \cdot d\!\Vol^{\cL}\,,
\]
where $d=\dim_{\C}\cQ(\xi)$. In particular,
\begin{equation}
\label{eq:Vol:cL:Vol:cQ}
\Vol_1^{\cL}(\cL)
=\frac{\deg(P)}{4^d}\cdot
\Vol_1^{\cQ(\xi)}(\cQ_1(\xi)^{\mathit{comp}})\,.
\end{equation}
\item
If furthermore $\cQ(\xi) = \cQ(\nu, -1^{|\nu|+4})$ is a
stratum in genus $0$, then the cover
$\widehat{\cQ}(\nu,-1^{|\nu|+4})=\cL$ is connected and
$\deg(P)= 2^{\ell(\xi_{even})}$, where
$\ell(\xi_{even})=\xi_0+\xi_2+\xi_4+\dots$ is the total
number of marked points and of zeroes of even degrees.

\end{enumerate}
\end{Proposition}

\begin{proof}
We postpone the proof of
Proposition~\ref{prop:volumes:quad:and:double:cover} to
Appendix~\ref{a:Lattices} with exception for the
computation of proportionality coefficients which is
performed below.

The factor $4^d$ relating the volume elements
$P^*(d\!\Vol^{\cQ(\xi)})$ and $d\!\Vol^{\cL}$ comes from
the fact that by convention, square-tiled surfaces in
$\cQ(\xi)$ are tiled with squares of size $\frac{1}{2}
\times \frac{1}{2}$ while square-tiled surfaces in $\cL$
are tiled with \textit{unit} squares.

The cover $P: \widehat{\cQ}(\xi)^{\mathit{comp}} \to
\cQ(\xi)^{\mathit{comp}}$ associates
$\deg(P)=2^{\ell(\xi_{\mathit{even}})}$ square-tiled
surfaces in $\widehat{\cQ}(\xi)^{\mathit{comp}}$ each tiled with
$2N$ identical squares to every square-tiled surface tiled
with $N$ identical squares in $\cQ(\xi)^{\mathit{comp}}$, and,
reciprocally, any square-tiled surface in
$\widehat{\cQ}(\xi)^{\mathit{comp}}$ projects to a square-tiled
surface in $\cQ(\xi)^{\mathit{comp}}$. Comparing
definitions~\eqref{eq:Vol:L}
and~\eqref{eq:easy:normalization:Vol:cQxi} we
get~\eqref{eq:Vol:cL:Vol:cQ}.
\end{proof}

The proposition above allows us to treat any connected components of
a stratum in the moduli space of meromorphic quadratic differentials
with at most simple poles as a particular case of an invariant arithmetic orbifold.

\subsection{Densities and uniform densities}
\label{ss:densities}
Given an invariant arithmetic orbifold $\cL$, we denote by
$\cL_\Z$ the set of square-tiled surfaces in $\cL$. When
the size of the square is not explicitly specified we
always assume that square-tiled surfaces are tiled with
\textit{unit} squares. In particular, $\cL_\Z$ is the set
of ``unit-square-tiled'' surfaces, where we impose all
standard assumptions on the tiling. Given a subset $\cD_\Z$
of $\cL_\Z$ and a subset $V$ of $\cL$ define the following
counting function
\begin{equation}
\label{eq:definition:counting:function}
\cN_{\cD_\Z}(V, N) := \Card \{V\cap\cD_\Z:\ \Area(S) \leq N\}
=\Card \{V\cap\cD_\Z\cap C_{\sqrt{N}}\cL_1\}
\,.
\end{equation}
Recall from Section~\ref{ss:background:flat:surf} that the
Masur--Veech volume $\Vol_1(\cL_1)$ is finite and is defined
as the leading term of the asymptotic of $\cN_{\cL_\Z}(\cL,N)$
as $N\to+\infty$ normalized by a dimensional factor as
\begin{equation}
\label{eq:N:L:Z:N}
\cN_{\cL_\Z}(\cL, N) = \frac{1}{2d}\cdot\Vol_1(\cL_1) \cdot N^d
+ o(N^d)\quad \text{as } N\to+\infty\,,
\end{equation}
where $d= \dim_\C \cL$ (see~\eqref{eq:Vol:L}).

\begin{Definition}
\label{def:density}
Let $\cL \subset \cH(\kappa)$ be an invariant arithmetic
orbifold. We say that a subset $\cD_\Z\subset\cL_\Z$
\textit{has a density}, or, equivalently, \textit{is a
density subset}, if the following limit exists:
\begin{equation}
\label{eq:def:delta}
\delta(\cD_\Z) := \lim_{N\to+\infty} \frac{\cN_{\cD_\Z}(\cL, N)}{\cN_{\cL_\Z}(\cL, N)}\,.
\end{equation}
The value $\delta(\cD_\Z)$ of the limit is called the
\textit{density} of the subset $\cD_\Z$.
\end{Definition}

As two extreme examples we have $\delta(\cL_\Z) = 1$ and
$\delta(\emptyset) = 0$. As a one more example, one can
consider a sublattice of $\cL_\Z$ in the period coordinates
(see Section~\ref{ss:background:flat:surf}). The density
of such sublattice is $\frac{1}{K}$ where $K$ is the index
in $\cL_\Z$ of the sublattice.

Note that a subset $\cD_\Z$ of $\cL_\Z$ has a density
$\delta(\cD_\Z)$ if and only if
\begin{equation}
\label{equiv:density}
\cN_{\cD_\Z}(\cL, N) \sim
\delta(\cD_\Z)\cdot\frac{\Vol_1(\cL_1)}{2d} \cdot N^d\,,
\end{equation}
where equivalence is understood in the sense of
equation~\eqref{eq:N:L:Z:N}.
It is convenient to introduce the following quantity
\begin{equation}
\label{eq:def:c}
c(\cD_\Z) := \delta(\cD_\Z)\,\cdot\,\Vol_1(\cL_1)\,.
\end{equation}

We now introduce finer counting functions related to
equidistribution. Recall that a \emph{cone} in $\cL$ is a
subset of $\cL$ preserved by dilation $S \mapsto \lambda
S$, see~\eqref{eq:cone:infty}.
\begin{Definition}
\label{def:uniform:density}
Let $\cL \subset \cH(\kappa)$ be an invariant arithmetic
orbifold. We say that a subset $\cD_\Z\subset\cL_\Z$
\textit{has a uniform density} if for any open cone $C$ such
that $\Vol_1 (\partial C \cap \cL_1) = 0$ we have
\begin{equation}
\label{eq:prop}
\lim_{N\to+\infty} \frac{\cN_{\cD_\Z}(C, N)}{\cN_{\cL_\Z}(C, N)}
\,=\,
\delta(\cD_\Z).
\end{equation}
\end{Definition}
Taking $C = \cL$ we see that a uniform density is a
density. The set $\cL_\Z$ has uniform density. Indeed,
it follows from zero measure boundary condition
$\Vol_1 (\partial C \cap \cL_1) = 0$ that the measure of
$C \cap \cL_1$ coincides with the leading coefficient
in the integer point count~\eqref{eq:prop}. See also
Remark~\ref{rk:jordan:measurable:riemann:integrable}.

The Proposition below suggests several equivalent
definitions of uniform density. In particular, it shows
that uniformity of density of $\cD_\Z$ is equivalent to
equidistribution of $\cD_\Z$ in $\cL$, see
property~\eqref{label:counting2:prop:uniform:density}, and
that uniformity of density of $\cD_\Z$ in $\cL$ is
equivalent to uniformity of the projection of $\cD_\Z$ to
$\cL_1$, see
property~\eqref{label:counting:prop:uniform:density:in:L1}.
\begin{Proposition}
\label{prop:uniform:density}
Let $\cL \subset \cH(\kappa)$ be an invariant arithmetic
orbifold in some stratum of Abelian differentials.
Let $\cD_\Z$ be a subset of the set $\cL_\Z$ of square-tiled
surfaces in $\cL$. Then the following assertions are
equivalent
\begin{enumerate}
\item
\label{label:counting1:prop:uniform:density}
$\cD_\Z$ has uniform density.
\item
\label{label:counting:prop:uniform:density:in:L1}
For any open set $V_1 \subset \cL_1$ with $\Vol_1(\partial V_1) = 0$
one has
\[
\lim_{\epsilon \to 0^+} \epsilon^{2d}
\Card \{S \in \cD_\Z:\ \epsilon S \in C_1 V_1\}
\,=\,
\frac{\delta(\cD_\Z)}{2d}\,\cdot\,\Vol_1(V_1)\,.
\]
\item
\label{label:integral1:prop:uniform:density}
For any bounded non-negative continuous function
$f: \cL_1 \to \R_+$ one has
\[
\lim_{N \to +\infty} \frac{1}{N^d}
\sum_{\substack{S \in \cD_\Z\\ \Area(S) \leq N}}
f \left( \frac{S}{\sqrt{\Area(S)}} \right)
\,=\,
\frac{\delta(\cD_\Z)}{2d}\,\cdot\,\int_{\cL_1} f\, d\!\Vol_1.
\]
\item
\label{label:counting2:prop:uniform:density}
For any open relatively compact set $V \subset \cL$ with
$\Vol(\partial V) = 0$ one has
\[
\lim_{\epsilon \to 0^+} \epsilon^{2d} \Card \{S \in \cD_\Z:\ \epsilon S \in V\}
\,=\,
\delta(\cD_\Z)\,\cdot\,\Vol(V).
\]
\item
\label{label:integral2:prop:uniform:density}
For any compactly supported non-negative continuous function $f: \cL \to \R_+$
one has
\[
\lim_{\epsilon \to 0^+} \epsilon^{2d} \sum_{S \in \cD_\Z} f(\epsilon S)
\,=\,
\delta(\cD_\Z)\, \int_\cL f\, d\!\Vol\,.
\]
\end{enumerate}
\end{Proposition}

\begin{Remark}
\label{rk:jordan:measurable:riemann:integrable}
One of the equivalent ways to define the Riemann integral consists in
rescaling the mesh of integral points to approximate measure of the sets or integrals of functions. As a consequence, the
above proposition remains true if one replaces a continuous function by a
Riemann integrable function in assertions~\ref{label:integral1:prop:uniform:density}
and~\ref{label:integral2:prop:uniform:density}. Similarly, open sets with zero
measure boundaries can be replaced by Jordan measurable sets in
assertions~\ref{label:counting:prop:uniform:density:in:L1}
and~\ref{label:counting2:prop:uniform:density}. Recall that a set $V$ is called Jordan
measurable if it is measurable and $\Vol(\mathring{V}) = \Vol(\overline{V})$.
It is a standard fact in integration that a set is Jordan measurable if and only
if the characteristic function of this set is Riemann integrable.
\end{Remark}

\begin{proof}
Any open cone $C\subset\cL$ can be realized as
$C=C_\infty V_1$ for $V_1=C\cap\cL_1$, see
definition~\eqref{eq:cone:infty}, and conversely,
any open set $V_1\subset\cL_1$ is the intersection
of the open cone $C_\infty V_1$ with $\cL_1$.
Thus, similarly
to~\eqref{equiv:density}, a subset $\cD_\Z$ of $\cL_\Z$ has
a \textit{uniform} density $\delta(\cD_\Z)$ if and only if for every
open set $V_1\subset \cL_1$ one has
$$
\cN_{\cD_\Z}(C_\infty V_1, N) \sim
\delta(\cD_\Z)\cdot\frac{\Vol_1(V_1)}{2d} \cdot N^d\,.
$$
Having a subset $V \subset \cL$ define the quantity
$\widetilde{\cN}_{\cD_\Z}(V, \epsilon) := \Card \{S \in \cD_\Z:\ \epsilon S \in V\}$.
We can rewrite the left-hand side of the above
asymptotic relation as
\begin{multline*}
\cN_{\cD_\Z}(C_\infty V_1, N) =
\Card\{S\in \cD_\Z\cap C_{\sqrt{N}} V_1\}
=\\=
\Card\left\{S\in \frac{1}{\sqrt{N}} \cD_\Z \cap C_1 V_1\right\}=
\widetilde\cN_{\cD_\Z}\left(C_1 V_1, \frac{1}{\sqrt N}\right)\,.
\end{multline*}
We conclude that the assertion~\eqref{label:counting1:prop:uniform:density}
is equivalent to the following assertion:
$$
\lim_{N\to+\infty}
\left(\tfrac{1}{\sqrt{N}}\right)^{2d}
\Card \{S \in \cD_\Z:\ \left(\tfrac{1}{\sqrt{N}}\right)\cdot S \in C_1 V_1\}
\,=\,
\frac{\delta(\cD_\Z)}{2d}\,\cdot\,\Vol_1(V_1)\,.
$$
It is clear that existence of the above limit for the
discrete parameter $\epsilon=\frac{1}{N}$ with $N\in\N$ and
for the continuous parameter $\epsilon\in\R_+$ are
equivalent. This proves that the first two assertions are
equivalent,
\eqref{label:counting1:prop:uniform:density}$\Leftrightarrow$\eqref{label:counting:prop:uniform:density:in:L1}.

By definition~\eqref{eq:normalization:masur:veech} of the
Masur--Veech volume we have
$\frac{1}{2d}\Vol_1(V_1)=\Vol(C_1 V_1)$. This implies that
assertion~\eqref{label:counting:prop:uniform:density:in:L1} is a
particular case of
assertion~\eqref{label:counting2:prop:uniform:density} when
$V= C_1 V_1$, so
\eqref{label:counting:prop:uniform:density:in:L1}
$\Leftarrow$\eqref{label:counting2:prop:uniform:density}.

\begin{figure}[htb]

\includegraphics{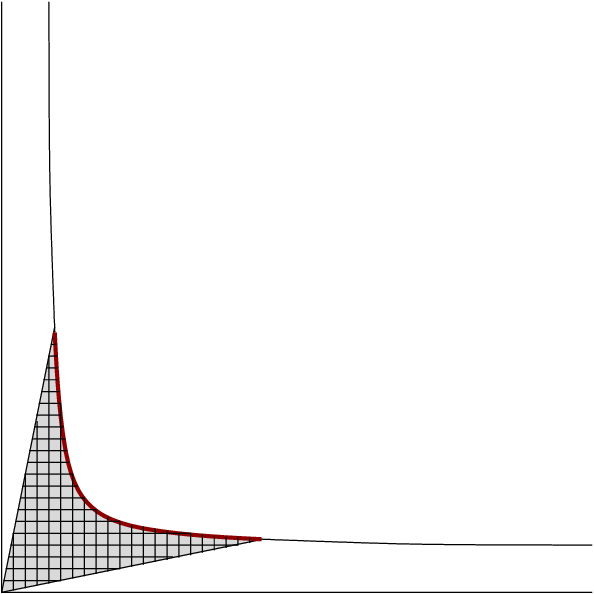}

\includegraphics{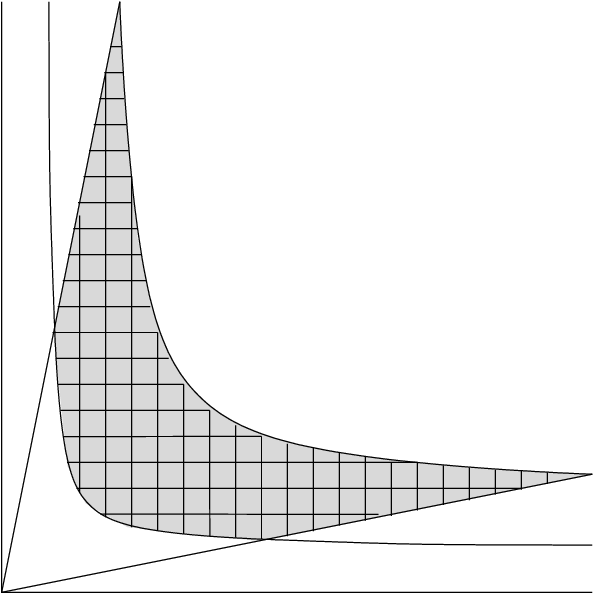}

\vspace{70pt}

\caption{
\label{fig:trapezoid}
Cone based on the ``unit hyperboloid'' on the left
and part of the cone confined between two ``hyperboloids''
on the right.
}
\end{figure}

We now prove the converse implication
\eqref{label:counting:prop:uniform:density:in:L1}
$\Rightarrow$\eqref{label:counting2:prop:uniform:density}.
Assertion~\eqref{label:counting:prop:uniform:density:in:L1}
implies
assertion~\eqref{label:counting2:prop:uniform:density} when
$V$ has the form of a cone $V = C_1 V_1$ based on an open
subset of the unit hyperboloid $\cL_1$ (see the left
picture in Figure~\ref{fig:trapezoid}). By homogeneity,
this implies
assertion~\eqref{label:counting2:prop:uniform:density} for
any cone $C_r V_1$ with any $r>0$. Hence it is also valid
for any complement $C_R V_1- C_r V_1$ for any $R>r>0$ (see
the right picture in Figure~\ref{fig:trapezoid}). Now, for
any open relatively compact set $V$ in $\cL$ and any
$\upsilon>0$ one can find a finite
disjoint collection of trapezoids $V^{(i)} = C_{R_i}
V^{(i)}_1 - C_r V^{(i)}_1$ contained in $V$ and such that
the difference $V^{\mathit{dif}}$ between the union of
trapezoids and $V$ has Masur--Veech measure less than
$\upsilon$. As $\upsilon$ can be chosen arbitrarily small
and as
\begin{multline*}
0\le
\limsup_{\epsilon \to 0^+}
\epsilon^{2d} \Card \{S \in \cD_\Z:\ \epsilon S \in V\}
\,\le\\ \le\,
\lim_{\epsilon \to 0^+}
\epsilon^{2d} \Card \{S \in \cL_\Z:\ \epsilon S \in V\}
=
\Vol(V^{\mathit{dif}})\le\upsilon
\end{multline*}
we conclude that
assertion~\eqref{label:counting:prop:uniform:density:in:L1}
implies
assertion~\eqref{label:counting2:prop:uniform:density}.
We have proved equivalence
\eqref{label:counting:prop:uniform:density:in:L1}
$\Leftrightarrow$\eqref{label:counting2:prop:uniform:density}.

Equivalence of
assertions~\eqref{label:counting:prop:uniform:density:in:L1}
and~\eqref{label:integral1:prop:uniform:density}
(respectively of
assertions~\eqref{label:counting2:prop:uniform:density}
and~\eqref{label:integral2:prop:uniform:density}) follows
directly from the duality of measures seen as linear forms
on continuous functions.
\end{proof}

\subsection{Horocyclic invariance}
\label{ss:uniform:density:via:horocyclic:invariance}
Consider the following horocyclic subgroups of the group $\SLZ$:
$$
\UhZ=\left\{\begin{pmatrix} 1&n\\0&1\end{pmatrix}:\ n \in \Z\right\}
\quad \text{and} \quad
\UvZ=\left\{\begin{pmatrix} 1&0\\n&1\end{pmatrix}:\ n \in \Z\right\}.
$$
By definition the $\SLZ$ action preserve the set of
square-tiled surfaces. The aim of this Section is to prove
that for $\UhZ$ or $\UvZ$-invariant sets, a density is
always uniform (see Definitions~\ref{def:density}
and~\ref{def:uniform:density}).

\begin{Theorem}
\label{th:equidistribution}
Let $\cL$ be an invariant arithmetic orbifold and let
$\cD_\Z$ be a density subset of square-tiled surfaces in $\cL$.
If $\cD_\Z$ is invariant under at least one of
$\UhZ$ or $\UvZ$, then $\cD_\Z$ has uniform density.
\end{Theorem}

We start the proof of Theorem~\ref{th:equidistribution} with the
following preparatory Lemma.

\begin{Lemma}
\label{lm:Moore}
Any finite $\SLR$-invariant ergodic measure $\nu_1$ on
any unit hyperboloid of a stratum of Abelian differentials is
ergodic with respect to the actions of the discrete parabolic subgroups
$\UhZ$ and $\UvZ$.
\end{Lemma}

\begin{proof}
Let $\operatorname{G}$  be  a simple  Lie  group,  $\operatorname{H}$
be a closed non-compact subgroup of $\operatorname{G}$ and let
$\operatorname{G}$-action be ergodic with respect to a finite
invariant measure. By a particular case of Moore's  Ergodicity
theorem (Theorem~2.2.15 in~\cite{Zimmer}) the
$\operatorname{H}$-action is also ergodic.

In our case the simple Lie group is $\SLR$ and the closed non-compact
subgroup $\operatorname{H}$ is $\UhZ$.
\end{proof}

\begin{Remark}
Note that in the general statement of Moore's Ergodic Theorem, the
group $\operatorname{G}$ is a finite product of simple Lie groups
with finite center, and the ergodic $\operatorname{G}$-action is
supposed to be \textit{irreducible} (see Theorem~2.2.15 and
Definition~2.2.11 in~\cite{Zimmer}). However, for a \textit{simple}
Lie group $\operatorname{G}$ the requirement of irreducibility of the
action is satisfied automatically; see the remark after
Definition~2.2.11 in~\cite{Zimmer}.
\end{Remark}

\begin{proof}[Proof of Theorem~\ref{th:equidistribution}]
Let $\cL_1$ be the intersection of $\cL$ with the unit
hyperboloid. The measure corresponding to the Masur--Veech
volume element $d\!\Vol_1$ on $\cL_1$ is ergodic with
respect to the $\SLR$-action.

Now, the subset $\cD_\Z$ of $\cL_\Z$ defines a
sequence of discrete measures $\mu^{(N,\cD_\Z)}$ on  $\cL_1$, where $N\in\N$.
Namely, we put Dirac masses to all points represented by square-tiled
surfaces tiled with at most $N$ unit squares which belong to
$\cD_\Z$. Then we project these points from $\cL$ to $\cL_1$ by the
natural projection, and normalize the resulting measure by
$2d \cdot N^{-d}$, where $d=\dim_{\C}\cL$. We will show that for
any bounded non-negative continuous  function $f: \cL_1 \to \R_+$
we have
\begin{equation} \label{eq:f:lim:equi}
\lim_{N \to \infty} \int_{\cL_1} f\, d \mu^{(N,\cD_\Z)}
= \delta(\cD_{\Z}) \cdot \int_{\cL_1} f\, d\! \Vol_1\,.
\end{equation}

Taking \textit{all} square-tiled surfaces in $\cL_\Z$, and
not only those which belong to the subset  $\cD_\Z$, we get
a sequence of measures which we denote by
$\mu^{(N,\cL_\Z)}$ and which weakly converges to our
canonical invariant Masur--Veech measure $\Vol_1$ on
$\cL_1$, see~\eqref{eq:normalization:masur:veech}. Recall
that $\cL_1$ is not compact, so by ``weak convergence'' we
mean that for any bounded non-negative continuous  function
$f: \cL_1 \to \R_+$ we have
$$
\lim_{N \to \infty} \int_{\cL_1} f\, d \mu^{(N,\cL_\Z)}
= \int_{\cL_1} f\, d\! \Vol_1\,.
$$

By definition, for any subset $\cD_\Z$ we have
\begin{equation}
\label{eq:domination}
\mu^{(N,\cD_\Z)} \le \mu^{(N,\cL_\Z)}
\end{equation}
for we take
\textit{only part} of square-tiled surfaces  of area at
most $N$ to  define $\mu^{(N,\cD_\Z)}$ while we take
\textit{all} square-tiled  surface of area at most $N$ to
define $\mu^{(N,\cL_\Z)}$. Since the normalization factor
$2d\cdot N^{-d}$ is the same in both cases, we get the
desired inequality. This implies, in particular, that
for any Jordan measurable subset $K\subset\cL_1$ with
compact closure one has
$$
\limsup_{N\to+\infty} \mu^{(N,\cD_\Z)}(K)\le
\lim_{N\to+\infty} \mu^{(N,\cL_\Z)}(K)=\Vol_1(K)\,.
$$
Therefore any subsequence of $\{ \mu^{(N,\cD_\Z)} \}_{N
\geq 0}$ contains a converging subsequence.

From now on, we fix a measure $\mu_J$ obtained as the weak
limit of some subsequence $\{\mu^{(N_k, \cD_\Z)}\}_{k \geq
0}$. Domination~\eqref{eq:domination} and the fact that
$\mu^{(N,\cL_\Z)}$ converges to the Masur--Veech measure
$\Vol_1$, implies that $\mu_J$ is absolutely continuous
with respect to $\Vol_1$. Moreover, the $\UhZ$-invariance
of $\cD_\Z$ implies that all measures $\mu^{(N_k,\cD_\Z)}$
are $\UhZ$-invariant. Hence, the weak limit $\mu_J$ is also
$\UhZ$-invariant.

By Lemma~\ref{lm:Moore} the Masur--Veech measure $\Vol_1$
is ergodic with respect to the action of $\UhZ$.
Ergodicity of $\Vol_1$ with respect to the action of
$\UhZ$,
invariance of $\mu_J$ under the action of
$\UhZ$, and the fact that $\mu_J$ is absolutely continuous
with respect to $\Vol_1$  all
together imply that the two measures are proportional:
$$
\mu_{J} = k_J\cdot \Vol_1\,.
$$
The coefficient of proportionality $k_J$ can be obtained by
integrating the constant function 1. By definition~\eqref{eq:def:delta}
of density $\delta(\cD_\Z)$ we have
\[
k_J = \lim_{k \to \infty} \frac{\mu^{(N_k,\cD_\Z)}(\cL_1)}{\mu^{(N_k,\cL_\Z)}(\cL_1)}
= \lim_{k \to \infty} \frac{\cN_{\cD_\Z}(\cL, N_k)}{\cN_{\cL_\Z}(\cL, N_k)}
=\delta(\cD_\Z)\,.
\]
for any weak limit $\mu_J$. This
implies~\eqref{eq:f:lim:equi}, which correspond to
assertion~\eqref{label:integral1:prop:uniform:density} in
the list of equivalent definitions of uniform density of
$\cD_\Z$ given in Proposition~\ref{prop:uniform:density}.
This concludes the proof of
Theorem~\ref{th:equidistribution}.
\end{proof}

\begin{Remark}
\label{rm:coprime:in:Z:plus:Z}
Our proof is inspired by the proof of Theorem~6.4
in~\cite{Mirzakhani:growth:of:simple:geodesics} where (in
the language of the current paper) Mirzakhani proves that
ergodicity of the action of the mapping class group
$\operatorname{Mod}_{g,n}$ with respect to the Thurston
measure on the space of measured laminations $\cM\cL_{g,n}$
implies uniform density of the
$\operatorname{Mod}_{g,n}$-orbit of any integral
multicurve with respect to the set of all integral multicurves
$\cM\cL_{g,n}^\Z$.
\end{Remark}

\begin{Remark}
\label{rk:coprime}
Theorem~\ref{th:equidistribution} can be easily illustrated
in the context of the usual Lebesgue measure on $\R^2$.
Let us consider the subset $\cD_{\Z,1}$ of those points of
$\Z^2$ which have coprime coordinates. For any $R>0$ one can place a
ball of radius $R$ in $\R^2$ in such a way that the ball
would not contain a single point of $\cD_{\Z,1}$. In this
aspect $\cD_{\Z,1}$ does not resemble a sublattice of
$\Z^2$. Nevertheless, $\cD_{\Z,1}$ is
$\SLZ$-invariant, which implies that $\cD_{\Z,1}$ has
uniform density $\delta(\cD_{\Z,1})$. The value of
$\delta(\cD_{\Z,1})$ can be evaluated explicitly as
$\delta(\cD_{\Z,1})=\frac{6}{\pi^2}$. In particular, one
can compute areas of Jordan measurable sets $V$ in $\R^2$
by counting the number of points of $\epsilon\cD_{\Z,1}$
which get to $V$ and letting the mesh $\epsilon$ of the
grid tend to zero:
$$
\Area(V)=\frac{\pi^2}{6}\cdot\lim_{\epsilon\to+0}
\epsilon^2\cdot\Card\{V\cap \epsilon\cD_{\Z,1}\}\,.
$$

For each $k\in\N$ consider the analogous subset
$\cD_{\Z,k}\subset\Z^2$ of points with integer
coordinates $(x,y)$ such that $\gcd(x,y)=k$.
By definition, $\cD_{\Z,k}$ is obtained from $\cD_{\Z,1}$ by
homothety with dilatation coefficient $k$. Thus, any subset
$\cD_{\Z,k}$ also has uniform density, and
$\delta(\cD_{\Z,k})=\frac{1}{k^2}\cdot \delta(\cD_{\Z,1})$.
We have
$$
\Z^2=\bigsqcup_{k\in\N} \cD_{\Z,k}\,,
$$
which leads to
$$
1=\delta(\Z^2)=\sum_{k=1}^{+\infty} \delta(\cD_{\Z,k})=\sum_{k=1}^{+\infty} \frac{1}{k^2}\cdot\delta(\cD_{\Z,1})
=\zeta(2)\cdot\delta(\cD_{\Z,1})\,.
$$
\end{Remark}

\subsection{Invariance with respect to $\Re$- and $\Im$-foliations}
\label{ss:Invariance:along:Re:and:Im:foliations}

Given two subsets $\cD_\Z$ and $\cD'_\Z$ of square-tiled
surfaces that admit (uniform) densities, it is in general
false that $\delta(\cD_\Z \cap \cD'_\Z) =
\delta(\cD_\Z)\cdot \delta(\cD'_\Z)$. We introduce in this
Section $\Re$- and $\Im$-invariance that provides
sufficient conditions for this equality to hold.

A stratum of Abelian differentials of complex dimension $d$
is endowed with a pair of transverse foliations
of real dimension $d$ induced from the canonical
direct sum decomposition in period coordinates
\begin{equation}
\label{eq:re:plus:im}
H^1(S,\Sigma;\C) = H^1(S,\Sigma;\R) \oplus H^1(S,\Sigma;i\R)\,.
\end{equation}
In particular, in a neighborhood of any point $(X,\omega)$
of the stratum one has canonical direct product structure
in period coordinates. Locally, leafs of the
$\Im$-foliation are preimages of points under projection to
the first summand, and leafs of the $\Re$-foliation are
projections to the second summand. In other words, pairs
$(X,\omega)$ in a leaf of the $\Im$-foliation (respectively
of the $\Re$-foliation) share the cohomology class $[\Re
\omega]\in H^1(S,\Sigma;\R)$ (respectively
$[\Im \omega]\in H^1(S,\Sigma;\R)$).

By the results of Eskin--Mirzakhani--Mohammadi~\cite{Eskin:Mirzakhani},
\cite{Eskin:Mirzakhani:Mohammadi} (see the precise statement
on page~\pageref{thm:EMM}), the analogous
decomposition holds for any $\GLR$-invariant suborbifold
$\cL$. If $\cL$ is locally represented in period
coordinates as a finite union of several linear subspaces,
every such subspace is foliated by $\Re$ and
$\Im$-foliations. We formalize this trivial corollary of
highly nontrivial results of Eskin--Mirzakhani--Mohammadi
as a separate assertion.
\begin{Proposition}
\label{eq:Re:leaf:L} Let $\cL$ be an invariant arithmetic
suborbifold in some stratum $\cH(\kappa)$ of Abelian
differentials. Then the $\Re$-leaf in $\cL$ passing through
a point $(X,\omega)$ of $\cL$ coincides with the connected
component of the intersection of the $\Re$-leaf in the
ambient stratum $\cH(\kappa)$ passing through $(X,\omega)$
with $\cL$.

In period coordinates $H^1(S,\Sigma;\C)$ in
the neighborhood of $(X,\omega)$, the $\Re$-leaf in $\cL$
passing through a point $(X,\omega)$ is represented by a
finite union of linear subspaces of real dimension
$d=\dim_{\C}\cL$.
\end{Proposition}

\begin{Definition}
\label{def:Im:Re:invariant}
We say that a subset $\cD_\Z\subset\cL_\Z$ of square-tiled
surfaces in an invariant arithmetic suborbifold $\cL$ is
$\Re$-\textit{invariant} (respectively
$\Im$-\textit{invariant}) if for any point $S$ in
$\cD_\Z$ all points of $\cL_\Z$ located in the leaf of the
$\Re$-foliation (respectively $\Im$-foliation) in $\cL$
passing through $S$ also belong to $\cD_\Z$.
\end{Definition}

\begin{Remark}
In a very similar context in homogeneous dynamics the
analogous properties are called \textit{stable} and
\textit{unstable horocyclic invariance}.
\end{Remark}

Note that the unipotent subgroups
$$
\UhR = \left\{\begin{pmatrix} 1&t\\0&1\end{pmatrix}:\, t \in \R\right\}
\quad \text{and} \quad
\UvR = \left\{\begin{pmatrix} 1&0\\t&1\end{pmatrix}:\, t \in \R\right\},
$$
act along the leaves of the $\Re$-foliation and $\Im$-foliation
respectively. Thus any $\Re$-invariant (respectively $\Im$-invariant)
subset $\cD_\Z$ of square-tiled surface is automatically $\UhZ$-invariant
(respectively $\UvZ$-invariant). The converse is obviously not
true: the $\UhZ$-orbit of any square-tiled surface of genus $g \geq 2$
is $\UhZ$-invariant but usually not $\Re$-invariant.

\begin{Theorem}
\label{th:uncorrelated}
Let $\cD^h_\Z,\cD^v_\Z$ be subsets of square-tiled surfaces
in an invariant arithmetic orbifold $\cL$ that are respectively
$\Re$-invariant and $\Im$-invariant. Then $\cD^h_\Z$,
$\cD^v_\Z$ and $\cD^h_\Z \cap \cD^v_\Z$ have uniform
densities and the following equality holds
\[
\delta(\cD^h_\Z \cap \cD^v_\Z)
\,=\,
\delta(\cD^h_\Z)\ \cdot\ \delta(\cD^v_\Z).
\]
\end{Theorem}
\begin{proof}
The uniformity of the density of $\cD^h_\Z$ and of
$\cD^v_\Z$ follows from Theorem~\ref{th:equidistribution}.

Following the proof of Proposition~\ref{prop:uniform:density}
define the following quantity:
\[
\widetilde{\cN}_{\cD_\Z}(V, \epsilon)
\,=\,
\Card \{S \in \cD_\Z: \epsilon S \in V\}.
\]

We prove now
assertion~\eqref{label:counting2:prop:uniform:density} from
Proposition~\ref{prop:uniform:density} for the set
$\cD^h_\Z \cap \cD^v_\Z$. By
Proposition~\ref{prop:uniform:density} this assertion is
equivalent to uniformity of the density of $\cD^h_\Z \cap
\cD^v_\Z$. It is sufficient to prove
assertion~\eqref{label:counting2:prop:uniform:density} for
sets $V$ of the form $I \times J$ where $I$ and $J$ are
respectively Jordan measurable in
the summands~\eqref{eq:re:plus:im} in the natural product
structure of $\cL$ provided by $\Re$ and $\Im$ foliations
(see Remark~\ref{rk:jordan:measurable:riemann:integrable}
for the definition of Jordan measurable).
Indeed, any Jordan measurable set in $\cL$ can be approximated
from below and above by a finite union of products of Jordan
measurable sets of the form $I \times J$ up to an arbitrary
small difference in measures.
We also assume that any $V$ is located in a single
coordinate chart in period coordinates so that the products
$p_{\Re}$ and $p_{\Im}$ to the first and to the second term
of~\eqref{eq:re:plus:im} are well-defined, $I=p_{\Re}(V)$;
$J=p_{\Im}(V)$ and $V=I\times J$.

$\Re$- and $\Im$-invariance of $\cD^h_\Z$ and $\cD^v_\Z$
respectively, implies that the sets $\epsilon\cD^h_\Z$,
$\epsilon\cD^v_\Z$  and $\epsilon\cdot(\cD^h_\Z\cap\cD^v_\Z)$
locally have the following product structures
\begin{align*}
\epsilon\cD_\Z^h\cap V
&=
(I \cap p_{\Re}(\epsilon \cL_\Z))\times
(J \cap p_{\Im}(\epsilon \cD^h_\Z))\,,
\\
\epsilon\cD_\Z^v\cap V
&=
(I \cap p_{\Re}(\epsilon \cD^v_\Z))\times
(J \cap p_{\Im}(\epsilon \cL_\Z))\,,
\\
\epsilon\cdot(\cD_\Z^h\cap\cD_\Z^v)\cap V
&=
(I \cap p_{\Re}(\epsilon \cD^v_\Z))\times
(J \cap p_{\Im}(\epsilon \cD^h_\Z))\,,
\end{align*}
where
\begin{align*}
I \cap p_{\Re}(\epsilon \cL_\Z)
&= I \cap H^1(S,\Sigma;\epsilon\Z)\,,
\\
J \cap p_{\Im}(\epsilon \cL_\Z)
&= J \cap H^1(S,\Sigma;i\epsilon\Z)\,.
\end{align*}
This implies the following
relations for the counting functions:
\begin{align*}
\widetilde{\cN}_{\cL_\Z}(V, \epsilon)
&=
\Card(I \cap p_{\Re}(\epsilon \cL_\Z))\cdot
\Card(J \cap p_{\Im}(\epsilon \cL_\Z))\,,
\\
\widetilde{\cN}_{\cD_\Z^h}(V, \epsilon)
&=
\Card(I \cap p_{\Re}(\epsilon \cL_\Z))\cdot
\Card(J \cap p_{\Im}(\epsilon \cD^h_\Z))\,,
\\
\widetilde{\cN}_{\cD_\Z^v}(V, \epsilon)
&=
\Card(I \cap p_{\Re}(\epsilon \cD^v_\Z))\cdot
\Card(J \cap p_{\Im}(\epsilon \cL_\Z))\,,
\\
\widetilde\cN_{\cD^h \cap \cD^v}(V,\epsilon)
&=
\Card(I \cap p_{\Re}(\epsilon \cD^v_\Z))\cdot
\Card(J \cap p_{\Im}(\epsilon \cD^h_\Z))\,.
\end{align*}
Thus, for any set $V = I \times J$ as above we have the
equality
$$
\frac{\widetilde\cN_{\cD^h_\Z \cap \cD^v_\Z}(V,\epsilon)}
{\widetilde{\cN}_{\cL_\Z}(V, \epsilon)}
\,=\,
\frac{\widetilde{\cN}_{\cD^h_\Z}(V, \epsilon)}
{\widetilde{\cN}_{\cL_\Z}(V, \epsilon)}
\ \cdot\
\frac{\widetilde{\cN}_{\cD^v_\Z}(V,\epsilon)}
{\widetilde{\cN}_{\cL_\Z}(V, \epsilon)}\,.
$$
The definition of the Masur--Veech volume and the fact that
$\cD^h_\Z$ and $\cD^v_\Z$ have uniform densities imply the
existence of the limits:
\begin{align*}
\lim_{\epsilon\to+\infty}
\frac{\widetilde{\cN}_{\cD^h_\Z}(V, \epsilon)}
{\widetilde{\cN}_{\cL_\Z}(V, \epsilon)}
&=\delta(\cD^h_\Z)\,,
\\
\lim_{\epsilon\to+\infty}
\frac{\widetilde{\cN}_{\cD^v_\Z}(V, \epsilon)}
{\widetilde{\cN}_{\cL_\Z}(V, \epsilon)}
&=\delta(\cD^v_\Z)\,.
\end{align*}
Hence,
$$
\lim_{\epsilon\to+\infty}
\frac{\widetilde\cN_{\cD^h_\Z \cap \cD^v_\Z}(V,\epsilon)}
{\widetilde{\cN}_{\cL_\Z}(V, \epsilon)}
\,=\,
\delta(\cD^h_\Z)\cdot\delta(\cD^v_\Z)\,,
$$
which is equivalent to analogous relation written in
the form of the
assertion~\eqref{label:counting2:prop:uniform:density} from
Proposition~\ref{prop:uniform:density}.
\end{proof}

\subsection{Separatrix diagrams and critical graphs}
\label{s:sep:diagrams}
In this Section we recall some basic facts concerning
combinatorial geometry of Jenkins--Strebel differentials.
Assume that all leaves of the horizontal foliation of an
Abelian or quadratic differential are
either closed or connect critical points (a leaf joining
two critical points and having no critical points in its
interior is called a \textit{saddle connection} or a
\textit{separatrix}). An Abelian or
quadratic differential having this property is called a
\textit{Jenkins-Strebel} differential, see~\cite{Strebel}.
For example, square-tiled surfaces provide particular cases
of Jenkins-Strebel differentials.

\begin{figure}[htb]
\includegraphics{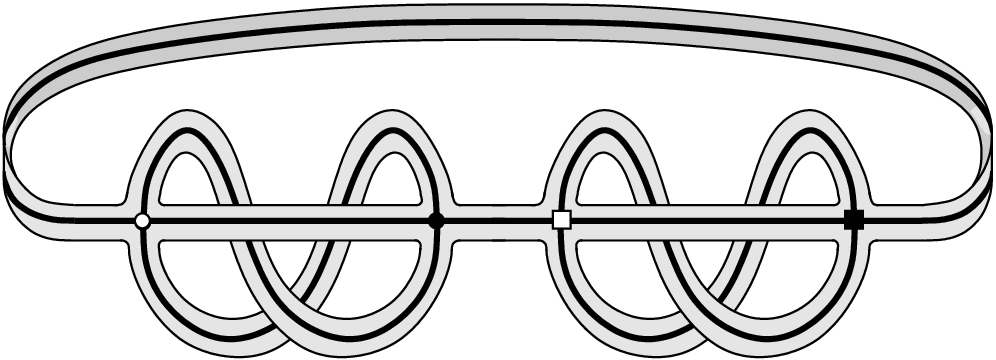}
\includegraphics{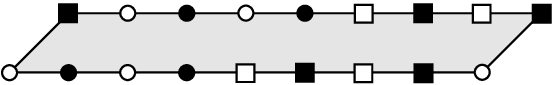}
\begin{picture}(0,0)(35,-89) 
\put(-88,-138){$\lambda_1$}
\put(-59,-112){$\lambda_2$}
\put(-32,-138){$\lambda_3$}
\put(-1,-112){$\lambda_4$}
\put(30,-138){$\lambda_5$}
\put(61,-112){$\lambda_6$}
\put(90,-138){$\lambda_7$}
\put(120,-112){$\lambda_8$}
\put(146,-138){$\lambda_1$}
\end{picture}

\begin{picture}(0,0)(3,-13) 
\begin{picture}(0,0)(0,5)
\put(-88,-112){$\lambda_1$}
\put(-60,-112){$\lambda_2$}
\put(-32,-112){$\lambda_3$}
\put(-4,-112){$\lambda_4$}
\put(24,-112){$\lambda_5$}
\put(53,-112){$\lambda_6$}
\put(82,-112){$\lambda_7$}
\put(110,-112){$\lambda_8$}
\end{picture}
\begin{picture}(0,0)(28,5)
\put(-89,-155){$\lambda_4$}
\put(-60,-155){$\lambda_3$}
\put(-32,-155){$\lambda_2$}
\put(-4,-155){$\lambda_5$}
\put(24,-155){$\lambda_8$}
\put(53,-155){$\lambda_7$}
\put(81,-155){$\lambda_6$}
\put(110,-155){$\lambda_1$}
\end{picture}
\begin{picture}(0,0)(3,5)
\put(-127,-132){$\gamma_1$}
\put(117,-134){$\gamma_1$}
\end{picture}
\end{picture}
\vspace{147bp}
\caption{
\label{fig:Jenkins:Strebel}
A  Jenkins--Strebel Abelian
differential with a single cylinder
represented
as a ribbon graph (on top) and as
a parallelogram on the bottom.
}
\end{figure}
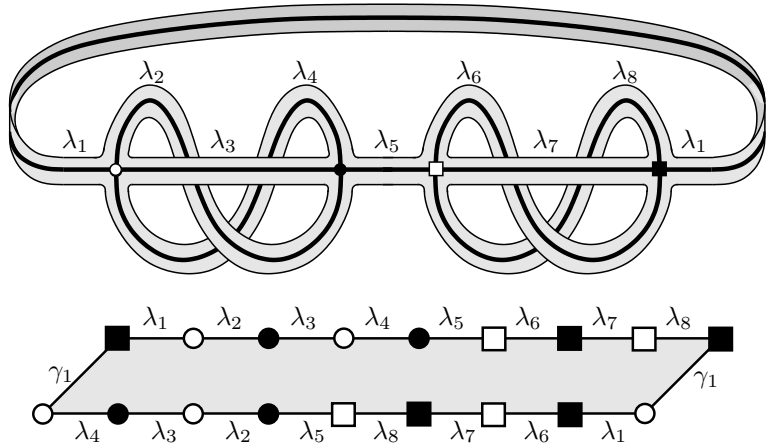

Following~\cite{Kontsevich:Zorich} we associate to each
Jenkins--Strebel differential a combinatorial data called a
\textit{separatrix diagram} (or \textit{critical graph} in
terminology of~\cite{Douady:Hubbard}.) A separatrix diagram
$\Gamma$ encodes the combinatorial geometry of the
completely periodic horizontal foliation. It is a ribbon
graph composed from a tubular neighborhood of the union of
all horizontal saddle connections. This ribbon graph is
endowed with the partition of boundary components into
pairs, where two components are paired when they are joined
by a cylinder filled with regular periodic horizontal
leafs.

\begin{Lemma}
\label{lm:number:of:edges:of:sep:diagram}
A separatrix diagram representing a Jenkins--Strebel Abelian differential
in a stratum of complex dimension $d$ has
exactly $d-1$ edge (i.e. $d-1$ horizontal saddle
connections).

A separatrix diagram representing a meromorphic
Jenkins--Strebel quadratic  differential with at most
simple poles in a stratum of complex dimension
$d$ has exactly $d$ edges (i.e. $d$ horizontal saddle
connections).
\end{Lemma}
\begin{proof}
A critical point corresponding to a zero of an Abelian
differential of degree $k$ has $k+1$ incoming and $k+1$
outgoing separatrix rays, where a marked point is
interpreted in this context as a ``zero of degree $0$''.
Thus, the total number $n$ of horizontal saddle connections
satisfies
$$
n=\text{sum of degrees of zeroes}
+ (\text{number of zeroes and marked points})\,.
$$
Note that
$$
2g-2=\text{sum of degrees of zeroes}\,,
$$
and that
$$
d=2g
+ (\text{number of zeroes and marked points}) -1\,,
$$
which implies the equality $n=d-1$.

Similarly, a critical point corresponding to a zero (or to
a simple pole) of degree $k$ of a quadratic differential
has $k+2$ horizontal separatrix rays, where a simple pole
is interpreted in this context as a ``zero of degree
$-1$''. Thus, the total number of horizontal saddle
connections of a Jenkins--Strebel quadratic differential
equals half the total number of separatrix rays. On the
other hand, the sum of degrees of zeroes and simple poles
equals $4g-4$ and the dimension of the stratum equals
$$
d=2g+(\text{number of zeroes, marked points and simple
poles})-2\,,
$$
which implies the assertion for quadratic differentials.
\end{proof}

For example, the Jenkins--Strebel Abelian differential
represented at Figure~\ref{fig:Jenkins:Strebel} belongs to
the stratum $\cH(1^4)$ of complex dimension $d=9$; it has
$9-1=8$ edges; see~\cite{Zorich:representatives} for this
and for further examples.

Removing the union of all horizontal saddle connections of
an Abelian Jenkins-Strebel differential from the
associated translation surface we decompose the surface
into a disjoint union of maximal cylinders filled with
closed regular leaves of the horizontal foliation. Denote
by $\nofcyl$ the number of these maximal horizontal
cylinders. The corresponding separatrix diagram viewed as a
ribbon graph has $2\nofcyl$ boundary components.

We will say that a vector $v\in\R^n$ is \textit{strictly
positive} if all its components are strictly positive
numbers, that is if $v$ belongs to $\R^n_+=\R^n_{>0}$. Any
Jenkins-Strebel differential having a given separatrix
diagram $\Gamma$ is completely determined by the following
additional \textit{length data}
\begin{itemize}
\item the length of each horizontal saddle connection $(\ell_i)_{i \in \{1,\ldots,d-1\}}
\in \R_{>0}^{d-1}$,
\item the height of each maximal horizontal cylinder $(h_j)_{j \in \{1, \ldots, m\}} \in \R_{>0}^m$,
\item a twist parameter for each cylinder $(t_j)_{j \in \{1, \ldots, m\}}\in\R^m$.
The twist parameter $t_j$ is determined by a choice of a
saddle connection $\gamma_j$ entirely contained in the corresponding
cylinder and joining the two opposite boundary
components of this cylinder.
\end{itemize}

For an Abelian Jenkins--Strebel differential $\omega$ the
length $\ell_i$ of each horizontal saddle connection
$\lambda_i$ coincides with the period
$\int_{\lambda_i}\omega$ under appropriate choice of
orientation of $\lambda_i$. Under appropriate choice of
orientation of $\gamma_j$, the period of $\omega$ along the
saddle connection $\gamma_j$ equals $t_j+i\cdot h_j$, where
$i=\sqrt{-1}$ in the latter expression.

In the example presented in
Figure~\ref{fig:Jenkins:Strebel} all horizontal saddle
connections of the separatrix diagram corresponding to the
Abelian differential $\omega$ have lengths
$\ell_i=|\lambda_i|=1$ for $i=1,\dots,8$. The height $h_1$
of the single maximal horizontal cylinder is equal to $1$
and the twist parameter corresponding to the saddle
connection $\gamma_1$ (represented by the lateral sides of
the parallelogram pattern) is also equal to $1$. The
relative period $\int_{\gamma_1}\omega$ is equal to $1+i$,
where $i=\sqrt{-1}$ in the latter expression.

Similarly, for a Jenkins--Strebel quadratic differential
$q$ on a complex curve $X$ the length parameters
$(\ell,t,h)$ can be interpreted in terms of periods of the
holomorphic form $\omega$ on the canonical double cover
$p:\hat X\to X$ where $p^\ast q =\omega^2$.

The length data is subject to the constraint that the sum
of the lengths of saddle connections on respectively top
and bottom boundary component of each maximal horizontal
cylinder should be the same. We denote by
$w_{\Gamma,j}^{top}(\ell)$ (respectively by
$w_{\Gamma,j}^{bot}(\ell)$) the sum of the lengths $\ell_i$
of those horizontal saddle connections that appear on the
top (respectively bottom) of the $j$-th cylinder. Note that
in the case of quadratic differentials the same saddle
connection might appear at the same boundary component
twice. In this case the corresponding parameter $\ell_i$ is
counted with weight $2$. The linear conditions are then
expressed as
\begin{equation}
\label{eq:bot:equal:top}
w_{\Gamma,j}^{top}(\ell) = w_{\Gamma,j}^{bot}(\ell)
\quad\text{for } j=1,\ldots,m\,.
\end{equation}
For an Abelian Jenkins--Strebel differential there is an
obvious linear relation between these constraints. Namely,
since every horizontal saddle connection is present exactly
once on top of some cylinder, and exactly once on the
bottom of some cylinder, taking the sum of
equations~\eqref{eq:bot:equal:top} we get a tautological
identity $\sum_{j=1}^m w_{\Gamma,j}^{top} = \sum_{j=1}^m
w_{\Gamma,j}^{bot}$ having the sum
$\ell_1+\ell_2+\dots+\ell_{d-1}$ (arranged in certain
order) on each side of the identity.
Lemma~\ref{lm:rank} below proves that this is the only
relation between equations~\eqref{eq:bot:equal:top} for
Abelian differentials and that the
equations~\eqref{eq:bot:equal:top} are independent for
quadratic differentials.

\begin{Lemma}
\label{lm:rank}
Consider the decomposition of a Jenkins--Strebel
Abelian differential or of a Jenkins--Strebel
quadratic differential with at most simple poles
into maximal horizontal cylinders obtained by removing
from the surface all horizontal saddle connections.
Let $m$ be the number of such cylinders.

The linear system~\eqref{eq:bot:equal:top}
has rank $m-1$ for Abelian Jenkins--Strebel differentials
and rank $m$ for quadratic Jenkins--Strebel differentials.
\end{Lemma}
\begin{proof}
We start with the case of Abelian Jenkins--Strebel
differentials. Let $r$ be the rank of linear
system~\eqref{eq:bot:equal:top}. We have seen that the
length parameters $\ell_i$, $t_j$ can be interpreted as
cocycles in the relative cohomology group
$H^1(S,\Sigma,\R)$. The cylinder twists $t_j$ can be chosen
arbitrary, which gives $m$ free parameters independent of
parameters $\ell_i$. We can complete them with $(d-1)-r$
additional free parameters $\ell_i$. Since
$\dim_{\R}H^1(S,\Sigma,\R)=d$, we conclude that $r\ge
m-1$. On the other hand, we have already seen that $r\le
m-1$ since there is at least one linear relation between
equations~\eqref{eq:bot:equal:top}.

The count for quadratic Jenkins--Strebel differentials is
analogous. The $m$ twists $t_j$ provide $m$ independent
parameters. However, this time by
Lemma~\ref{lm:number:of:edges:of:sep:diagram} we have $d$
horizontal saddle connections, and the linear span of all
independent parameters $\ell_i,t_j$ can have real dimension
at most $d$, which implies that linear
system~\eqref{eq:bot:equal:top} has to have the full rank
$r=m$.
\end{proof}

We will use the following elementary Lemma in the proof of
Proposition~\ref{prop:Re:leaf}.

\begin{Lemma}
\label{lem:span:hom}
Let $(X,\omega)$ be a Jenkins--Strebel holomorphic 1-form
on a complex curve $X$. Consider the decomposition of $X$
into union of maximal cylinders filled with closed
horizontal leaves. Consider the collection of all
horizontal saddle connections completed for each cylinder
by a non-horizontal segment inside the cylinder
joining some pair of singularities on the two boundary
components of the cylinder. Viewed as a collection of
relative homology cycles, such collection of saddle
connections spans the entire relative homology group
$H_1(X,\Sigma;\Z)$.
\end{Lemma}
\begin{proof}
The complement to the union of our segments is a disjoint
union of topological discs. Thus, the collection of
segments as above defines a $1$-skeleton of a $CW$
decomposition of the topological surface $S$ underlying the
complex curve $X$, where all $0$-cells belong to the finite
set $\Sigma$ of zeros and marked points of the differential.
\end{proof}

The following simple observation is of crucial importance
for us.

\begin{Proposition}
\label{prop:Re:leaf}
Let $\Gamma$ be the separatrix diagram corresponding to a
Jenkins--Strebel Abelian differential $(X,\omega)$. Denote
by $h$ the vector of heights of the associated maximal
horizontal cylinders. Let $\cH(\kappa)^{\mathit{comp}}$ be the
connected component of the ambient stratum containing
$(X,\omega)$. Consider the subset
$\cH(\kappa)^{\mathit{comp}}_{\Gamma,h}\subset\cH(\kappa)^{\mathit{comp}}$ of
all Jenkins--Strebel differentials in $\cH(\kappa)^{\mathit{comp}}$
sharing with $(X,\omega)$ the separatrix diagram $\Gamma$
and the vector of heights $h$. The subset
$\cH(\kappa)^{\mathit{comp}}_{\Gamma,h}$ coincides with the
$\Re$-leaf through $(X,\omega)$.
\end{Proposition}
\begin{proof}
The collection of parameters $\ell_i$, where
$i=1,\dots,d-1$, and $t_j+i\cdot h_j$, where $j=1,\dots,m$,
associated to the separatrix diagram $\Gamma$ represents
relative periods over a generating family of relative
cycles, see Lemma~\ref{lem:span:hom}. Thus, keeping the
vector of heights $h$ fixed and deforming the parameters
$\ell_i, t_j$ we stay inside the $\Re$-leaf through
$(X,\omega)$. On the other hand, Lemma~\ref{lm:rank}
implies that the space of $(\ell_i,t_j)$-parameters has the
same dimension $d=\dim_{\R}H^1(X,\Sigma;\R)$ as the
$\Re$-leaf. This implies, that the connected component of
$\cH(\kappa)^{\mathit{comp}}_{\Gamma,h}$ containing $(X,\omega)$
coincides with the $\Re$-leaf passing through $(X,\omega)$.

It remains to note that the set
$\cH(\kappa)^{\mathit{comp}}_{\Gamma,h}$ is connected. Indeed, in
$(\ell,t,h)$-coordinates, the coordinate $h$ is fixed; the
coordinates $t\in\R^m$ can be chosen arbitrary and the
coordinate vector $\ell$ can be chosen as arbitrary vector
in the intersection of $\R^{d-1}_+$ with the linear
subspace defined by equations~\eqref{eq:bot:equal:top}.
The latter intersection is, clearly, connected.
\end{proof}

Similarly, let $\Gamma$ be the separatrix diagram
corresponding to a Jenkins--Strebel quadratic differential
$(X,q)$. As before, denote by $h$ the vector of heights of
the associated maximal horizontal cylinders. Let
$\cQ(\xi)^{\mathit{comp}}$ be the connected component of the ambient
stratum containing $(X,q)$. Let $\cL$ be the arithmetic
invariant suborbifold obtained by the canonical double
cover construction from $\cQ(\xi)^{\mathit{comp}}$ (see
Section~\ref{ss:quad:as:invariant:orbifold}); let
$\cH(\kappa)$ be the stratum of Abelian differentials
ambient for $\cL$. Let $(\hat X, \omega)\in \cL$ be the
Abelian differential representing the canonical double
cover $p:\hat X\to X$ such that $p^\ast q=\omega^2$.

Consider the subset
$\cQ(\xi)^{\mathit{comp}}_{\Gamma,h}\subset\cQ(\xi)^{\mathit{comp}}$ of all
Jenkins--Strebel differentials in $\cQ(\xi)^{\mathit{comp}}$ sharing
with $(X,q)$ the separatrix diagram $\Gamma$ and the vector
of heights $h$. Consider the subset obtained by applying
the double cover construction to Jenkins--Strebel quadratic
differentials in $\cQ(\xi)^{\mathit{comp}}_{\Gamma,h}$. Let
$\cL_{\Gamma,h}\subset\cL$ be the connected
component of this set containing $(\hat X,\omega)$.

\begin{Proposition}
\label{prop:Re:leaf:quadratic}
The subset $\cL_{\Gamma,h}(\xi)$ coincides with
the $\Re$-leaf of $\cL$ passing through $(\hat X,\omega)$.
\end{Proposition}
\begin{proof}
The proof is analogous to the proof of the
Proposition~\ref{prop:Re:leaf}.
\end{proof}

As it was already implicitly done above, we always assume
that the maximal horizontal cylinders associated to a
separatrix diagram are numbered from $1$ to $m$ and that
the horizontal saddle connections are numbered from $1$ to
$d-1$. Suppose that a Jenkins--Strebel differential
$(X,\omega)$ has separatrix diagram $\Gamma$ and associated
parameters $(\ell, h, t)$. A collection of parameters
$(\ell', h', t')$ determine the same translation surface
$(X,\omega)$ if and only if $\ell = \ell'$, $h = h'$ and
\[
t_j - t'_{j} \equiv 0\, (\operatorname{mod}\, w_j)\,,\qquad\text{for }j=1,2,\dots,m\,,
\]
where $w_j = w_{\Gamma,j}^{bot}(\ell) =
w_{\Gamma,j}^{bot}(\ell')$. For this reason the parameters
$t_j$ are only relevant modulo $w_j$.

A square-tiled surface always represents a Jenkins--Strebel
differential. In the coordinate system as above a
Jenkins--Strebel differential corresponds to a square-tiled
surface if and only if the metric data $(\ell, h, t)$ is
integer.

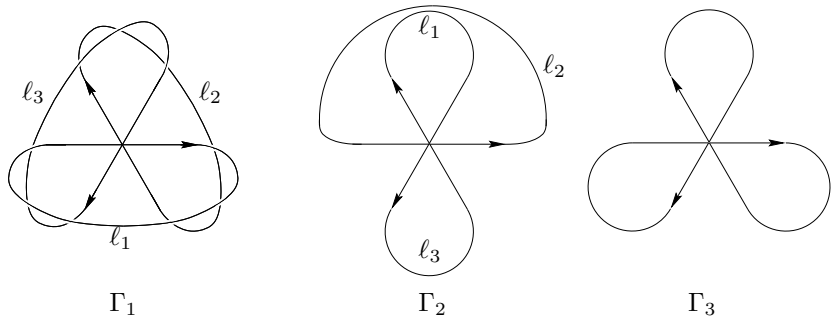
\begin{figure}[htb]
\includegraphics{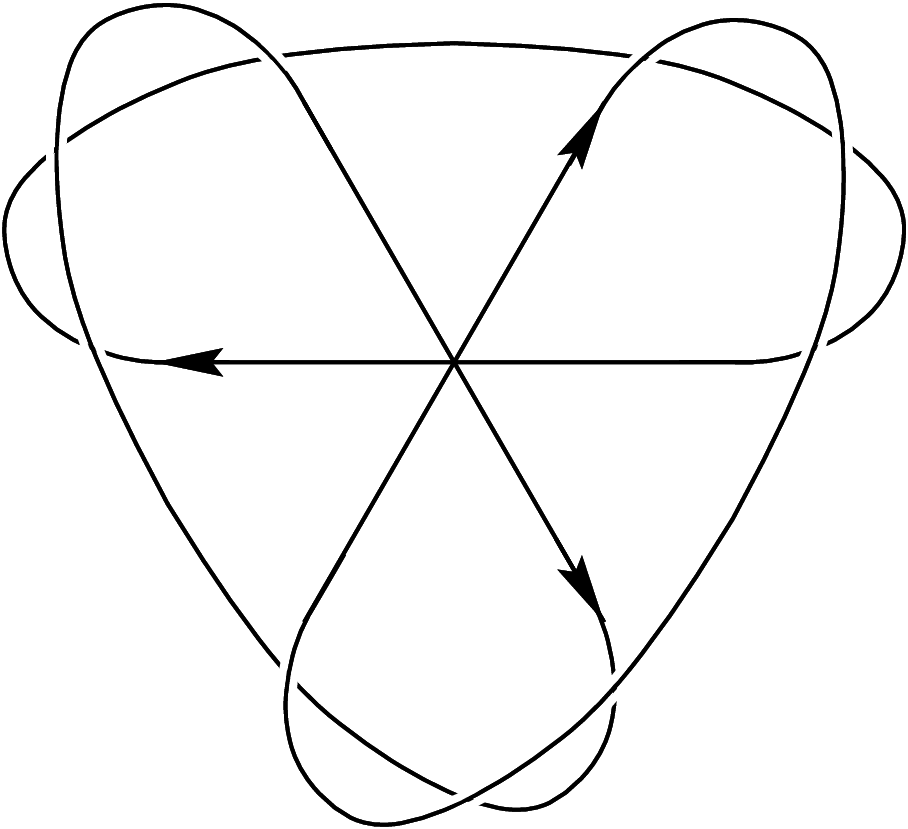}
\includegraphics{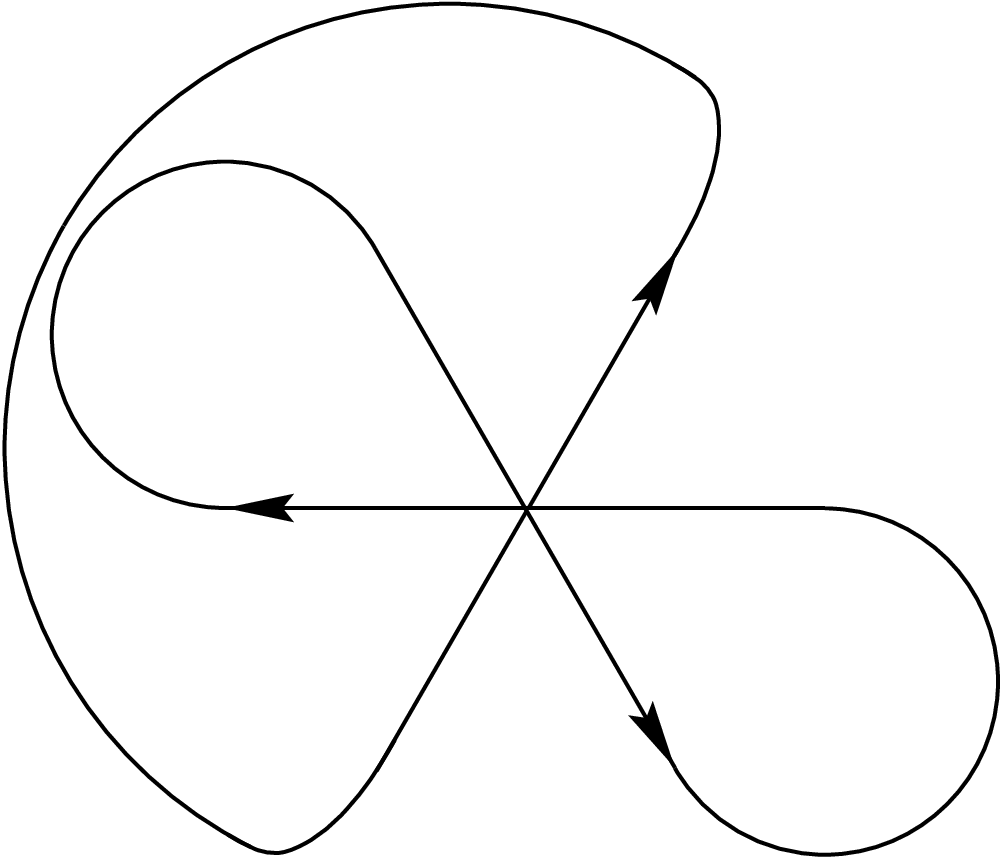}
\includegraphics{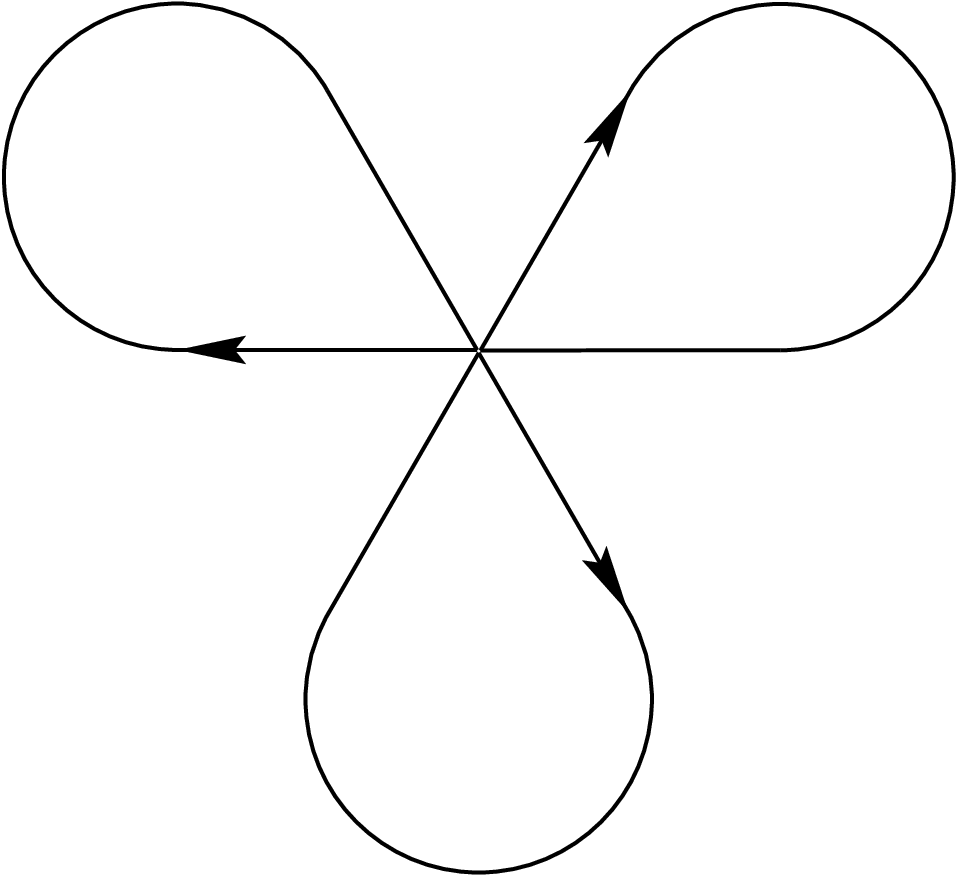}
\vspace{125bp} %
\begin{picture}(0,0)(170,-10)
\put(180,0){\begin{picture}(0,0)(0,0)
\put(-129,20){$\ell_1$}
\put(-95,75){$\ell_2$}
\put(-162,75){$\ell_3$}
\put(-12,100){$\ell_1$}
\put(35,85){$\ell_2$}
\put(-12,15){$\ell_3$}
\put(-129,-5){$\Gamma_1$}
\put(-12,-5){$\Gamma_2$}
\put(90,-5){$\Gamma_3$}
\end{picture}}
\end{picture}
\caption{
\label{fig:diag}
Ribbon graphs $\Gamma_1$ and $\Gamma_2$ are realizable
as separatrix diagrams while $\Gamma_3$ is not.}
\end{figure}

Note that a general ribbon graph $\Gamma$ does not
necessarily represent a separatrix diagram of some
Jenkins--Strebel Abelian differential. The horizontal
foliation of a Jenkins--Strebel Abelian differential is
oriented so the ribbon graph $\Gamma$ should be
\textit{orientable}: it should admit orientation of edges
in such way that at every vertex the incoming edges and
outgoing edges alternate with respect to the cyclic order
induced by the ribbon structure. The pairing of boundary
components should respect orientation of the surface.
And finally, the system of equations~\eqref{eq:bot:equal:top}
should admit a strictly positive solution
$(\ell_1,\dots,\ell_{d-1})$.

\begin{Example}
Figure~\ref{fig:diag} represents all oriented ribbon graphs
having a single vertex of valence $6$. Thus, any separatrix
diagram realized by a Jenkins--Strebel differential
from the stratum $\cH(2)$ is represented
by one of these three ribbon graphs.

The graph $\Gamma_1$ has two boundary components, so there
is a unique way to glue a cylinder to this ribbon graph. We
have already seen that in the case of a $1$-cylinder
separatrix diagram the system of
equations~\eqref{eq:bot:equal:top} degenerates to an
identity, which in the particular case of $\Gamma_1$ has
the form
$$
\ell_1+\ell_2+\ell_3=\ell_1+\ell_3+\ell_2\,.
$$
Thus, $\Gamma_1$ is realizable as a separatrix diagram,
which is the unique $1$-cylinder separatrix diagram in
$\cH(2)$. Any collection of parameters $(\ell,t,h)$ with
$\ell\in\R^{3}_+$, $t\in\R$, $h\in\R_+$ defines a legal
Jenkins--Strebel differential in $\cH(2)$. Restricting the
twist parameter to the subdomain $0\le t <
\ell_1+\ell_2+\ell_3=w_1$ we get a polyhedral cone
representing a single coordinate chart for all
Jenkins--Strebel differentials in $\cH(2)$ corresponding to
this separatrix diagram. Any point of the polyhedral cone
defines a unique well-defined Jenkins--Strebel differentials
in $\cH(2)$. Up to the symmetry of order $3$ which
cyclically changes the coordinates $(\ell_1,\ell_2,\ell_3)$,
every $1$-cylinder Jenkins--Strebel differentials in
$\cH(2)$ is represented by a unique point $(\ell,t,h)$ of
the resulting polyhedral cone.

The graph $\Gamma_2$ has four boundary components. However,
there is a unique way to glue two cylinders to this ribbon
graph in such way that the resulting closed surface would
be orientable and that the system of
equations~\eqref{eq:bot:equal:top} would admit a strictly
positive solution. Namely, we have to glue one of the two
cylinders to the inner boundary component of the loop
$\ell_1$ and to the inner boundary component of the loop
$\ell_3$, and the other cylinder to the remaining pair of
boundary components. This time the system of
equations~\eqref{eq:bot:equal:top} imposes a nontrivial
constraint $\ell_1=\ell_3$. As a coordinate chart in the
space of parameters $(\ell,t,h)$ we can choose the
following polyhedron. Take the intersection of the strictly
positive octant $\R^3_+$ with the plane
$\ell_1=\ell_3$ for the parameters $\ell$. Let $0\le t_1<
\ell_1$ and $0\le t_2< \ell_1+\ell_2$. Let $h\in\R^2_+$.
Any Jenkins--Strebel differential having this separatrix
diagram can be represented by appropriate parameters
$(\ell,t,h)$ in this polyhedral cone, and distinct points
of the polyhedral cone represent distinct (and
well-defined) Jenkins--Strebel differentials.

It is easy to verify that the ribbon graph $\Gamma_3$
is not realizable as a separatrix diagram: no matter how we
arrange the boundary components into pairs,
the system~\eqref{eq:bot:equal:top} does not admit any
strictly positive solution.
\end{Example}

By Proposition~\ref{prop:Re:leaf} every subset
$\cH(\kappa)^{\mathit{comp}}_{\Gamma,h}$ obtained by fixing the
collection $h$ of heights of cylinders of a Jenkins--Strebel
differential in $\cH(\kappa)^{\mathit{comp}}$ and by varying the
length parameters $(\ell,t)$ associated to the
corresponding separatrix diagram $\Gamma$ coincides with a
$\Re$-leaf in $\cH(\kappa)^{\mathit{comp}}$. Geometrically, each
such $\cH(\kappa)^{\mathit{comp}}_{\Gamma,h}$ is a torus bundle. The base of
this bundle is the polyhedral cone $C_+(\Gamma)$ (or possibly
its quotient with respect to a finite symmetry group) obtained as the
intersection of $\R^{d-1}_+$ with the linear
subspace~\eqref{eq:bot:equal:top}. The fiber over a point
$\ell$ in $C_+(\Gamma)$ is the torus
\begin{equation}
\label{eq:torus:of:twists}
\T^m=\R^m /(w_1\Z\oplus w_2\Z\oplus\dots\oplus w_m\Z)\,,
\end{equation}
where $w_1(\ell),\dots,w_m(\ell)$ are the perimeters of the
cylinders.

Some of separatrix diagrams realizable for the ambient
stratum $\cH(\kappa)$ might be not necessarily realizable
in a given invariant arithmetic suborbifold $\cL$. For
example, any Jenkins--Strebel differential in the
$\GLR$-orbit $\cL$ of the Eierlegende Wollmilchsau
represented in Figure~\ref{fig:Eierlegende:Wollmilchsau}
\begin{figure}[htb]
   %
   %
\includegraphics{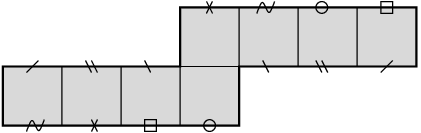}
\vspace{60pt}
\caption{
\label{fig:Eierlegende:Wollmilchsau}
Eierlegende Wollmilchsau.}
\end{figure}
has the same separatrix diagram as the Eierlegende
Wollmilchsau. In particular, $\cL$ does not contain a
single $1$-cylinder square-tiled surface.

Given an invariant arithmetic suborbifold $\cL$, we say
that a separatrix diagram $\Gamma$ is \textit{realizable in
$\cL$} if $\cL$ contains a Jenkins--Strebel differential
having the separatrix diagram $\Gamma$. We call a vector $h
= (h_1, h_2, \ldots, h_m)$ \emph{admissible for a
separatrix diagram} $\Gamma$ in $\cL$ if there is a
Jenkins--Strebel differential in $\cL$ with separatrix
diagram $\Gamma$ and with horizontal cylinders of heights
$h$.

\subsection{Densities of subsets of square-tiled surfaces with prescribed
horizontal or vertical decompositions}
\label{s:sep:diag:densities}
We now apply the results of
Sections~\ref{ss:uniform:density:via:horocyclic:invariance}--\ref{s:sep:diagrams}
to square-tiled surfaces with fixed combinatorics. In
particular, we prove
Corollary~\ref{cor:density:fixed:num:comp} that generalizes
Theorem~\ref{th:c11:in:genus:0} from the Introduction.

Though we state the main results only for the connected
components of the strata of Abelian and quadratic
differentials, all the results of this Section are valid
for any invariant arithmetic suborbifold. However, we postpone
these more general proofs to a separate
paper~\cite{DGZZ:arithmetic:suborbifolds} to avoid
overloading the current one.

In Section~\ref{ss:quad:as:invariant:orbifold} we assigned
to every connected component $\cQ(\xi)^{\mathit{comp}}$ of every
stratum of meromorphic quadratic differentials with at most
simple poles an invariant arithmetic orbifold $\cL$ in an
appropriate stratum of Abelian differentials. By
convention, everywhere in Section~\ref{s:Equidistribution}
we mark the preimages of simple poles on the cover. By
Proposition~\ref{prop:volumes:quad:and:double:cover} the
square-tiled surfaces in $\cQ(\xi)^{\mathit{comp}}$ and in $\cL$ are
in the natural correspondence. In particular, we have
direct relation between densities of the corresponding
subsets of square-tiled surfaces measured in
$\cQ(\xi)^{\mathit{comp}}$ and in $\cL$.

Consider a connected component $\cL$ of some stratum of
Abelian or quadratic differentials and a separatrix diagram
$\Gamma$ realizable in it. Given an integer height vector
$h$ admissible for $\Gamma$, we denote by
$\cL^{\mathit{horiz}}_{\Z,\Gamma,h}$ the set of square-tiled
surfaces in $\cL_{\Z}$ having separatrix diagram $\Gamma$
and vector $h$ of cylinder heights. We define
$\cL^{\mathit{vert}}_{\Z,\Gamma,h}$ as the set of square-tiled
surfaces obtained from square-tiled surfaces in
$\cL^{\mathit{horiz}}_{\Z,\Gamma,h}$ by rotation by $\frac{\pi}{2}$
clockwise.

\begin{Theorem}
\label{thm:density:cylinder:diag}
Let $\cL$ be a connected component of a stratum of Abelian
differentials, or an invariant arithmetic suborbifold in a
stratum of Abelian differentials obtained by the canonical
double cover construction from a connected component of a
stratum of quadratic differentials. Let $\Gamma$ be a
separatrix diagram realizable in $\cL$. Then for any
integer height vector $h$ admissible for $\Gamma$, the
subset $\cL^{\mathit{horiz}}_{\Z,\Gamma,h}$ of $\cL_{\Z}$ has a non-zero
density. Moreover, $c(\cL^{\mathit{horiz}}_{\Z,\Gamma,h}) =
\delta(\cL^{\mathit{horiz}}_{\Z,\Gamma,h})\cdot\Vol_1(\cL_1) $ is a
rational number and we have the following explicit
convergence rate as $N \to \infty$:
\[
\cN_{\cL^{\mathit{horiz}}_{\Z,\Gamma,h}}(\cL, N) = c(\cL^{\mathit{horiz}}_{\Z,\Gamma,h})\cdot \frac{N^d}{2d} + O(N^{d-1})\,,
\]
where $d = \dim_\C \cL$.
\end{Theorem}
\begin{proof}
We start the proof with the case when $\cL$ is a connected component of a stratum
of Abelian differentials.

Let $(\ell,t,h)$ be parameters of the horizontal cylinder
decomposition of a surface from $\cL^{\mathit{horiz}}_{\Z,\Gamma,h}$.
Then $\ell$ satisfies~\eqref{eq:bot:equal:top}, so the
width $w_{\Gamma,j}(\ell)$ of the $j$-th cylinder satisfies
$w_{\Gamma,j}(\ell)=w_{\Gamma,j}^{top}(\ell) =
w_{\Gamma,j}^{bot}(\ell)$. Hence, the flat area of the surface
can be expressed in $(\ell,h)$-coordinates as
\[
\Area(\ell, h, t) = \sum_{j=1}^m h_j \cdot w_{\Gamma,j}(\ell)\,.
\]

Let $L_\Gamma$ be the linear subspace of $\R^{d-1} \times \R^m$
consisting of couples $(\ell, t)$ satisfying the system of linear
equations~\eqref{eq:bot:equal:top}.

Let $P_{\Gamma,h}$ be the
polyhedron in $L_\Gamma$ defined by the following inequalities:
\begin{enumerate}
\item $\ell$ is strictly positive,
\item for each $j \in \{1, \ldots,m\}$, $0 \leq t_j < w_j(\ell)$,
\item $\Area(\ell, h, t) \leq 1$.
\end{enumerate}
It follows from the definition of a separatrix diagram that
$P_{\Gamma,h}$ is a relatively compact rational
polyhedron of full dimension in $L_\Gamma$. This implies that
$\Vol P_{\Gamma,h}$ is a (finite) strictly positive rational
number, where $\Vol$ is the Lebesgue measure in the vector
space $L_\Gamma$ normalized in such way that the fundamental
domain of the lattice $L_\Gamma \cap (\Z^{d-1}\times \Z^m)$ has
volume 1.

Up to a possible symmetry of finite order, square-tiled
surfaces tiled with at most $N$ unit squares in
$\cL^{\mathit{horiz}}_{\Z,\Gamma,h}$ are in the one-to-one
correspondence with the integral points in the inflated
polyhedron $N\cdot P_{\Gamma,h}$ (where the inflation does
not affect the parameters $h$ and acts on the space
$\R^{d-1}\times\R^m$ of $(\ell,t)$-coordinates by homothety
with coefficient $N$). Hence,
\begin{multline*}
\cN(\cL^{\mathit{horiz}}_{\Z,\Gamma,h}, N) =
\frac{1}{\Aut(\Gamma,h)}\cdot
\Card \big(N\cdot P_{\Gamma,h} \cap (\Z^{d-1}\times\Z^{m})\big)
=\\=
\frac{1}{\Aut(\Gamma,h)}\cdot
\Vol(P_{\Gamma,h}) \cdot N^{\dim(P_\Gamma)} + O(N^{\dim(P_\Gamma)-1})\,,
\end{multline*}
where $\dim(P_\Gamma)=\dim(L_\Gamma)=\dim_{\C}\cL$ by
Proposition~\ref{prop:Re:leaf}. By equation~\eqref{equiv:density}
this implies that $\cL^{\mathit{horiz}}_{\Z,\Gamma,h}$ has strictly
positive density $\delta(\cL^{\mathit{horiz}}_{\Z,\Gamma,h})$ and that
the quantity
$$
c(\cL^{\mathit{horiz}}_{\Z,\Gamma,h}) =
\delta(\cL^{\mathit{horiz}}_{\Z,\Gamma,h})\cdot\Vol_1(\cL_1)=
\dim_{\R}\cL\cdot
\frac{\Vol(P_{\Gamma,h})}{\Aut(\Gamma,h)}
$$
is a rational number.

Suppose now that $\cL$ is an invariant arithmetic
suborbifold in an appropriate stratum $\cH(\kappa)$ of
Abelian differentials obtained by the canonical double
cover construction from a connected component $\cQ(\xi)^{\mathit{comp}}$ of a stratum
of meromorphic quadratic differentials with at
most simple poles. Any realizable separatrix diagram
$\hat\Gamma$ in $\cL$ is induced from a realizable
separatrix diagram $\Gamma$ in the original connected
component of a stratum of quadratic differentials. By
convention, we mark all preimages of simple poles (if any).

It would be convenient to denote by $(\ell,t,h)$ the
lengths of horizontal saddle connections, the twist
parameters and the heights of the cylinders of the original
Jenkins--Strebel quadratic differential $(X,q)\in\cQ(\xi)^{\mathit{comp}}$.
Suppose that the original separatrix diagram $\Gamma$
contains $m$ cylinders, that $\gamma_1,\dots,\gamma_m$
represent the non-horizontal saddle connections crossing these
cylinders and that $\lambda_1,\dots,\lambda_d$ denote the
horizontal saddle connections. By
Lemma~\ref{lm:number:of:edges:of:sep:diagram} the number
$d$ of horizontal saddle connections coincides with the
dimension $d=\dim_{\C}\cQ(\xi)$ of the ambient stratum of
quadratic differentials. By construction, the separatrix
diagram $\hat\Gamma$ corresponding to the Jenkins--Strebel
differential on the double cover has $2m$ cylinders.

Every horizontal saddle connection $\lambda_i$ on $(X,q)$
has two preimages $\lambda'_i$ and $\lambda''_i$ on $(\hat
X,\omega)$. Every nonhorizontal saddle connection
$\gamma_j$ also has two preimages $\gamma'_j$ and
$\gamma''_j$ crossing the two maximal horizontal cylinders
on the cover corresponding to two copies of the horizontal
cylinder number $j$ on the base. The induced collection of
parameters $(\ell',\ell'',t',t'',h',h'')$ serves as the
complete collection of length parameters on the covering
Jenkins--Strebel Abelian differential with the separatrix
diagram $\hat\Gamma$. By construction they satisfy the
obvious relations
\begin{equation}
\label{eq:M1:M2}
\begin{cases}
\ell'_i=\ell''_i&\text{for }i=1,\dots,d\,,
\\
t'_j=t''_j&\text{for }j=1,\dots,m\,,
\\
h'_j=h''_j&\text{for }j=1,\dots,m\,,
\end{cases}
\end{equation}
where by construction we have
$\ell'_j=\ell''_j=\ell_j$ for $j=1,\dots,d$;
as well as
$t'_j=t''_j=t_j$ for $j=1,\dots,m$ and
$h'_j=h''_j=h_j$ for $j=1,\dots,m$.

Consider the linear subspace
$L_\Gamma\subset\R^{m-1}\times\R^m$ of simultaneous
solutions of the following two systems of linear equations:
system~\eqref{eq:bot:equal:top} associated to the
separatrix diagram $\hat\Gamma$ and
system~\eqref{eq:M1:M2}. Suppose that some collection
$(\ell',\ell'',t',t'',h',h'')$ with strictly positive
$\ell',\ell'',h',h''$ defines a point of $L_\Gamma$. Define
$\ell_j:=\ell'_j=\ell''_j$ for $j=1,\dots,d$; define
$t_j:=t'_j=t''_j=t_j$ for $j=1,\dots,m$ and define
$h_j:=h'_j=h''_j$ for $j=1,\dots,m$. The collection
$(\ell,t,h)$ defines a Jenkins--Strebel differential in
$\cQ(\xi)^{\mathit{comp}}$ with separatrix diagram $\Gamma$ and vector of
heights $h$. Reciprocally, any Jenkins--Strebel
differential in $\cQ(\xi)^{\mathit{comp}}$ with separatrix diagram $\Gamma$
and vector of heights $h$ defines a point in $L_\Gamma$.

The rest of the proof is completely analogous to the case
of Abelian differentials treated above.
\end{proof}

Applying the results from Section~\ref{ss:uniform:density:via:horocyclic:invariance}
and~\ref{ss:Invariance:along:Re:and:Im:foliations} we obtain the following
corollary.

\begin{Corollary}
\label{cor:fixed:horiz:vert:decompositions}
Let $\cL$ be a connected component of a stratum of Abelian
differentials, or an invariant arithmetic suborbifold in a
stratum of Abelian differentials obtained by the canonical
double cover construction from a connected component of a
stratum of quadratic differentials. Let $\Gamma$ and
$\Gamma'$ be realizable separatrix diagrams in $\cL$ and let
$h$ and $h'$ be admissible integer vectors of heights of
cylinders for $\Gamma$ and $\Gamma'$ respectively. Then all
of $\cL^{\mathit{horiz}}_{\Z,\Gamma,h}$,
$\cL^{\mathit{vert}}_{\Z,\Gamma',h'}$ and
$\cL^{\mathit{horiz}}_{\Z,\Gamma,h}\, \cap\,
\cL^{\mathit{vert}}_{\Z,\Gamma',h'}$ have strictly positive uniform
densities in $\cL$ and
\[
\delta(\cL^{\mathit{horiz}}_{\Z,\Gamma,h} \cap \cL^{\mathit{vert}}_{\Z,\Gamma',h'})
=
\delta(\cL^{\mathit{horiz}}_{\Z,\Gamma,h})\ \cdot\ \delta(\cL^{\mathit{vert}}_{\Z,\Gamma',h'})\,,
\]
where $\delta(\cL^{\mathit{vert}}_{\Z,\Gamma',h'})=\delta(\cL^{\mathit{horiz}}_{\Z,\Gamma',h'})$.
\end{Corollary}
\begin{proof}
Theorem~\ref{thm:density:cylinder:diag} shows that the set
$\cL^{\mathit{horiz}}_{\Z,\Gamma,h}$ has a strictly positive
density. By Proposition~\ref{prop:Re:leaf} (respectively by
Proposition~\ref{prop:Re:leaf:quadratic} in the case of
quadratic differentials) the set
$\cL^{\mathit{horiz}}_{\Z,\Gamma,h}$ is $\Re$-invariant, which
allows us to use Theorem~\ref{th:uncorrelated} and
conclude that the density is, actually, uniform.

The analogous results hold for the vertically completely
periodic surfaces with separatrix diagram $\Gamma'$, namely, by
the same arguments the set $\cL^{\mathit{vert}}_{\Z,\Gamma',h'}$ has
uniform strictly positive density and is $\Im$-invariant.

Theorem~\ref{th:uncorrelated} implies that the intersection
$\cL^{\mathit{horiz}}_{\Gamma,h} \cap \cL^{\mathit{vert}}_{\Gamma',h'}$ has uniform
strictly positive density equal to the product of the
densities of the two sets.

To complete the proof note that the set
$\cL^{\mathit{vert}}_{\Z,\Gamma',h'}$ is constructed by applying to each
square-tiled surface in the set $\cL^{\mathit{horiz}}_{\Z,\Gamma',h'}$
clockwise rotation by $\frac{\pi}{2}$. Thus,
$\cN_{\cL^{\mathit{vert}}_{\Z,\Gamma',h'}}(\cL,N)=\cN_{\cL^{\mathit{horiz}}_{\Z,\Gamma',h'}}(\cL,N)$.
The equality
$\delta(\cL^{\mathit{vert}}_{\Z,\Gamma',h'})=\delta(\cL^{\mathit{horiz}}_{\Z,\Gamma',h'})$
now follows directly from the
definition~\eqref{eq:def:delta} of the density.
\end{proof}

In Section~\ref{ss:pairs:of:multicurves:as:square:tiled:surfaces}
we assigned to every
square-tiled surfaces on the sphere
a pair of transverse multicurves.
Analogously,
one can assign to any square-tiled surface in a stratum of
Abelian differential $\cH(\kappa)$
a pair of transverse multicurves composed of all closed
regular flat geodesics passing through the centers of the squares.

Recall that closed regular horizontal geodesics on a
square-tiled surface $(X,\omega,\Sigma)$ are organized into
maximal horizontal cylinders having at least one point of
the set $\Sigma$ on each of the boundary components and no
points of $\Sigma$ in the interior of the cylinder. All
regular closed horizontal geodesics inside each maximal
horizontal cylinder belong to the same free homotopy class
on $X\setminus\Sigma$. Similar property is valid for the
vertical cylinder decomposition. This observation allows to
express the horizontal (respectively vertical) multicurve
associated to a square-tiled surface in terms of the
horizontal (respectively vertical) cylinder decomposition
as the weighted sum $h_1 \gamma_1 + h_2 \gamma_2 + \ldots +
h_m \gamma_m$, where $h_i$ and $\gamma_i$ are respectively
the height and the core curve of the $i$-th cylinder in
this decomposition. (Here by the ``height'' of a
square-tiled cylinder we mean the number of circular bands
of squares which form the cylinder.) In particular, the
number of components of the resulting horizontal
(respectively vertical) multicurve counted with weights is
the total height $h_1 + h_2 + \ldots + h_m$ of all maximal
horizontal (respectively vertical) cylinders.

Let $\cL$ be a connected component of a stratum of Abelian
differentials or the invariant arithmetic orbifold obtained
by the canonical double cover construction from a component
of a stratum of quadratic differentials. Denote by
$\cL^{\mathit{horiz}}_{\Z,k}$ (respectively by
$\cL^{\mathit{vert}}_{\Z,k}$) the subset of square-tiled
surfaces in $\cL$ whose associated horizontal (respectively
vertical) multicurve has $k$ components. In other words
$\cL^{\mathit{horiz}}_{\Z,k}$ (respectively
$\cL^{\mathit{vert}}_{\Z,k}$) is the subset of square-tiled
surfaces in $\cL$ tiled with unit squares arranged into
exactly $k$ horizontal (respectively vertical) bands of
squares. In terms of the docomposition into maximal
horizontal (respectively vertical cylinders) we have
$k=h_1+\dots+h_m$, where $h_1,\dots,h_m$ are the heights of
the cylinders.

\begin{Corollary}
\label{cor:density:fixed:num:comp}
Let $\cL$ be a connected component of a stratum of Abelian
differentials or the invariant arithmetic orbifold obtained
by the canonical double cover construction from a component
of a stratum of quadratic differentials. Let $k, k_h, k_v\in
\N$ be positive integers in the case of Abelian
differentials and even positive integers in the case of
meromorphic quadratic differentials with at most simple
poles.

For any such $\cL, k, k_h, k_v$, all of
$\cL^{\textit{horiz}}_{\Z,k_h}$,
$\cL^{\textit{horiz}}_{\Z,k_v}$ and
$\cL^{\textit{horiz}}_{\Z,k_h}\cap\cL^{\textit{horiz}}_{\Z,k_v}$
have strictly positive uniform densities in $\cL$ and
\begin{equation}
\label{eq:delta:delta}
\delta(\cL^{\mathit{horiz}}_{\Z,k_h}
\cap \cL^{\mathit{vert}}_{\Z,k_v})
=
\delta(\cL^{\mathit{horiz}}_{\Z,k_h})\
\cdot\ \delta(\cL^{\mathit{vert}}_{\Z,k_v})\,,
\end{equation}
where $\delta(\cL^{\mathit{horiz}}_{\Z,k}) =
\delta(\cL^{\mathit{vert}}_{\Z,k})$.

Moreover,
$c(\cL_{\Z,k}) = \delta(\cL^{\mathit{horiz}}_{\Z,k}) \cdot\Vol_1(\cL_1) $
is a rational number and we have the explicit convergence
rate
$$
\cN_{\cL^{\mathit{horiz}}_{\Z,k}}(\cL, N)
=\cN_{\cL^{\mathit{vert}}_{\Z,k}}(\cL, N)
= c(\cL_{\Z,k})\cdot \frac{N^d}{2d} + O(N^{d-1})\,,
$$
where $d = \dim_\C \cL$.
\end{Corollary}
\begin{proof}
The proof follows from Corollary~\ref{cor:fixed:horiz:vert:decompositions}.
Note that each stratum has
only finite number of separatrix diagrams and hence
\begin{equation}
\label{eq:L:horiz:Z:k}
\cL^{\mathit{horiz}}_{\Z,k}=\bigsqcup_{\Gamma} \bigsqcup_{h = (h_1, \ldots, h_m)} \cL^{\mathit{horiz}}_{\Z,\Gamma,h}\,.
\end{equation}
Here the first union is taken with respect to the finite
set of separatrix diagrams realizable in $\cL$. For each
such separatrix diagram we denote by $m=m(\Gamma)$ the
number of maximal horizontal cylinders associated to
$\Gamma$. The second (finite) union is taken with respect
to all positive integer heights $h_1,\dots,h_m$ admissible
for $\Gamma$ in $\cL$ and satisfying the relation $h_1 +
\ldots + h_m = k$. This proves that for any integer $k$ the
set $\cL^{\mathit{horiz}}_{\Z,k}$ has uniform density
satisfying the relation
\begin{equation}
\label{eq:delta:horiz:Z:k}
\delta(\cL^{\mathit{horiz}}_{\Z,k})=\sum_{\Gamma} \sum_{h}
\delta(\cL^{\mathit{horiz}}_{\Z,\Gamma,h})\,,
\end{equation}
where the $\delta(\cL^{\mathit{horiz}}_{\Z,\Gamma,h})$ are
the densities appearing in
Corollary~\ref{cor:fixed:horiz:vert:decompositions}.

To justify that for any $k\in\N$, in the case when $\cL$ is
a component of a stratum of Abelian differentials
(respectively, for any $k\in 2\N$ in the case when $\cL$ is
an invariant arithmetic orbifold associated to a component
of a stratum of quadratic differentials)
$\delta(\cL^{\mathit{horiz}}_{\Z,k})$ is a strictly
positive number, it is sufficient to use the fact that each
component of any stratum of Abelian or quadratic
differentials admits a one-cylinder square-tiled surface,
see~\cite{Kontsevich:Zorich}
and~\cite{Zorich:representatives}. Considering the
associated separatrix diagram $\Gamma$ we conclude that the
term $\delta_{\Z,\Gamma,k}$ is already strictly positive.

The proof that $\cL^{\mathit{vert}}_{\Z,k}$ has strictly
positive uniform density is completely analogous. The proof
that
$\cL^{\textit{horiz}}_{\Z,k_h}\cap\cL^{\textit{horiz}}_{\Z,k_v}$
has uniform density equal to product of densities now
follows from Theorem~\ref{th:uncorrelated} in the same way
as was done in
Corollary~\ref{cor:fixed:horiz:vert:decompositions}.

Finally, the rationality of $c(\cL_{\Z,k})$ and the
explicit convergence rate in the remaining assertion follow
now from Theorem~\ref{thm:density:cylinder:diag} applied to
each of the terms in~\eqref{eq:L:horiz:Z:k}
and~\eqref{eq:delta:horiz:Z:k} respectively.
\end{proof}

\begin{Remark}
\label{rm:other:combinations}
We have seen that each of the sets
$\cL^{\mathit{horiz}}_{\Z,\Gamma,h}$ and
$\cL^{\mathit{horiz}}_{\Z,k_h}$, constructed above are
$\Re$-invariant. Similarly, the sets
$\cL^{\mathit{vert}}_{\Z,\Gamma,h}$,
$\cL^{\mathit{vert}}_{\Z,k_v}$ are $\Im$-invariant. Thus,
by Theorem~\ref{th:uncorrelated} all these sets as well as
any intersection of a set from the first group and a set
from the second group has uniform density, and the density
of the intersection is the product of densities of the two
corresponding sets. We do not need to fix combinatorics of
the horizontal and vertical cylinder decomposition in the
same way: the statement is valid for any combination like
$\cL^{\mathit{horiz}}_{\Z,\Gamma,h}\cap\cL^{\mathit{vert}}_{\Z,k}$.
\end{Remark}

We are ready to prove Theorem~\ref{th:c1:in:genus:0} and
Theorems~\ref{th:c11:in:genus:0} with exception for the
explicit value~\eqref{eq:c1:answer} for
$\cyl_1\left(\cQ(\nu,-1^{|\nu|+4})\right)$ which will be
proved in
Section~\ref{s:Computations:for:square:tiled:surfaces}.

\begin{proof}[Proof of Theorems~\ref{th:c1:in:genus:0} and~\ref{th:c11:in:genus:0}]
By assertions (1) and (2) of
Proposition~\ref{prop:volumes:quad:and:double:cover} the
number $\cS^{\mathit{labeled}}_{k,\nu}(N)$ of square-tiled
surfaces in the stratum $\cQ(\nu,-1^{|\nu|+4})$ with
labeled zeros and poles tiled with at most $2N$ squares
organized into $k$ horizontal bands and the number
$\cS^{\mathit{labeled}}_{k_{h},k_{v},\nu}(N)$ of
square-tiled surfaces in the stratum
$\cQ(\nu,-1^{|\nu|+4})$ with labeled zeroes and poles tiled
with at most $2N$ squares composed of $k_{h}$ horizontal
and $k_{v}$ vertical bands of squares satisfies the
following relations:
\begin{align}
\label{eq:h}
\cS^{\mathit{labeled}}_{k,\nu}(N)
&=
\frac{1}{\deg(P)}\cdot
\cN_{\cL^{\mathit{horiz}}_{\Z,k}}(\cL, 4N)\,,
\\
\label{eq:h:v}
\cS^{\mathit{labeled}}_{k_{h},k_{v},\nu}(N)
&=
\frac{1}{\deg(P)}\cdot
\cN_{\cL^{\mathit{horiz}}_{\Z,k_h}\cap\cL^{\mathit{vert}}_{\Z,k_v}}(\cL, 4N)\,.
\end{align}
By Corollary~\ref{cor:density:fixed:num:comp} we have
$$
\cN_{\cL^{\mathit{horiz}}_{\Z,k}}(\cL, 4N)
= \delta(\cL^{\mathit{horiz}}_{\Z,k}) \cdot\Vol_1(\cL_1)
\cdot \frac{(4N)^d}{2d} + O(N^{d-1})
\quad\text{as }N\to+\infty\,.
$$
Combining~\eqref{eq:h} with the latter expression and
applying~\eqref{eq:Vol:cL:Vol:cQ} we obtain:
\begin{multline*}
\cS^{\mathit{labeled}}_{k,\nu}(N)
=
\frac{1}{\deg(P)}\cdot
\cN_{\cL^{\mathit{horiz}}_{\Z,k}}(\cL, 4N)
=\\=
\delta(\cL^{\mathit{horiz}}_{\Z,k})
\cdot \frac{4^d}{\deg(P)}
\cdot\Vol_1(\cL_1)
\cdot \frac{N^d}{2d} + O(N^{d-1})
=\\=
\delta(\cL^{\mathit{horiz}}_{\Z,k})
\cdot\Vol_1(\cQ_1(\nu,-1^{|\nu|+4}))
\cdot \frac{N^d}{2d} + O(N^{d-1})
\quad\text{as }N\to+\infty\,.
\end{multline*}
Letting
$$
\cyl_k\left(\cQ(\nu,-1^{|\nu|+4})\right)
:=\delta(\cL^{\mathit{horiz}}_{\Z,k})
\cdot\Vol_1(\cQ(\nu,-1^{|\nu|+4}))
$$
we obtain the desired formula~\eqref{eq:c1:N:d}.

The rationality of
$\cyl_k\left(\cQ(\nu,-1^{|\nu|+4})\right)$ follows from
rationality of $c(\cL_{\Z,k})$ proved in
Corollary~\ref{cor:density:fixed:num:comp} combined with
the relation
\begin{equation}
\label{eq:cyl:as:delta}
\cyl_k\left(\cQ(\nu,-1^{|\nu|+4})\right)
=\cfrac{\deg(P)}{4^d}\cdot c(\cL_{\Z,k})
\end{equation}
This completes the proof of
Theorem~\ref{th:c1:in:genus:0} with exception for the
explicit value~\eqref{eq:c1:answer} of
$\cyl_1\left(\cQ(\nu,-1^{|\nu|+4})\right)$ which we
compute in
Section~\ref{s:Computations:for:square:tiled:surfaces}.

The proof of Theorem~\ref{th:c11:in:genus:0} is analogous.
Namely, by property~\eqref{equiv:density} of a density of a
subset
$\cL^{\mathit{horiz}}_{\Z,k_h}\cap\cL^{\mathit{vert}}_{\Z,k_v}$
of $\cL_\Z$ we have
$$
\cN_{\cL^{\mathit{horiz}}_{\Z,k_h}\cap\cL^{\mathit{vert}}_{\Z,k_v}}(\cL, 4N)
= \delta(\cL^{\mathit{horiz}}_{\Z,k_h}\cap\cL^{\mathit{vert}}_{\Z,k_v})
\cdot\Vol_1(\cL_1)
\cdot \frac{(4N)^d}{2d} + o(N^{d})
\text{ as }N\to+\infty\,.
$$
Combining~\eqref{eq:h:v} with the latter expression and
applying~\eqref{eq:Vol:cL:Vol:cQ} as above we obtain:
\begin{multline*}
\cS^{\mathit{labeled}}_{k_{h},k_{v},\nu}(N)
=
\frac{1}{\deg(P)}\cdot
\cN_{\cL^{\mathit{horiz}}_{\Z,k_h}\cap\cL^{\mathit{vert}}_{\Z,k_v}}(\cL, 4N)\,,
=\\=
\delta(\cL^{\mathit{horiz}}_{\Z,k_h}\cap\cL^{\mathit{vert}}_{\Z,k_v})
\cdot \frac{4^d}{\deg(P)}
\cdot\Vol_1(\cL_1)
\cdot \frac{N^d}{2d} + o(N^{d})
=\\=
\delta(\cL^{\mathit{horiz}}_{\Z,k_h}\cap\cL^{\mathit{vert}}_{\Z,k_v})
\cdot\Vol_1(\cQ_1(\nu,-1^{|\nu|+4}))
\cdot \frac{N^d}{2d} + o(N^{d})
\quad\text{as }N\to+\infty\,.
\end{multline*}
Letting
\begin{equation}
\label{eq:cyl:h:v:as:delta}
\cyl_{k_{h},k_{v}}\left(\cQ(\nu,-1^{|\nu|+4})\right)
:=\delta(\cL^{\mathit{horiz}}_{\Z,k_h}\cap\cL^{\mathit{vert}}_{\Z,k_v})
\cdot\Vol_1(\cQ(\nu,-1^{|\nu|+4}))
\end{equation}
we obtain formula~\eqref{eq:c11:Q:nu}. The remaining
equation~\eqref{eq:c11:as:c1:squared:over:Vol} is obtained
by rewriting~\eqref{eq:delta:delta} in terms of the
quantities $\cyl_k\left(\cQ(\nu,-1^{|\nu|+4})\right)$ and
$\cyl_{k_{h},k_{v}}\left(\cQ(\nu,-1^{|\nu|+4})\right)$
defined by expressions~\eqref{eq:cyl:as:delta}
and~\eqref{eq:cyl:h:v:as:delta} respectively.
Theorem~\ref{th:c11:in:genus:0} is proved.
\end{proof}

We complete this section with the following Lemma used in
the proof of
Theorem~\ref{th:any:trees:connected:proportion}.
We state it for strata of meromorphic quadratic differentials
in genus zero, which allows us to use notations
introduced in Section~\ref{s:From:arc:systems:and:meanders:to:square:tiled:surfaces}.
However, both the statement and the proof of the Lemma
are applicable to any connected component of any stratum
of Abelian or quadratic differentials up to adjustment
of notations specific strata in genus zero.

Let $\Gamma$ be a separatrix diagram realizable in a
stratum $\cQ(\nu,-1^{|\nu|+4})$, let $h$ be an admissible
integer vector of heights, let $k$ be a positive integer.
We denote by $\cS^{\textit{labeled}}_{\nu}(N)$ the number
of all square-tiled surfaces in $\cQ(\nu,-1^{|\nu|+4})$
tiled with at most $2N$ identical squares. We denote by
$\cS^{\textit{labeled}}_{\Gamma,h,\nu}(N)$ the
number of square-tiled surfaces as above with additional
restriction that their horizontal cylinder decomposition is
represented by $(\Gamma,h)$. We denote by
$\cS^{\textit{labeled}}_{k,\nu}(N)$ the number
of square-tiled surfaces in $\cQ(\nu,-1^{|\nu|+4})$ tiled
with at most $2N$ identical squares and having $k$ vertical
bands of squares. Finally, we denote by
$\cS_{\Gamma,h,k,\nu}(N)$ the number of square-tiled
surfaces as above with horizontal cylinder decomposition
represented by $(\Gamma,h)$ and having $k$ vertical bands
of squares.

\begin{Lemma}
\label{lm:fixed:Gamma:h:kvert}
For any stratum $\cQ(\nu,-1^{|\nu|+4})$, any separatrix
diagram $\Gamma$ realizable in this stratum, any admissible
integer vector $h$ of heights and any positive integer
$k$ we have:
\begin{equation}
\label{eq:fixed:Gamma:h:kvert}
\lim_{N\to+\infty}
\frac{\cS^{\textit{labeled}}_{\Gamma,h,k,\nu}(N)}
{\cS^{\textit{labeled}}_{\Gamma,h,\nu}(N)}
=
\lim_{N\to+\infty}
\frac{\cS^{\textit{labeled}}_{k,\nu}(N)}{\cS^{\textit{labeled}}_{\nu}(N)}\,.
\end{equation}
\end{Lemma}
\begin{proof}
We follow the proof of Theorems~\ref{th:c1:in:genus:0}
and~\ref{th:c11:in:genus:0}. Let $\cL$
be the affine arithmetic orbifold
associated to the stratum $\cQ(\nu,-1^{|\nu|+4})$.
We have
\begin{align*}
\cS^{\mathit{labeled}}_{\Gamma,h,\nu}(N)
&=
\frac{1}{\deg(P)}\cdot
\cN_{\cL^{\mathit{horiz}}_{\Z,\Gamma,h}}(\cL, 4N)\,,
\\
\cS^{\mathit{labeled}}_{\Gamma,h,k,\nu}(N)
&=
\frac{1}{\deg(P)}\cdot
\cN_{\cL^{\mathit{horiz}}_{\Z,\Gamma,h}
\cap\cL^{\mathit{vert}}_{\Z,k}}(\cL, 4N)\,,
\end{align*}
where the sets $\cL^{\mathit{horiz}}_{\Z,\Gamma,h}$ and
$\cL^{\mathit{vert}}_{\Z,k}$ are as in
Corrolaries~\ref{cor:fixed:horiz:vert:decompositions}
and~\ref{cor:density:fixed:num:comp} respectively.
Thus,
$$
\frac
{\cS^{\mathit{labeled}}_{\Gamma,h,k,\nu}(N)}
{\cS^{\mathit{labeled}}_{\Gamma,h,\nu}(N)}
=
\frac
{\cN_{\cL^{\mathit{horiz}}_{\Z,\Gamma,h}
\cap\cL^{\mathit{vert}}_{\Z,k}}(\cL, 4N)}
{\cN_{\cL^{\mathit{horiz}}_{\Z,\Gamma,h}}(\cL, 4N)}\,.
$$
By Remark~\ref{rm:other:combinations} we have
$$
\delta(\cL^{\mathit{horiz}}_{\Z,\Gamma,h}
\cap\cL^{\mathit{vert}}_{\Z,k})
=
\delta(\cL^{\mathit{horiz}}_{\Z,\Gamma,h})\
\cdot\ \delta(\cL^{\mathit{vert}}_{\Z,k})\,,
$$
which implies that
$$
\lim_{N\to+\infty}\frac
{\cN_{\cL^{\mathit{horiz}}_{\Z,\Gamma,h}
\cap\cL^{\mathit{vert}}_{\Z,k}}(\cL, 4N)}
{\cN_{\cL^{\mathit{horiz}}_{\Z,\Gamma,h}}(\cL, 4N)}
=
\lim_{N\to+\infty}\frac
{\cN_{\cL^{\mathit{vert}}_{\Z,k}}(\cL, 4N)}
{\cN_{\cL_{\Z}}(\cL, 4N)}\,.
$$
Applying~\eqref{eq:h} and analogous relation
$
\cS^{\mathit{labeled}}_{\nu}(N)
=
\frac{1}{\deg(P)}\cdot
\cN_{\cL^{\mathit{horiz}}_{\Z}}(\cL, 4N)
$
to the expressions in the right hand side of the latter
equality, we complete the proof of the Lemma.
\end{proof}

\subsection{Explicit densities for 1-cylinder diagrams on the sphere}
\label{ss:one:cylinder}

In this Section we consider separatrix diagrams $\Gamma$ with a
single cylinder in a stratum of quadratic differentials
$\cQ(\nu, -1^{|\nu|+4})$ on the sphere. In the Lemma below
we reproduce formula~(2.2) from Proposition~2.3
in~\cite{DGZZ-Yoccoz} adapting it to the language of the
current paper.

Consider a Jenkins--Strebel meromorphic quadratic
differential with simple poles on $\CP$. Suppose that it
has single maximal horizontal cylinder. The union of all
horizontal saddle connections of such Jenkins--Strebel
differential forms two connected components and each of
the two components is a tree. In other words, a 1-cylinder
separatrix diagram $\Gamma$ on the sphere can be encoded by
a pair of plane trees (one for each boundary component of
the maximal horizontal cylinder). There is no ambiguity in
such definition since the corresponding ribbon graph has
only two boundary components, so there is a single way to
join the boundary components by the cylinder.

Let $\cQ(\nu, -1^{|\nu|+4})$ be a stratum of meromorphic
quadratic differentials with at most simple poles on $\CP$
and let $\Gamma$ be a $1$-cylinder separatrix diagram given
by a pair of plane trees $\left(\cT_{bottom}(\iota),
\cT_{top}(\nu-\iota)\right)$ with profiles $\iota$ and
$\nu - \iota$ respectively, see
Section~\ref{ss:Meanders:and:square:tiled:surfaces:in:a:given:stratum}.
We have seen in Section~\ref{ss:Meanders:and:arc:systems}
that any such separatrix diagram $\Gamma$ is realizable in
$\cQ(\nu, -1^{|\nu|+4})$ and that any $1$-cylinder
separatrix diagram has this form.

Let $\cD^{\textit{horiz}}_{\Z,\Gamma,1}$ be the set of square-tiled surfaces
in $\cQ(\nu, -1^{|\nu|+4})$ having $\Gamma$ as the
separatrix diagram and having a single horizontal band of
squares. Recall that by convention all zeroes and poles
of such square-tiled surfaces are labeled.

\begin{Lemma}
\label{lm:contribution:quadratic}
The number
$\cN_{\cD^{\textit{horiz}}_{\Z,\Gamma,1}}\big(\cQ(\nu, -1^{|\nu|+4}),2N\big)$
of all square-tiled surfaces sharing the fixed realizable
$1$-cylinder separatrix diagram $\Gamma$ and
tiled with a single band of at most $2N$ identical squares has the
following asymptotics when $N\to+\infty$
\begin{equation}
\label{eq:general:contribution:of:D:with:N}
\cN_{\cD^{\textit{horiz}}_{\Z,\Gamma,1}}\big(\cQ(\nu, -1^{|\nu|+4}),2N\big)
=\cyl_1(\Gamma)\cdot \frac{N^d}{2d} + o(N^d)\,,
\end{equation}
where $d = \dim_\C(\cQ(\nu,-1^{|\nu|+4}))$ is given by equation~\eqref{eq:dim} and
\begin{equation}
\label{eq:c1:D:iota}
\cyl_1(\Gamma) =
\cfrac{4}{|\Aut(\Gamma)|}\cdot
\frac{(|\nu|+4)!\cdot\mult_0!\cdot\mult_1!\cdot \mult_2! \cdots}
{\big(|\iota|+\ell(\iota)\big)!
\cdot
\big(|\nu-\iota|+\ell(\nu-\iota)\big)!
}\,.
\end{equation}
\end{Lemma}
\begin{proof}
In this Section we denote by $m$ the number of edges of
$\cT_{bottom}$ and we denote by $n$ the number of edges of
$\cT_{top}$. The numbers $m$ and $n$ can be expressed as
\begin{align*}
m&=|\iota|+\ell(\iota)+1
\\
n&=|\nu-\iota|+\ell(\nu-\iota)+1
\end{align*}
and the dimension $d$ of the stratum satisfies relation $d=m+n$.

Consider any square-tiled surface having the diagram
$\Gamma$ as the diagram of horizontal saddle connections.
Cut it open along all horizontal saddle connections. By
definition of $\Gamma$ it has $m$ pairs of saddle
connections on one boundary component of the cylinder, $n$
pairs of saddle connections on the other boundary component
of the cylinder, and each saddle connection has its twin
on the same boundary component.

The proof now follows line by line the second part of the
proof of the more general Proposition~2.2
in~\cite{DGZZ-Yoccoz}. Note that the parameter $l$ used in
Proposition~2.2 to denote the number of saddle connections
which after the surgery as above appear on both sides of
the cylinder is equal to zero in genus zero. One extra
simplification comes from the fact that in the proof of
Proposition~(2.2) in~\cite{DGZZ-Yoccoz} we sum over various
possible heights of the horizontal cylinder, while in our
context the height of the cylinder equals to the height of
the square: the single horizontal cylinder of square-tiled
surfaces in $\cD^{\textit{horiz}}_{\Z,\Gamma,1}$ is tiled
with a single band of squares. As a result we do not get
the extra factor $\zeta(d)$ present in the original
expression~(2.2) in Proposition~2.2 in~\cite{DGZZ-Yoccoz};
see equation~\eqref{eq:factor:zeta:d} and the Remark below.
\end{proof}

\begin{Remark}
\label{rm:zeta:d}
In this paper we denote by $\cyl_1(\Gamma)$ the coefficient of
the leading term in the asymptotics of the number of
square-tiled surfaces sharing the fixed realizable
$1$-cylinder separatrix diagram $\Gamma$ and tiled with a single
band of at most $2N$ squares. Clearly, this number does not
depend on the size of the identical squares. We tacitly
assumed above that the squares are unit squares, but we
could equally assume that our identical squares have size
$\frac{1}{2} \times \frac{1}{2}$. The latter choice
corresponds to our normalization for the Masur--Veech
volume in Section~\ref{ss:quad:as:invariant:orbifold}.

Thus, $\cyl_1(\Gamma) =
c(\cL^{\mathit{horiz}}_{\Z,\Gamma,1})$ can be seen as the coefficient
of the leading term in the asymptotics of the number of
square-tiled surfaces tiled with at most $2N$ squares of
size $\frac{1}{2} \times \frac{1}{2}$, having a single
horizontal cylinder of the \textit{minimal possible} height
$\frac{1}{2}$, and the separatrix diagram $\Gamma$. In the companion paper~\cite{DGZZ-Yoccoz} we
used a similar notation $c_1(\Gamma)$ for the coefficient in
asymptotics where we made no restriction on the height of
the cylinder. Choosing a different fixed height $h/2$ of
the cylinder, where $h\in\N$, one decreases the asymptotic
number of square-tiled surfaces as above by the factor
$h^{-d}$ (see the proof of Proposition~2.2
in~\cite{DGZZ-Yoccoz} for details). Here
$d=\dim_{\mathbb{C}}\cQ(\nu,-1^{|\nu|+4})$ is given by
formula~\eqref{eq:dim}. Moreover, the following
summation formula holds
\begin{equation}
\label{eq:factor:zeta:d}
c_1(\Gamma) =
c \left( \bigsqcup_{h\in\N} \cL^{\mathit{horiz}}_{\Z,\Gamma,h} \right)
= \sum_{h=1}^{+\infty} c(\cL^{\mathit{horiz}}_{\Z,\Gamma,h})
= \zeta(d) \cdot \cyl_1(\Gamma)\,.
\end{equation}
This is a particular case of a general
formula that holds for any number of cylinders
in any arithmetic invariant orbifold. Compare also with
the Remark~\ref{rk:coprime}.
\end{Remark}

\subsection{Explicit count of square-tiled surfaces in genus $0$}
\label{s:Computations:for:square:tiled:surfaces}
In this section we complete the proof of
Theorem~\ref{th:c1:in:genus:0} proving the remaining
relation~\eqref{eq:c1:answer}. We also prove
Corollaries~\ref{cor:principal} and~\ref{cor:leading:0}
used in the proofs of Theorem~\ref{th:meander:counting} and
of Theorem~\ref{th:trivalent:trees:connected:proportion}
respectively.

Consider a
(generalized) partition
$\nu=[0^{\nu_0} 1^{\nu_1} 2^{\nu_2} \dots]$ of a
natural number $|\nu|$ into the sum of nonnegative integer numbers
(in this Section we allow entries $0$):
$$
|\nu|:=
\underbrace{0+\dots+0}_{\nu_0}+
\underbrace{1+\dots+1}_{\nu_1}+
\underbrace{2+\dots+2}_{\nu_2}+
\dots
$$

The common convention on Masur--Veech volumes of the strata
of meromorphic quadratic differentials with at most simple
poles suggests to label (give names) to all zeroes and
poles. Following notations of Section~\ref{ss:Meanders:and:square:tiled:surfaces:in:a:given:stratum}
denote by $\cP^{\mathit{labeled}}_\nu(N)$ the number
of square-tiled surfaces with labeled zeroes and poles in
the stratum $\cQ(\nu,-1^{|\nu|+4})$ in genus zero tiled
with at most $2N$ identical squares and having a single
horizontal and a single vertical band of squares. It is
easy to see that a square-tiled surface as above cannot
have any symmetries. Convention~\ref{conv:symmetry} on
weights with which we count square-tiled surfaces with
non-labeled zeroes and poles is designed to assure the
following relation between the two counts valid for any
$N\in\N$:
\begin{equation}
\label{eq:P:labeled:through:non}
\cP^{\mathit{labeled}}_\nu(N)=
\left(\prod_{\degofz=0}^\infty \nu_\degofz !\right)
\cdot(|\nu|+4)!\,\cdot\,
\cP_\nu(N)\,,
\end{equation}
where the product above contains, actually, only finite number of factors.

Recall that a type $\iota=[0^{\iota_0} 1^{\iota_1} 2^{\iota_2}
\dots]$ of a plane tree $\cT$ records the number $\iota_\degofz$ of
vertices of valence $\degofz+2$ for $\degofz=0,1,2,\dots$. Note that
in Section~\ref{s:Computations:for:square:tiled:surfaces} we allow to the
tree have several vertices of valence $2$. Recall also that $|\nu|$
denotes the sum of the entries of the partition $\nu=[0^{\nu_0}
1^{\nu_1} 2^{\nu_2}\dots]$; by $\ell(\nu)$ we denote the length of
$\nu$, where this time we count the entries $0$ if any:
\begin{align*}
|\nu|&:=1\cdot\nu_1+2\cdot\nu_2+3\cdot\nu_3+\dots
\\
\ell(\nu)&:=\nu_0+\nu_1+\nu_2+\nu_3+\dots
\end{align*}

Consider a separatrix diagram $\Gamma$ given by a pair of trees
$\cT_{bottom}$ $\cT_{top}$ as in Section~\ref{ss:one:cylinder}.
Defining the automorphism group $\Aut(\Gamma)$ we assume that none
of the vertices,  edges, or boundary components of the ribbon graph
$\Gamma$ is labeled; however, we assume that the orientation of the
ribbons is fixed. Thus
\begin{equation}
\label{eq:order:Gamma:D}
|\Aut(\Gamma)|=|\Aut(\cT_{bottom})|\cdot|\Aut(\cT_{top})|\cdot
\begin{cases}
2&\text{if } \cT_{bottom}\simeq\cT_{top}\\
1&\textit{otherwise}
\end{cases}
\end{equation}
Here $\simeq$ stands for an isomorphism of plane (``ribbon'')
trees.

The following counting Theorem for plane trees is well known; see,
for example, \cite[2, p.6]{Moon}. It is the last element needed for
proof of Theorem~\ref{th:c1:in:genus:0}.
\begin{NNTheorem}
For any partition $\iota=[0^{\iota_0} 1^{\iota_1} 2^{\iota_2}\dots]$
the following expression holds
$$
\sum_{\cT \text{ with profile $\iota$}} \frac{1}{|\Aut(\cT)|}=
\frac
{\big(|\iota|+\ell(\iota)\big)!}
{\big(|\iota|+2\big)!\cdot \iota_0!\cdot \iota_1!\cdot\iota_2!\cdots}\,,
$$
where we sum over all plane trees corresponding to a partition
$\iota$ and $|\Aut(\cT)|$ is the order of the automorphism group
of the tree $\cT$.
\end{NNTheorem}

Now everything is ready to complete the proof of Theorem~\ref{th:c1:in:genus:0}.

\begin{proof}[Completion of the proof of Theorem~\ref{th:c1:in:genus:0}]
It only remains to prove expression~\eqref{eq:c1:answer}.

Combining equation~\eqref{eq:c1:D:iota} with the above Theorem we
conclude that the sum of $\cyl_1(\Gamma)$ over all realizable one-cylinder
separatrix diagrams $\Gamma$ in any given stratum
$\cQ(\nu,-1^{|\nu|+4})$ in genus zero can be expressed as follows
\begin{multline*}
\cyl_1\left(\cQ(\nu,-1^{|\nu|+4})\right)=
\sum_{\Gamma} \cyl_1(\Gamma)=
\frac{1}{2}\sum_{\iota\subset\nu}
\left(
\frac{4\cdot(|\nu|+4)!\cdot\mult_0!\cdot\mult_1!\cdot \mult_2! \cdots}
{\big(|\iota|+\ell(\iota)\big)!
\cdot
\big(|\nu-\iota|+\ell(\nu-\iota)\big)!
}
\right)
\cdot\\ \cdot
\left(
\frac
{\big(|\iota|+\ell(\iota)\big)!}
{\big(|\iota|+2\big)!\cdot \iota_0!\cdot \iota_1!\cdots}
\right)
\cdot
\left(
\frac
{\big(|\nu-\iota|+\ell(\nu-\iota)\big)!}
{\big(|\nu-\iota|+2\big)!\cdot (\nu_0-\iota_0)!\cdot (\nu_1-\iota_1)!\cdots}
\right)
=\\=
2\sum_{\iota\subset\nu}
\binom{|\nu|+4}{|\iota|+2}
\binom{\nu_0}{\iota_0}
\binom{\nu_1}{\iota_1}
\binom{\nu_2}{\iota_2}
\cdots
\end{multline*}
\end{proof}

We complete this section with two Corollaries from
Theorem~\ref{th:c1:in:genus:0}.

\begin{Corollary}
\label{cor:principal}
For the partition $\nu=[1^k]$,
the number $\cP^{\mathit{labeled}}_{[1^k]}(N)$ of
square-tiled surfaces with labeled zeroes and poles in the
stratum $\cQ(1^k,-1^{k+4})$ tiled with at most $2N$
identical squares and having a single horizontal and a
single vertical band of squares, has the following
asymptotics as $N\to+\infty$:
$$
\cP^{\mathit{labeled}}_{[1^k]}(N)=
\cyl_{1,1}\left(\cQ(1^k,-1^{k+4})\right)\cdot\frac{N^{2k+2}}{4k+4} +
o\left(N^{2k+2}\right)\text{ as } N\to+\infty\,,
$$
where
$$
\cyl_{1,1}\left(\cQ(1^k,-1^{k+4})\right)=
\frac{\left(\cyl_1\left(\cQ(1^k,-1^{k+4})\right)\right)^2}
{4\left(\cfrac{\pi^2}{2}\right)^{k+1}}
$$
and
\begin{equation}
\label{eq:c1:principal:answer}
\cyl_1\left(\cQ(1^k,-1^{k+4})\right)=
2\cdot\binom{2k+4}{k+2}
\end{equation}

The number $\cP^{\mathit{labeled}}_{[0,1^k]}(N)$ of square-tiled surfaces
as above with a single marked regular vertex of the tiling
has the following asymptotics as
$N\to+\infty$:
$$
\cP^{\mathit{labeled}}_{[0,1^k]}(N)=
2\cdot \cyl_{1,1}\left(\cQ(1^k,-1^{k+4})\right)\cdot\frac{N^{2k+3}}{4k+6} +
o\left(N^{2k+3}\right)\text{ as } N\to+\infty\,,
$$
\end{Corollary}
\begin{proof}
By~\eqref{eq:volume} we have
$$
\Vol_1 \cQ_1(1^k,-1^{k+4})=2\pi^2\cdot\left(\frac{\pi^2}{2}\right)^k
=4\cdot\left(\frac{\pi^2}{2}\right)^{k+1}
$$

To prove~\eqref{eq:c1:principal:answer} we apply the following
combinatorial identity to simplify formula~\eqref{eq:c1:answer} in
the particular case when $\nu=[1^k]$:
$$
\sum_{\iota_1=0}^{\nu_1}
\binom{\nu_1}{\iota_1}
\binom{|\nu|+4}{|\iota|+2}
=
\sum_{\iota_1=0}^k \binom{k}{\iota_1} \binom{k+4}{\iota_1+2}
= \binom{2k+4}{k+2}\,,
$$
see (3.20) in~\cite{Gould}.

It remains to prove that
\begin{equation}
\label{eq:tmp}
\cyl_{1,1}\left(\cQ(1^k,0,-1^{k+4})\right)=
2\cdot \cyl_{1,1}\left(\cQ(1^k,-1^{k+4})\right)\,.
\end{equation}
By~\eqref{eq:c11:as:c1:squared:over:Vol} we have
$$
\cyl_{1,1}\left(\cQ(1^k,0,-1^{k+4})\right)=
\frac{\Big(\cyl_1\left(\cQ(1^k,0,-1^{k+4})\right)\Big)^2}
{\Vol_1 \cQ_1(1^k,0,-1^{k+4})}\,.
$$
Equation~\ref{eq:c1:answer} implies that
$$
\cyl_1\left(\cQ(1^k,0,-1^{k+4})\right)=2\cdot
\cyl_1\left(\cQ(1^k,-1^{k+4})\right)
$$
Finally, by~\eqref{eq:volume} we have
$$
\Vol_1 \cQ_1(1^k,0,-1^{k+4})=2\Vol_1 \cQ_1(1^k,-1^{k+4})\,.
$$
and~\eqref{eq:tmp} follows.
\end{proof}

We also prove the following elementary technical Corollary
of Theorem~\ref{th:c1:in:genus:0}.

\begin{Corollary}
\label{cor:leading:0}
Consider a
(generalized)
partition $\nu=[0^{\nu_0} 1^{\nu_1} 2^{\nu_2}\dots]$ and its
subpartition $\nu'=[1^{\nu_1} 2^{\nu_2} \dots]$ obtained by suppressing
all zero entries. The following formulae are valid:
\begin{align}
\label{eq:c1:with:and:without:0}
\cyl_1\big(\nu,-1^{|\nu|+4}\big)
&=
2^{\nu_0}\cdot \cyl_1\big(\nu',-1^{|\nu'|+4}\big)
\\
\label{eq:p1:with:and:without:0}
\prob_1\big(\nu,-1^{|\nu|+4}\big)
&=
\prob_1\big(\nu',-1^{|\nu'|+4}\big)
\,.
\end{align}
\end{Corollary}
\begin{proof}
Note that $|\nu'|=|\nu|$. Similarly, having any subpartition
$\iota'\subset\iota$ obtained from a partition $\iota$ by suppressing
all zero entries we have $|\iota'|=|\iota|$. Thus we can rewrite
formula~\eqref{eq:c1:answer} as
\begin{multline*}
\cyl_1\left(\cQ(\nu,-1^{|\nu|+4})\right)=
2\cdot
\sum_{\iota_0=0}^{\nu_0}
\sum_{\iota_1=0}^{\nu_1}
\sum_{\iota_2=0}^{\nu_2}
\sum_{\dots}^{\dots}
\binom{\nu_0}{\iota_0}
\binom{\nu_1}{\iota_1}
\binom{\nu_2}{\iota_2}
\cdots
\binom{|\nu|+4}{|\iota|+2}
=\\=\!\!
\left(\sum_{\iota_0=0}^{\nu_0}
\binom{\nu_0}{\iota_0}
\right)\cdot
\left(
2\!\sum_{\iota_1=0}^{\nu_1}
\sum_{\iota_2=0}^{\nu_2}
\sum_{\dots}^{\dots}
\binom{\nu_1}{\iota_1}
\!\cdots\!
\binom{|\nu'|+4}{|\iota'|+2}
\right)=
2^{\nu_0} \cyl_1\!\left(\cQ(\nu',-1^{|\nu'|+4})\right)\!,
\end{multline*}
which proves~\eqref{eq:c1:with:and:without:0}.
To prove~\eqref{eq:p1:with:and:without:0} it suffices
to note that by formula~\eqref{eq:volume}, we have
$$
\Vol_1 \cQ_1(\nu,-1^{|\nu|+4})=(f(0))^{\nu_0}
\Vol_1 \cQ_1 (\nu',-1^{|\nu'|+4})= 2^{\nu_0}\Vol_1 \cQ_1(\nu',-1^{|\nu'|+4})\,.
$$
Passing to the ratios
\begin{multline*}
\prob_1\left(\cQ(\nu,-1^{|\nu|+4})\right):=\
\frac{\cyl_1\left(\cQ(\nu,-1^{|\nu|+4})\right)}
{\Vol_1 \left(\cQ_1(\nu,-1^{|\nu|+4})\right)}
\ =\\=\
\frac{\cyl_1\left(\cQ(\nu',-1^{|\nu'|+4})\right)}
{\Vol_1 \left(\cQ_1(\nu',-1^{|\nu'|+4})\right)}
\ =:
\prob_1\left(\cQ(\nu',-1^{|\nu'|+4})\right)
\end{multline*}
we get the desired equation~\eqref{eq:c1:with:and:without:0}.
\end{proof}


\appendix
\section{Lattices in strata of meromorphic quadratic differentials
and associated square-tiled surfaces}
\label{a:Lattices}

The definition of the Masur--Veech volume of strata of
meromorphic quadratic differentials with at most simple
poles involves several normalization conventions. The
mismatch in the choice of one of the conventions is a
constant source of confusion. We describe in
Section~\ref{ss:lattices} various conventions and specify
the one used in the current paper. We discuss the
topological origin of the lattices giving rise to
different normalizations of the volume element and describe
the square-tiled surfaces associated to these lattices.

Throughout this Appendix, by \textit{a stratum
$\cQ(\nu,-1^{|\nu|+4})$ of meromorphic quadratic
differentials in genus zero} we always mean the stratum in
the moduli space of meromorphic quadratic differentials
$(\CP,q)$ such that $q$ has $\nu_0$ marked points, $\nu_i$
zeroes of order $i$ for $i=1,2,\dots$, and $|\nu|+4$ simple
poles and no other poles, where $|\nu|=\nu_1+2\nu_2+\dots$.
Since we use only strata in genus zero in the study of
meanders, we limit our considerations in this Appendix to
genus zero where certain normalization issues do not
manifest. Note that all strata in genus zero are nonempty
and connected.

\subsection{Lattices in strata of meromorphic quadratic differentials}
\label{ss:lattices}
By convention, all zeroes and simple poles of Abelian or
quadratic differentials are always numbered (labeled). Thus,
expressing Masur--Veech volumes through count of
square-tiled surfaces one has to either label conical
singularities on square-tiled surfaces (which is uncommon),
or apply necessary normalization by the product of
factorials as in~\eqref{eq:P:labeled:through:non}.

Next, there are two particularly natural ways to define the
integer lattice in period coordinates $H^1_-(\hat X,
\hat\Sigma;\C)$ (the period coordinates for the strata of
quadratic differentials were defined at the end of
Section~\ref{ss:background:flat:surf}).
\smallskip

One can either choose as the ``integer lattice'' the following set:
\begin{equation}
\label{eq:half:int:lattice}
\left\{\text{elements of }
H^1_-(\hat X, \hat\Sigma;\C)
\text{ taking values in }
\Z\oplus i\Z
\text{ on } H_1^-(\hat X, \hat\Sigma;\Z)
\right\}
\,.
\end{equation}
or, alternatively, one can choose
as the ``integer lattice'' the set defined as follows:
\begin{equation*}
H^1_-(\hat X, \hat\Sigma;\C)\cap
H^1(\hat X,\hat\Sigma;\Z\oplus i\Z)\,.
\end{equation*}

The difference between the two choices reveals itself in
the linear holonomy along saddle connections joining two
distinct zeroes. Under the first convention the linear
holonomy along such saddle connections belongs to the half
integer lattice $\frac{1}{2}\Z\oplus \frac{i}{2}\Z$ while
under the second convention it belongs to the integer
lattice $\Z\oplus i\Z$.

Actually, the set $\hat\Sigma$ admits alternative natural
definitions; the choice of one of them is a matter of
another convention. In the construction of the canonical
double cover, the preimages of the simple poles of $q$ on
$X$ become regular points of $\omega$ on the canonical
double cover $\hat X$. One can either choose to mark the
resulting regular points and consider them as part of the
data of the cover, or not. In other words, one has to make
a choice whether to include these points in $\hat\Sigma$,
or not. We reserve notation $\hat\Sigma$ for the set where
all these marked points are included, and we use notation
$\hat\Sigma'$ for the set which does not contain preimages
of simple poles. In genus zero, when simple poles are
always present, $\hat\Sigma'$ is always a proper subset of
$\hat\Sigma$. For example, depending on this choice, the
stratum $\cQ(1^2, -1^6)$ is realized as an invariant
arithmetic suborbifold either in $\cH(0^6, 2^2)$ or in
$\cH(2^2)$.

Consider now the following lattices:
\begin{align*}
\L&:=
H^1_-(\hat X, \hat\Sigma;\C)\cap
H^1(\hat X,\hat\Sigma;\Z\oplus i\Z)\,,
\\
\L'&:=
H^1_-(\hat X, \hat\Sigma';\C)\cap
H^1(\hat X,\hat\Sigma';\Z\oplus i\Z)\,,
\\
\tfrac{1}{2}\L&:=
H^1_-(\hat X, \hat\Sigma;\C)\cap
H^1(\hat X,\hat\Sigma;\tfrac{1}{2}\Z\oplus \tfrac{i}{2}\Z)
\,.
\end{align*}

\begin{Lemma}
\label{lm:isomorphism}
Consider a stratum $\cQ(\nu,-1^{|\nu|+4})$ of meromorphic
quadratic differentials in genus zero. The natural map
\begin{equation}
\label{eq:map:of:H1}
H^1_-(\hat X,\hat\Sigma;\C)\to
H^1_-(\hat X,\hat\Sigma';\C)
\end{equation}
induced by the inclusion $\hat\Sigma'\subset\hat\Sigma$,
is an isomorphism of vector spaces.
\end{Lemma}
\begin{proof}
We start by constructing a convenient set of cycles,
which, depending on interpretation, provides a
basis in both $H_1^-(\hat X,\hat\Sigma;\Z)$ and
$H_1^-(\hat X,\hat\Sigma';\Z)$. This will prove, in
particular, that~\eqref{eq:map:of:H1} is an isomorphism.

Consider the following oriented non self-intersecting path
$\rho$ on the sphere $X=\CP$ at the base of the cover $\hat
X\to X$. The path $\rho$ starts at a zero or at a marked
point of $\Sigma$, then it passes through all other zeroes
and marked points of $\Sigma$, and only then $\rho$ passes
through all simple poles but one. We consider $\rho$ as a
curvilinear broken line with vertices in $\Sigma$. For
every oriented segment $\gamma_j=[P,Q]$ of this broken line
(where $P,Q$ are two points of $\Sigma$ passed by $\rho$
consecutively) consider the two preimages $\gamma_j',
\gamma_j''$ of $\gamma$ on $\hat X$ endowed with the
orientation induced from $\gamma_j$ and consider the cycle
$[\hat\gamma_j]:=[\gamma_j']-[\gamma_j'']\in H_1^-(\hat
X,\hat\Sigma;\Z)$. It is easy to see that the resulting
collections of cycles forms a basis in $H_1^-(\hat
X,\hat\Sigma;\Z)$.

By assumption, $X$ has genus zero, and quadratic
differential $q$ has only simple poles, so $\Sigma$ always
contains at least four poles of $q$. Note that we
intentionally omitted  one pole in our construction.
Extending the broken line to the remaining pole and
completing our collection of cycles with the resulting
extra cycle $\hat\gamma_0=[\gamma'_0]-[\gamma''_0]$ in
$H_1^-(\hat X,\hat\Sigma;\Z)$ we would get a collection of
cycles satisfying a linear relation: the sum of the cycles
(taken with appropriate signs) corresponding to all
segments of the extended broken line is zero in $H_1^-(\hat
X,\hat\Sigma;\Z)$.

We explain now why the constructed collection of paths
provides a basis in $H_1^-(\hat
X,\hat\Sigma';\C)$ as well. Note that all simple poles are
branch points of the cover. Recall that $\rho$ visits
first all zeroes and marked points of $\Sigma$ and only
then passes through simple poles. Thus, all cycles
$[\hat\gamma_j]:=[\gamma_j']-[\gamma_j'']$, for
$j=1,\dots,\ell(\nu)-1$ are well-defined cycles in $H_1^-(\hat
X,\hat\Sigma';\Z)$.

Let $\gamma_{\ell(\nu)}=[P,Q]$ be the segment joining the last
zero (or a marked point) $P\in\Sigma$ to the first simple
pole $Q\in\Sigma$. Consider that path
$\hat\gamma_{\ell(\nu)}:=\gamma'_{\ell(\nu)}-\gamma''_{\ell(\nu)}$
on $\hat X$ which first follows $\gamma_{\ell(\nu)}'$ and
then, when it arrives to the preimage of the simple pole
$Q$, it follows $\gamma_{\ell(\nu)}''$ in the opposite
direction. We get a connected path
$\hat\gamma_{\ell(\nu)}$ which is closed if $P$ is a
zero of odd degree. The path $\hat\gamma_{\ell(\nu)}$ is a
segment with endpoints at two preimages $P',P''$ of $P$
when $P$ is a regular point of the cover $p:\hat X\to X$
(zero of even degree of $q$ or a marked point). Note that
in the latter case both $P'$ and $P''$ belong to
$\hat\Sigma'$, so in both cases $\hat\gamma_{\ell(\nu)}$ is
a well-defined cycle in $H_1^-(\hat X,\hat\Sigma';\Z)$.
Similarly, any segment $\gamma=[Q_j, Q_{j+1}]$ of $\rho$
joining two simple poles defines a closed cycle
$[\hat\gamma]=[\gamma'-\gamma'']$ in $H_1^-(\hat
X,\hat\Sigma';\Z)$. It is immediate to see that the
resulting cycles also form a basis, but this time already
in $H_1^-(\hat X,\hat\Sigma';\Z)$. This proves, in
particular, that the natural map~\eqref{eq:map:of:H1}
induced by the inclusion $\hat\Sigma'\subset\hat\Sigma$ is
an isomorphism of vector spaces.
\end{proof}

\begin{Lemma}
\label{lm:A1:coinside:with:1:2:L}
Consider a stratum $\cQ(\nu,-1^{|\nu|+4})$ of meromorphic
quadratic differentials in genus zero. The lattice
$\tfrac{1}{2}\L$ coincides with the lattice defined
by~\eqref{eq:half:int:lattice}:
\begin{equation*}
\tfrac{1}{2}\L=
\left\{\text{elements of }
H^1_-(\hat X, \hat\Sigma;\C)
\text{ taking values in }
\Z\oplus i\Z
\text{ on } H_1^-(\hat X, \hat\Sigma;\Z)
\right\}
\,.
\end{equation*}
\end{Lemma}
\begin{proof}
Consider the basis of cocycles $\alpha_1,\dots,\alpha_d$ in
$H^1_-(\hat X, \hat\Sigma;\C)$ dual to the basis of cycles
$\hat\gamma_i$ in $H_1^-(\hat X, \hat\Sigma;\Z)$
constructed in the proof of Lemma~\ref{lm:isomorphism}. By
definition of a dual basis, we have
$\alpha_i(\hat\gamma_j)=\delta_{i,j}$. This implies that
the collection $\alpha_1,\dots,\alpha_d,
i\alpha_1,\dots,i\alpha_d$ serves as a basis of the
lattice~\eqref{eq:half:int:lattice}. Let us show now that
$\alpha_1,\dots,\alpha_d, i\alpha_1,\dots,i\alpha_d$ is
also a basis of the lattice $\tfrac{1}{2}\L$.

Denote by $\tau$ the canonical involution of $\hat X$
associated to the ramified double cover $p:\hat X\to X$.
For any $c\in H_1(\hat X, \hat\Sigma;\C)$ and for
any $\theta\in H^1_-(\hat X, \hat\Sigma;\C)$ we have
\begin{equation}
\label{eq:involution}
\theta(\tau_\ast(c))=
(\tau^\ast\theta)(c)=
-\theta(c)\,,
\end{equation}
since by definition, the subspace $H^1_-(\hat X,
\hat\Sigma;\C)$ is anti-invariant with respect to the
involution $\tau^\ast$. By construction,
$\hat\gamma_j=[\gamma'_j]-[\gamma''_j]=
[\gamma'_j]-\tau_\ast[\gamma'_j]$. Thus,
by~\eqref{eq:involution}
$$
\alpha_i([\hat\gamma'_j])=-\alpha_i([\hat\gamma''_j])=\tfrac{1}{2}\delta_{i,j}
$$
for any $i,j\in\{1,\dots,d\}$.
Since $\sum_{j=0}^d [\hat\gamma_j]=0$, we conclude that
$$
\alpha_i([\hat\gamma'_0])=-\alpha_i([\hat\gamma''_0])\in\tfrac{1}{2}\Z\ \text{for }i=1,\dots,d\,.
$$
The relative cycles $[\hat\gamma'_0],[\hat\gamma''_0],
[\hat\gamma'_1],[\hat\gamma''_1],\dots,
[\hat\gamma'_d],[\hat\gamma''_d]$ generate $H_1(\hat X,
\hat\Sigma;\Z)$. We conclude that all basic cocycles
$\alpha_j$ and $i\alpha_j$, where $j=1,\dots,d$, of the
lattice~\eqref{eq:half:int:lattice} in $H^1_-(\hat X,
\hat\Sigma;\C)$ take values in
$\frac{1}{2}\Z\oplus\frac{i}{2}\Z$ on $H_1(\hat X,
\hat\Sigma;\Z)$, which proves the inclusion
$$
\tfrac{1}{2}\L
\supseteq
\left\{\text{elements of }
H^1_-(\hat X, \hat\Sigma;\C)
\text{ taking values in }
\Z\oplus i\Z
\text{ on } H_1^-(\hat X, \hat\Sigma;\Z)
\right\}\,.
$$
Similar consideration proves the inclusion in the other
direction.
\end{proof}

The Lemma below shows that the
lattice~\eqref{eq:half:int:lattice} is independent of the
choice of $\hat\Sigma$ or $\hat\Sigma'$.
\begin{Lemma}
\label{lm:favorite:lattice}
Consider a stratum $\cQ(\nu,-1^{|\nu|+4})$ of meromorphic
quadratic differentials in genus zero.
The convention on the choice of $\hat\Sigma$ or
$\hat\Sigma'$ does not affect the
lattice~\eqref{eq:half:int:lattice}: the discrete subsets
of $H^1_-(\hat X, \hat\Sigma;\C)$ and of
$H^1_-(\hat X, \hat\Sigma';\C)$ defined
by~\eqref{eq:half:int:lattice} are in bijective
correspondence under the natural map $H^1_-(\hat X,
\hat\Sigma;\C)\to H^1_-(\hat X, \hat\Sigma';\C)$.
\end{Lemma}
\begin{proof}
By Lemma~\ref{lm:isomorphism}, the natural linear
map~\eqref{eq:map:of:H1} induced by the inclusion
$\hat\Sigma'\subset\hat\Sigma$ is an isomorphism of vector
spaces. Consider a basis of cocycles in each of these
spaces dual to the basis of integer cycles $\hat\gamma_j$,
$j=1,\dots,d$, in $H_1^-(\hat X,\hat\Sigma;\Z)$ and in
$H_1^-(\hat X,\hat\Sigma';\Z)$ respectively, constructed in
the proof of Lemma~\ref{lm:isomorphism}. The lattices
induced by these bases are exactly the discrete subsets of
$H^1_-(\hat X, \hat\Sigma;\C)$ and of $H^1_-(\hat X,
\hat\Sigma';\C)$ defined by~\eqref{eq:half:int:lattice}.
Moreover, it follows from the construction of the two bases
that the isomorphism~\eqref{eq:map:of:H1} sends one to the
other.
\end{proof}

Now we are ready to prove
inclusions~\eqref{eq:three:lattices} of lattices defined
above and compute the resulting indices of sublattices.

\begin{Lemma}
\label{lm:indices:of:sublattices}
Consider a stratum $\cQ(\nu,-1^{|\nu|+4})$ of meromorphic
quadratic differentials in genus zero.
The isomorphism~\eqref{eq:map:of:H1} induces the following chain of inclusions
of lattices:
\begin{equation}
\label{eq:three:lattices}
\L\subset
\L'\subset
\tfrac{1}{2}\L
\,.
\end{equation}
The indices of these sublattices
satisfy the following relations:
\begin{align}
\label{eq:index:marked:nonmarked}
|\tfrac{1}{2}\L:\L'|&=4^{\ell(\nu)-1}
\\
\label{eq:index:nonmarked:marked}
|\L':\L|&=4^{|\nu|+3}
\,.
\end{align}
\end{Lemma}

\begin{Remark}
\label{rm:4:power:d} Note that the lattice $\frac{1}{2}\L$
can be obtained from the lattice $\L$ by subdividing the
mesh by the factor of $2$, or, equivalently, by applying
homothety with coefficient $\frac{1}{2}$ to the lattice
$\L$. Thus, the index $|\frac{1}{2}\L:\L|$ equals
$2^{2d}=4^d$, where
$d=\dim_{\C}\cQ(\nu,-1^{|\nu|+4})=\ell(\nu)+|\nu|+2$, so
that
\begin{equation}
\label{eq:index:double}
|\tfrac{1}{2}\L:\L|=4^{\ell(\nu)+|\nu|+2}\,.
\end{equation}
The Lemma above computes the two remaining indices.
\end{Remark}

\begin{proof}[Proof of Lemma~\ref{lm:indices:of:sublattices}]
The fact that the map~\eqref{eq:map:of:H1} is a linear
isomorphism of vector spaces allows us to consider the
lattices $\L$ and $\tfrac{1}{2}\L$ as sublattices of
$H^1_-(\hat X,\hat\Sigma';\C)$ and allows to consider the
lattice $\L'$ as a sublattice of $H^1_-(\hat
X,\hat\Sigma;\C)$.

Passing from the bases of $H_1^-(\hat
X,\hat\Sigma;\Z)$ and $H_1^-(\hat X,\hat\Sigma';\Z)$ to
the dual bases in $H^1_-(\hat X,\hat\Sigma;\Z)$ and
in $H^1_-(\hat X,\hat\Sigma';\Z)$ respectively, we use the
resulting cohomology classes to construct bases of the corresponding
lattices. In this way we prove the
inclusions~\eqref{eq:three:lattices}. It remains to compute
the indices of these sublattices in the ambient lattices.

Let $\gamma_1,\dots,\gamma_{d}$ be the consecutive segments
of the path $\rho$ constructed in the proof of
Lemma~\ref{lm:isomorphism}. Here
$$
d=\dim_{\C} H^1_-(\hat X,\hat\Sigma;\C)
=\dim_{\C} H^1_-(\hat X,\hat\Sigma';\C)
=\ell(\nu)+|\nu|+2\,.
$$
By construction, the first $\ell(\nu)-1$ segments of $\rho$
have the endpoints at zeroes and at marked points of
$\Sigma$, while the remaining $|\nu|+3$ segments have
the endpoints at $|\nu|+3$ (i.e. all but one) simple poles.

Passing from the lattice $\frac{1}{2}\L$ to its sublattice
$\L'$ we impose the extra condition that the value of the
cocycles in the sublattice on each individual relative
cycle $[\gamma_j'], [\gamma_j'']$ belongs to $\Z\oplus i\Z$
for $j=1,\dots,\ell(\nu)-1$ (i.e. for those cycles which
involve points from $\hat\Sigma'$). The cocycles in the
sublattice $\frac{1}{2}\L$ take values in
$\frac{1}{2}\Z\oplus \frac{i}{2}\Z$ on these cocycles. Note
that since these cocycles belong to the subspace $H^1_-(\hat
X,\hat\Sigma;\C)$ anti-invariant with respect to the
hyperelliptic involution, their values on all pairs of
symmetric cycles $[\gamma_j']$ and $[\gamma_j'']$ are
coherent: both real and imaginary parts are simultaneously
integer or half-integer. This gives the index
$|\tfrac{1}{2}\L:\L'|=2^{2(\ell(\nu)-1)}$ as claimed
in~\eqref{eq:index:marked:nonmarked}. Extra factor $2$ in
the exponent of $2$ comes from the fact that we have to
take into account both real and imaginary parts.

To prove~\eqref{eq:index:nonmarked:marked} one can either
combine~\eqref{eq:index:double}
and~\eqref{eq:index:marked:nonmarked} or notice directly
that passing from the lattice $\L'$ to its sublattice $\L$
we impose the extra condition that the value of the
cocycles in the sublattice $\L$ on each individual relative
cycle $[\gamma_j'], [\gamma_j'']$ belongs to $\Z\oplus i\Z$
for $j=\ell(\nu),\dots,\ell(\nu)+|\nu|+2$, while the
cocycles in the ambient lattice $\L'$ evaluated
on  on these cycles take values in $\frac{1}{2}\Z\oplus
\frac{i}{2}\Z$.
\end{proof}

Consider the Masur--Veech volume elements on a stratum
$\cQ(\nu,-1^{|\nu|+4})$ of meromorphic quadratic
differentials in genus zero corresponding to the above
lattices. We have proved in Lemma~\ref{lm:isomorphism} that
we can use any of the two isomorphic vector spaces
\begin{equation}
\label{eq:isomorphic:cohomology}
H^1_-(\hat X,\hat\Sigma;\C)
\simeq
H^1_-(\hat X,\hat\Sigma';\C)
\end{equation}
as period coordinates in $\cQ(\nu,-1^{|\nu|+4})$. Recall
that there is a natural one-parameter family of volume
elements in any finite-dimensional vector space $V$; any
two volume elements in this family differ by a constant
factor. Any lattice of maximal rank in $V$ determines the
distinguished normalization of the volume element by
condition that the volume of the fundamental domain of the
lattice is equal to $1$. Considering one of the vector
spaces in~\eqref{eq:isomorphic:cohomology} and one of the
three lattices~\eqref{eq:three:lattices}, we get three
different normalizations of the Masur--Veech volume element
on $\cQ(\nu,-1^{|\nu|+4})$. Denote by
$\Vol^{\frac{1}{2}\L}_1\cQ_1(\nu,-1^{|\nu|+4})$,
$\Vol^{\L'}_1\cQ_1(\nu,-1^{|\nu|+4})$,
$\Vol^{\L}_1\cQ_1(\nu,-1^{|\nu|+4})$ the volumes of the
stratum $\cQ(\nu,-1^{|\nu|+4})$ with respect to these
volume elements.

\begin{Corollary}
The Masur--Veech volumes of a stratum
$\cQ(\nu,-1^{|\nu|+4})$ of meromorphic differentials in
genus zero induced by the
lattices~\eqref{eq:three:lattices} satisfy the following
relations:
\begin{align*}
\Vol^{\frac{1}{2}\L}_1\cQ_1(\nu,-1^{|\nu|+4})
:\Vol^{\L'}_1\cQ_1(\nu,-1^{|\nu|+4})
&=4^{\ell(\nu)-1}\,,
\\
\Vol^{\L'}_1\cQ_1(\nu,-1^{|\nu|+4})
:\Vol^{\L}_1\cQ_1(\nu,-1^{|\nu|+4})
&=4^{|\nu|+3}\,.
\end{align*}
\end{Corollary}
\begin{proof}
The Masur--Veech volume element is defined as linear volume
element in period coordinates normalized by an appropriate
lattice. Hence, passing from a lattice to a sublattice we
change the normalization of the Masur--Veech measure by the
constant factor equal to the index of the sublattice.
\end{proof}

\subsection{Degrees of correspondences induced by
canonical double covers}
\label{subsec:double:cover:correspondences}
The comparison of lattices in the previous Section was
performed in period coordinates, that is locally. This
local computation allowed us to keep track of the change of
the Masur--Veech volume $\Vol_1\cQ_1(\nu,-1^{|\nu|+4})$
when passing from one lattice normalization to the other.
Up to now we always stayed on $\cQ(\nu,-1^{|\nu|+4})$.

Recall that the canonical ramified double cover
construction associates to every stratum of meromorphic
quadratic differentials with at most simple poles an
invariant arithmetic suborbifold $\cL$ in the corresponding
stratum of Abelian differentials, where the preimages of
all simple poles are marked, see
Section~\ref{ss:quad:as:invariant:orbifold}. The invariant
arithmetic suborbifold $\cL$ is endowed with the natural
cover $P: \cL \to \cQ(\nu,-1^{|\nu|+4})$.

Alternatively, we can apply the double cover construction
without marking the preimages of simple poles. In this way
we obtain an alternative invariant arithmetic orbifold
$\cL'$ in a different stratum of Abelian differentials. For
example, under the first construction we associate to the
stratum $\cQ(1^2,-1^6)$ an invariant arithmetic suborbifold
$\cL$ in $\cH(2^2,0^6)$ while under the second construction
we associate to the same stratum $\cQ(1^2,-1^6)$ an
invariant arithmetic suborbifold $\cL'$ in $\cH(2^2)$. The
resulting invariant arithmetic orbifolds $\cL$ and $\cL'$
are related by the natural forgetful map $F: \cL \to \cL'$
forgetting the preimages of the simple poles. The choice
between $\cL$ and $\cL'$ corresponds to the choice between
the subsets $\hat\Sigma$ or $\hat\Sigma'$ in the local
computations of Section~\ref{ss:lattices}.

Recall that we have a canonical definition~\eqref{eq:Vol:L}
of the Masur--Veech volume of any arithmetic invariant
orbifold. In this Section we compare the Masur--Veech
volumes $\Vol_1\cQ_1(\nu,-1^{|\nu|+4})$, $\Vol_1\cL_1$
and $\Vol_1 \cL'_1$.

There is a natural correspondence between
$\cQ(\nu,-1^{|\nu|+4})$ and $\cL$. Each zero
of odd order of $q$ and each simple pole of $q$ has a
single preimage on $\hat X$. However, each zero of even
order of $q$ and each marked point has two distinct
preimages and both preimages are zeroes (respectively
marked points) of $\omega$. We denote by
$$
\ell(\nu_{\mathit{even}}):=\nu_0+\nu_2+\nu_4+\dots
$$
the total number of marked points and of
zeroes of even orders.
Note also that the triples $(\hat X,\omega, \hat\Sigma)$
and $(\hat X,-\omega, \hat\Sigma)$ represent the same point
of the stratum of Abelian differentials with
\textit{unlabeled} zeroes. These considerations show that
we have $2^{\ell(\nu_{\mathit{even}})}$ ways of labeling
the preimages of zeros of even degrees.
A priori, the resulting
set $\widehat{\cQ}(\nu,-1^{|\nu|+4})$ of labeled
Abelian differentials may be
not connected. By definition, $\cL$ denotes
any of the isomorphic connected
components of $\widehat{\cQ}(\nu,-1^{|\nu|+4})$.

\begin{Lemma}
\label{lem:cL:over:cQ}
Consider a stratum $\cQ(\nu,-1^{|\nu|+4})$ of meromorphic
quadratic differentials in genus zero.
The subset $\widehat{\cQ}(\nu,-1^{|\nu|+4})$
of Abelian differentials obtained
by applying the canonical double cover construction
to all $(X,q)$ in $\cQ(\nu,-1^{|\nu|+4})$ with marked
preimages of simple poles
is connected.
\end{Lemma}
\begin{proof}
The projectivization
$\operatorname{P}\!\cQ(\nu,-1^{|\nu|+4})=\cQ(\nu,-1^{|\nu|+4})/\C^{\ast}$
of any stratum $\cQ(\nu,-1^{|\nu|+4})$ in genus zero is
isomorphic to the space of configurations of $\ell(\nu) +
|\nu| + 4$ distinct labeled points on $\CP$ (the locations
of the zeros, of the marked points and of the simple poles
of the quadratic differential $q$). In particular, given
any zero of $q$ of even order (where we consider the marked
points as zeroes of $q$ of order zero) one can always
construct a closed trajectory in $\cQ(\nu, -1^{|\nu|+4})$
by moving this zero close to a simple pole, then around
this pole and then backtracking the zero to the initial
position following backwards the path which we used to move
it close to the simple pole. Lifting the resulting loop in
$\cQ(\nu, -1^{|\nu|+4})$ to a path in $\widehat{\cQ}(\nu,
-1^{|\nu|+4})$ we exchange the labels of the two preimages
of this zero of even order. This construction implies
connectedness of $\widehat{\cQ}(\nu, -1^{|\nu|+4})$.
\end{proof}

Considering the correspondence between
$\cQ(\nu,-1^{|\nu|+4})$ and $\cL'$ we observe the following
additional phenomenon. The set $\hat\Sigma$ can be
reconstructed from $\hat\Sigma'$ as an unlabeled set, by adding all
ramification points of $\hat X$ which are not yet in
$\hat\Sigma'$. However, since the preimages of the simple
poles are not marked in $\hat\Sigma'$ and thus not labeled,
the information on the labels of simple poles in $(X,q,
\Sigma)$ is lost when passing to the canonical double cover
$(\hat X,\omega,\hat\Sigma')$. Thus, the points of
$\cQ(\nu,-1^{|\nu|+4})$ and the points of $\cL'$ are in the
natural $(|\nu|+4)!:2^{\ell(\nu_{\mathit{even}})}$
correspondence, where $|\nu|+4$ is the number of simple
poles.

Let $p:\hat X\to X$ be the canonical ramified double cover
such that $p^\ast q=\omega^2$. Note that the flat area
defined by the Abelian differential $\omega$ on $\hat X$ is
twice the area defined by the initial quadratic
differential $q$ on $X$. Thus, under the correspondence
of the stratum
$\cQ(\nu,-1^{|\nu|+4})$ of quadratic differentials and the
suborbifold in the associated stratum of Abelian
differentials, the hypersurface $\cL_1$ corresponds to the
subset of quadratic differentials $(X,q)$ of flat area
$\frac{1}{2}$. Following the tradition, we denote this locus by
$\cQ_1(\nu,-1^{|\nu|+4})$, i.e.
\begin{equation}
\label{eq:def:Q:1}
\cQ_1(\nu,-1^{|\nu|+4}):=
\left\{
(X,q)\in\cQ(\nu,-1^{|\nu|+4})
\ \big|\
\Area(X,q)=\frac{1}{2}
\right\}\,.
\end{equation}
Given a Masur--Veech volume element $d\!\Vol$
on $\cQ(\nu,-1^{|\nu|+4})$ we define
\begin{equation}
\label{eq:def:Vol:Q:1}
\Vol_1\cQ_1(\nu,-1^{|\nu|+4}):=
2d\cdot \Vol C_1\cQ_1(\nu,-1^{|\nu|+4})\,,
\end{equation}
where the cone $C_1\cQ_1(\nu,-1^{|\nu|+4})$ is defined as
$$
C_1\cQ_1(\nu,-1^{|\nu|+4}):=
\left\{
(X,q)\in\cQ(\nu,-1^{|\nu|+4})
\ \big|\
\Area(X,q)\le\frac{1}{2}
\right\}\,.
$$
The numerical values in~\eqref{eq:volume} were obtained
using this definition and the Masur--Veech volume element
on $\cQ(\nu,-1^{|\nu|+4})$ normalized by the
lattice~\eqref{eq:half:int:lattice}, which by
Lemma~\ref{lm:A1:coinside:with:1:2:L} this lattice
coincides with the lattice $\tfrac{1}{2}\L$.

Consider the locus $\cL'$ of Abelian differentials
consisting of all canonical double covers
of all $(X,q)$ in $\cQ(\nu,-1^{|\nu|+4})$,
where the preimages of simple poles are not marked. Let $P: \cL \to
\cQ(\nu, -1^{|\nu|+4})$ be the natural cover and let $F:
\cL \to \cL'$ be the forgetful map (forgetting the marked
points that are preimages of poles). We complete
this section with the following Lemma summarizing
the above considerations.

\begin{Lemma}
\label{lm:Vol:star} Consider any normalization
$d\!\Vol^\ast$ of the Masur--Veech volume element on a
stratum $\cQ(\nu, -1^{|\nu|+4})$ of meromorphic quadratic
differentials in genus zero and
consider the volume element on $\cL$ induced by
$d\!\Vol^\ast$ via the cover $P:\cL\to \cQ(\nu,
-1^{|\nu|+4})$. The resulting volume element on $\cL$ can
be induced from a volume element $d\!\Vol^\ast$ on $\cL'$
by means of the forgetful map $F:\cL\to\cL'$. The volumes
$\Vol^\ast_1\cL_1$, $\Vol^\ast_1\cL'_1$ and
$\Vol^\ast_1\cQ_1(\nu,-1^{|\nu|+4})$ satisfy the following
relations:
\begin{align}
\label{eq:Vol:L:Sigma:0:Vol:Q}
\Vol^\ast_1\cL_1&:=2^{\ell(\nu_{\mathit{even}})}
\cdot\Vol^\ast_1\cQ_1(\nu,-1^{|\nu|+4})
\\
\label{eq:Vol:L:Sigma:1:Vol:Q}
\Vol^\ast_1\cL'_1&:=
\frac{2^{\ell(\nu_{\mathit{even}})}}{(|\nu|+4)!}
\cdot\Vol^\ast_1\cQ_1(\nu,-1^{|\nu|+4})\,.
\end{align}
\end{Lemma}
\begin{proof}
The proof is essentially reduced to the observation that
$\deg(P) = 2^{\ell(\nu_{even})}$ and  that $\deg(F) =
(|\nu|+4)!$.
\end{proof}

\subsection{Lattice  points as  square-tiled  surfaces}
\label{subsec:lattice:square:tiled:pillowcase}

In this Section we discuss the square-tiling in the
flat metric associated to lattice points $(X,q,\Sigma)$,
$(\hat X,\omega,\hat\Sigma)$, $(\hat
X,\omega,\hat\Sigma')$ for the lattices $\tfrac{1}{2}\L$, $\L$
and $\L'$ respectively.

We have seen that lattice points in period coordinates of
an invariant arithmetic suborbifold represent square-tiled
surfaces. Speaking of a square-tiling we always assume that
all the squares of the tiling are identical and
\textit{polarized}, i.e. we know which pair of opposite
sides of each square is horizontal; the remaining pair of
sides is vertical. Gluing the squares we identify sides to
sides respecting the polarization. Constructing a
square-tiled surface representing an Abelian differential
we impose additional translation structure.

We always consider only those square tilings of
$(X,q,\Sigma)$ (respectively of $(\hat X,q,\hat\Sigma)$,
$(\hat X,q,\hat\Sigma')$) for which all the points of the
discrete subset $\Sigma$ (respectively of $\hat\Sigma$,
$\hat\Sigma'$) are located at vertices of the squares.

\begin{Lemma}
\label{lm:lattices:as:square:tilings}
Consider a stratum $\cQ(\nu,-1^{|\nu|+4})$ of meromorphic
quadratic differentials in genus zero.

\begin{enumerate}
\item
$\tfrac{1}{2}\L$-lattice points in
$\cQ(\nu,-1^{|\nu|+4})$ are exactly those triples
$(X,q,\Sigma)$ for which the induced (polarized) flat
metric admits square-tiling by
$\tfrac{1}{2}\times\tfrac{1}{2}$ squares. The induced
normalization $d\!\Vol^{\frac{1}{2}\L}$ of the
Masur--Veech volume element on $\cQ(\nu,-1^{|\nu|+4})$
combined with definitions~\eqref{eq:def:Q:1}
and~\eqref{eq:def:Vol:Q:1} provides the same value of the
Masur--Volume $\Vol_1\cQ_1(\nu,-1^{|\nu|+4})$ as~\eqref{eq:easy:normalization:Vol:cQxi}.

\item
$\L$-lattice points in $\cQ(\nu,-1^{|\nu|+4})$ are
exactly those triples $(X,q,\Sigma)$ for which the
induced (polarized) flat metric admits square-tiling by
unit squares. The induced Masur--Veech volume element
$d\!\Vol^{\L}$ on $\cQ(\nu,-1^{|\nu|+4})$ satisfies:
$$
d\!\Vol^{\L}=4^d\cdot
d\!\Vol^{\frac{1}{2}\L}
$$

\item
$\L$-lattice points in the arithmetic invariant
suborbifold $\cL$ are exactly those triples $(\hat
X,\omega,\hat\Sigma)$ for which the induced (polarized)
flat metric admits square-tiling by unit squares. The
associated Masur--Veech volume element $d\!\Vol^{\L}$ on
$\cL$ coincides with the volume element induced from the
volume element $d\!\Vol^{\L}$ on $\cQ(\nu,-1^{|\nu|+4})$
by means of the cover $P:\cL\to \cQ(\nu,-1^{|\nu|+4})$.
It provides the same value of the Masur--Volume
$\Vol_1\cL_1$ as~\eqref{eq:Vol:L}.

\item
$\L'$-lattice points in the arithmetic invariant
suborbifold $\cL'$ are exactly those triples $(\hat
X,\omega,\hat\Sigma')$ for which the induced (polarized)
flat metric admits square-tiling by unit squares. The
induced normalization $d\!\Vol^{\L'}$ of the Masur--Veech
volume element on $\cL'$ provides the same value of the
Masur--Volume $\Vol_1\cL'_1$ as~\eqref{eq:Vol:L}.
\end{enumerate}
\end{Lemma}
\begin{proof}
By Lemma~\ref{lm:A1:coinside:with:1:2:L} the lattices
$\tfrac{1}{2}\L$ and~\eqref{eq:half:int:lattice} coincide.
Upon this observation, the assertion (1) becomes a direct
corollary of Lemma~B.1. from~\cite{AEZ:genus:0} (reproduced
as Lemma~\ref{lm:pillowcase} below).

Statement (2) is proved in Remark~\ref{rm:4:power:d}.

The assertions of the remaining statements (3) and (4)
concerning the square-tilings are proved by the following
standard argument, see, for
example~\cite{Zorich:square:tiled}. We recall the proof for
the lattice $\L$; the proof for $\L'$ is completely
analogous.

By definition of the lattice
$\L$, the relative cohomology class $[\omega]$
of any $\L$-lattice point $(\hat
X,\omega,\hat\Sigma)$ in $\cL$ belongs to
$H^1(\hat X,\hat\Sigma;\Z\oplus i\Z)$. Fix a point
$P\in\hat\Sigma$. Consider the map $p_{\omega}  : X \to
\C{}/(\Z\oplus i\,\Z)$ defined as
$$
p_{\omega} : Q\mapsto \Big(\int_{P}^Q \omega\Big)\
\text{mod } \Z\oplus i\,\Z\,.
$$
The condition $[\omega]\in H^1(\hat X,\hat\Sigma;\Z\oplus
i\Z)$ implies that the map does not depend on the path
joining the distinguished point $P$ to a point $Q$ of $X$,
so $p_\omega$ is well-defined. It is easy  to see that
$p_{\omega}$ is a  ramified cover and all ramification
points of $p_\omega$ belong to $\hat\Sigma$. Consider the
flat torus $\C{}/(\Z\oplus i\,\Z)$ as a unit square with
the identified  opposite sides. The cover $p_{\omega}$
endows $X$ with a  tiling by unit  squares. By construction
all points from $\hat\Sigma$ are located in the corners of
the resulting square-tiling and, thus, the assertions of
statements (3) and (4) concerning the square-tilings are
proved.

The coincidence of the volume element $d\!\Vol^{\L}$ on
$\cL$ with the volume element obtained by
pulling back the volume element $d\!\Vol^{\L}$
on $\cQ(\nu,-1^{|\nu|+4})$ by means of the cover $P:\cL\to
\cQ(\nu,-1^{|\nu|+4})$ follows from the fact that both
$\cL$ and $\cQ(\nu,-1^{|\nu|+4})$ share the same local
period coordinates $H^1_-(\hat X,\hat\Sigma,\C)$ in which
the two volume elements are determined by the same lattice
$\L$. The cover $P$ acts in these coordinates as the
identity map.

Finally, the assertions about the values of the resulting
Masur--Veech volumes $\Vol_1\cL_1$ and $\Vol_1\cL'_1$ in
(3) and (4) respectively now follow from equivalence of
definition~\eqref{eq:normalization:masur:veech} (where we
respectively take $\cL_1$ or $\cL'_1$ as $V_1$) and
definition~\eqref{eq:Vol:L}; this equivalence was proved in
Section~\ref{ss:Masur:Veech:volume:element}.
\end{proof}

We are now ready to prove
Proposition~\ref{prop:volumes:quad:and:double:cover}.
Though, formally speaking, we provide a proof of
Proposition~\ref{prop:volumes:quad:and:double:cover} only
for strata of meromorphic quadratic differentials in genus
zero, the general case is completely analogous with
exception for connectedness of
$\widehat{\cQ}(\xi)^{\mathit{comp}}$ which is claimed in
Proposition~\ref{prop:volumes:quad:and:double:cover} only
for strata in genus zero.

\begin{proof}[Proof of  Proposition~\ref{prop:volumes:quad:and:double:cover}]
Let $(\hat X,\omega,\hat\Sigma)$ be a square-tiled surface
in $\cL$. By construction, the cover $p_\omega:\hat X\to\T$ constructed in
the proof of Lemma~\ref{lm:lattices:as:square:tilings}
intertwines the involutions $\tau:\hat X\to\hat X$ and
$\iota:\T\to\T$, where $\iota$ is the elliptic involution
on $\T$:
$$
\begin{CD}
\hat X   @>\tau>>  \hat X\\
@V p_\omega VV @V p_\omega VV \\
\T @>\iota>> \T
\end{CD}
$$
This implies that the involution $\tau$ maps the squares
of the tiling onto squares of the tiling.

The involution $\tau$ maps $\omega$ to $-\omega$, i.e.
$\tau^\ast\omega=-\omega$. Thus, if one of the squares
maps to itself, this is done by
the central symmetry, which fixes the center of the square.
Hence, the center of this square would be the preimage of a
simple pole of $q$ on $X$ which contradicts the assumption
that the points of $\hat\Sigma$ are located only at the
corners of the squares. This argument proves assertions (1)
and (2) of
Proposition~\ref{prop:volumes:quad:and:double:cover}.

Assertion (3) of
Proposition~\ref{prop:volumes:quad:and:double:cover}
follows directly from assertions (2) and (3) of
Lemma~\ref{lm:lattices:as:square:tilings}.

The remaining assertion (4) of
Proposition~\ref{prop:volumes:quad:and:double:cover} is
specific for genus zero. It was already proved in
Lemma~\ref{lem:cL:over:cQ}.
\end{proof}

Throughout Section~\ref{s:Equidistribution} we have predominantly
considered the invariant arithmetic orbifold $\cL$.
Corollary~\ref{cor:hyperelliptic:components} below
describes the important particular case when it is
natural to consider the invariant arithmetic orbifold
$\cL'$.

\begin{Corollary}
\label{cor:hyperelliptic:components}
Consider a stratum $\cQ(k,-1^{k+4})$ of those meromorphic
quadratic differentials in genus zero, which have a single
zero of order $k$, where $k\in\N$, and $k+4$ simple poles.
The associated invariant arithmetic orbifold $\cL'$
in the corresponding stratum of Abelian differentials
coincides with the hyperelliptic connected component
$\cH^{\mathit{hyp}}(k+1)$ when $k$ is odd and with the
hyperelliptic connected component
$\cH^{\mathit{hyp}}(\left(\frac{k}{2}\right)^2)$ when $k$ is
even.

In this particular case the lattices $\tfrac{1}{2}\L$ and
$\L'$ coincide and thus the Masur--Veech volumes
$\Vol_1\cL'_1$ and $\Vol_1\cQ_1(k,-1^{k+4})$ defined
by~\eqref{eq:Vol:L}
and~\eqref{eq:easy:normalization:Vol:cQxi} respectively are
related by~\eqref{eq:Vol:L:Sigma:1:Vol:Q}.
\end{Corollary}
\begin{proof}
The first assertion is the definition of a
hyperelliptic connected component,
see~\cite{Kontsevich:Zorich}.

The coincidence of lattices $\tfrac{1}{2}\L$ and $\L'$
follows from~\eqref{eq:index:marked:nonmarked}. Namely, in
our case the partition $\nu$ contains a single entry, and
thus $4^{\ell(\nu)-1}=1$. Hence, choosing in
Lemma~\ref{lm:Vol:star} the Masur--Veech volume element
$d\!\Vol^{\frac{1}{2}\L}$ on $\cQ(k,-1^{k+4})$ we get as
the induced volume element $d\!\Vol^{\frac{1}{2}\L}$ on
$\cL'$ the Masur--Veech volume element $d\!\Vol^{\L'}$.
Thus, the corresponding volumes are related by
formula~\eqref{eq:Vol:L:Sigma:1:Vol:Q}.

It remains to note that by
Lemma~\ref{lm:lattices:as:square:tilings} the Masur--Veech
volume $\Vol_1\cL'_1$ defined by~\eqref{eq:Vol:L}
corresponds to the normalization of the Masur--Veech volume
element on $\cL'$ by the lattice $\L'$ and that the
Masur--Veech volume $\Vol_1\cQ_1(k,-1^{k+4})$ defined
by~\eqref{eq:easy:normalization:Vol:cQxi} corresponds to
the normalization of the Masur--Veech volume element on
$\cQ(k,-1^{k+4})$ by the lattice $\tfrac{1}{2}\L$, or,
equivalently, by the lattice~\eqref{eq:half:int:lattice}.
\end{proof}

Remark~1.2 at the end of Section~1.1 in~\cite{AEZ:genus:0}
presents the following two illustrations of
Corollary~\ref{cor:hyperelliptic:components}
(see~\eqref{eq:volume} for the definition of $f$):
\begin{align*}
\Vol_1\cQ_1(1, -1^5) &=
2\pi^2\!\cdot\! f(1)\!\cdot\!\left(f(-1)\right)^5=
2\pi^2\cdot\frac{\pi^2}{2}\cdot 1^5=
5!\cdot\frac{\pi^4}{120}=5!\cdot\Vol_1\cH_1(2)\,,
\\
\Vol_1\cQ_1(2, -1^6) &=
2\pi^2\!\cdot\! f(2)\!\cdot\!\left(f(-1)\right)^6=
2\pi^2\!\cdot\!\frac{4\pi^2}{3}\!\cdot\! 1^6=
\frac{6!}{2}\!\cdot\!\frac{\pi^4}{135}
=\frac{6!}{2}\!\cdot\!\Vol_1\cH_1(1,1)\,.
\end{align*}

We complete this Section reproducing the statement of the
result from Appendix~B in~\cite{AEZ:genus:0} proving that
in genus zero the $\frac{1}{2}$-square-tiled surfaces
$(X,q,\Sigma)$ (i.e. the points of the stratum
corresponding to the lattice $\tfrac{1}{2}\L$, or,
equivalently, to the same lattice defined
as~\eqref{eq:half:int:lattice}) are represented by
pillowcase covers specified below. We warn the reader that
the statement of Lemma~\ref{lm:pillowcase} below is not
valid for genera higher than zero.

Let  $\Lambda \subset \C$ be a lattice, and let $\T^2 = \C/\Lambda$
be the associated torus. The quotient
$$
\cP : = \T^2/\pm
$$
by  the  map  $z  \rightarrow  -z$ is known as the
\textit{pillowcase orbifold}.  It  is  a  sphere with four
$(\Z/2)$-orbifold points (the corners  of  the
pillowcase). The quadratic differential $(dz)^2$ on $\T^2$
descends  to  a  quadratic differential on $\cP$. Viewed as
a quadratic  differential  on  the  Riemann sphere,
$(dz)^2$ has simple poles  at  corner  points. When the
lattice $\Lambda$ is the standard integer  lattice
$\Z\oplus i\Z$, the flat torus $\T^2$ is obtained by
isometrically  identifying  the  opposite sides of a unit
square, and the  pillowcase  $\cP$  is  obtained by
isometrically identifying two $\frac{1}{2}\times\frac{1}{2}$
squares along the perimetwer, see the right picture
in Figure~\ref{fig:square:pillow} in
section~\ref{ss:pairs:of:multicurves:as:square:tiled:surfaces}.

Consider  a  connected  ramified cover $\hat\cP$ over $\cP$ of degree
$d$  having  ramification  points  only  over  the  corners of the
pillowcase.  Clearly,  $\hat\cP$  is  tiled by $2d$ squares of the
size $\frac{1}{2}\times\frac{1}{2}$. Coloring the two squares
of the pillowcase $\cP$ one in black and the other in white, we get a
chessboard  coloring of the square tiling of the cover $\hat\cP$:
the white squares are always glued to the black ones and vice versa.

\begin{Lemma}[\cite{AEZ:genus:0}]
\label{lm:pillowcase}
Let $S=(X,q,\Sigma)$ be a point in the stratum
$\cQ(\nu,-1^{|\nu|+4})$ of meromorphic quadratic
differentials in genus zero. The following properties are
equivalent:
\begin{enumerate}
\item
$S$ represents a point
of the lattice~\eqref{eq:half:int:lattice}
in $\cQ(\nu,-1^{|\nu|+4})$;
\item
$S$ is a cover over $\cP$ ramified only over the corners of the pillowcase;
\item
$S$ is tiled by black and white $\frac{1}{2}\times\frac{1}{2}$ squares in the
chessboard order.
\end{enumerate}
\end{Lemma}

Lemma~\ref{lm:pillowcase} implies that any square-tiled
surface in any stratum of meromorphic quadratic
differentials in genus zero is always tiled with
\textit{even} number of squares and that such tiling always
admits chessboard coloring.

Note that in genera one and higher one finds square-tiled
surfaces tiled with odd number of squares and square-tiled
surfaces tiled with even number of squares which do not
admit chessboard coloring.
Note also that ``pillowcase covers'' are defined
differently by different authors, see e.g.~\cite{Eskin:Okounkov:pillowcase}.


\end{document}